\documentclass[de, en, nobiblatex, secnum,noparskip]{cls/myamspaper}

\numberwithin{equation}{section} 

\usepackage[top=40mm, bottom=40mm, left=30mm, right=30mm]{geometry}

\subjclass[2020]{Primary 37A15, 43A65; Secondary 46H25, 06E15.}

\newcommand{\olim}{\operatorname*{o-lim}}
\newcommand\logeq{\mathrel{\vcentcolon\Leftrightarrow}}

\DeclareMathAlphabet{\mathpzc}{OT1}{pzc}{m}{it}

\newcommand{\Om}{\Omega}
\newcommand{\vom}{\omega}
\newcommand{\noi}{\noindent}

\newcommand{\smsk}{\smallskip}
\newcommand{\emdf}{\bf}
\newcommand{\vanish}[1]{\relax}
\newcommand{\abs}[1]{\vert #1 \vert}           
\newcommand{\Stoneanalgebra}{Stone algebra}
\newcommand{\norm}[1]{\Vert #1 \Vert}           
\newcommand{\car}{\mathbf{1}}
\newcommand{\Ce}{\mathrm{C}}
\newcommand{\vphi}{\varphi}
\newcommand{\veps}{\varepsilon}

\DeclareMathOperator{\HS}{HS}

\newcommand{\sprod}[2]{(#1|#2)}
\newcommand{\cls}[1]{\overline{#1}}
\newcommand{\spann}{\operatorname{span}}
\newcommand{\konj}[1]{\overline{#1}}

\newcommand{\KH}{Kaplansky--Hilbert}

\newcommand{\ohne}{\setminus}
\newcommand{\BL}{\mathscr{L}}
\newcommand{\calB}{\mathcal{B}}

\newcommand{\calN}{\mathcal{N}}
\newcommand{\calT}{\mathcal{T}}
\newcommand{\tensor}{\otimes}
\newcommand{\Ell}[1]{\mathrm{L}^{#1}}
\newcommand{\Id}{\mathrm{I}}
\newcommand{\B}{\mathbb{B}}
\newcommand{\dann}{\Rightarrow}
\newcommand{\gdw}{\Leftrightarrow}
\newcommand{\Homeo}{\mathrm{Homeo}}
\newcommand{\nach}{\circ}
\newcommand{\ocl}{\operatorname{ocl}}
\newcommand{\dprod}[2]{\langle #1,#2\rangle}

\renewcommand{\ds}{\mathrm{ds}}
\newcommand{\wm}{\mathrm{wm}}
\newcommand{\rev}{\mathrm{rev}}

\newcommand{\cl}{\mathrm{cl}}
\newcommand{\Kro}{\mathrm{Kro}}

\newcommand{\alg}{\mathrm{alg}}

\newcommand{\scrE}{\mathscr{E}}
\newcommand{\bfX}{\mathbf{X}}
\newcommand{\bfY}{\mathbf{Y}}
\newcommand{\bfZ}{\mathbf{Z}}
\newcommand{\dps}{\displaystyle}

\theoremstyle{definition}
\newtheorem{remarks}[theorem]{Remarks}

\title[Group Representations on  KH-Modules and the FZ-Theorem]{A Decomposition
  Theorem for Unitary Group Representations on Kaplansky--Hilbert Modules and the 
Furstenberg--Zimmer Structure Theorem}

\dedicatory{To  Rainer Nagel in tribute to his tremendous achievements within and outside mathematics}

\begin{document}

\begin{abstract}
In this paper, 
a decomposition theorem for 
(covariant) unitary group representations on Kaplansky--Hilbert
modules over \Stoneanalgebra{s} is established, which 
generalizes the well-known Hilbert space case (where it coincides with the
decomposition of Jacobs, deLeeuw and Glicksberg). 

The proof rests heavily on the operator theory on Kaplansky--Hilbert
modules, in particular the  spectral theorem for Hilbert--Schmidt 
homomorphisms on such modules.

As an application, a generalization of the  celebrated 
  Furstenberg--Zimmer structure theorem to the case of measure-preserving 
actions of arbitrary groups on arbitrary probability spaces is established.
\end{abstract}

\maketitle

\tableofcontents

\section*{Introduction}

The Furstenberg--Zimmer structure theorem is one of the central results in the
structure theory of measure-preserving systems. It was conceived by Zimmer in
\cite{Zimm1977} and, independently, by Furstenberg \cite{Furs1977}   in his seminal work on an ergodic-theoretic proof of Szemerédi's theorem.  

The fundamental insight at the heart of the Furstenberg--Zimmer theorem is the following
{\em dichotomy}: a $G$-system\footnote{Here and everywhere else in this paper,
  $G$ is an arbitrary group.}  $\bfX$ is either {\em relatively weakly mixing} with respect
to a given factor $\bfY$, or there exists  a non-trivial intermediate factor
$\bfZ$ 
which has {\em relative discrete spectrum} with respect to
$\bfY$. (One can use {\em compact} extensions instead, but that does not make a big
difference, see the notes to Part III on page \pageref{s.notesIII}.)
This dichotomy was reformulated heuristically by Tao 
in terms of {\em structure}
(rel. discrete spectrum)  and (pseudo)-{\em randomness} (rel.~weak mixing).

\medskip
In order  to understand this result (as well as our contribution in this paper),
it is helpful to consider   the ``non-relative''
situation first. Recall that a measure-preserving system $\bfX$ has {\em
  discrete spectrum} if $\Ell{2}(\uX)$ is generated by the finite-dimensional
invariant
subspaces, while it is said to be {\em weakly mixing}, if the product system
on $\uX \times \uX$ is ergodic, i.e., the fixed space of the corresponding
Koopman representation on $\Ell{2}(\uX \times \uX)$ is trivial. 

The link between elements of the fixed space of the product dynamics
and finite-dimensional invariant
subspaces consists in the following observation: if $e_1, \dots, e_n$ is an
orthonormal basis of an invariant subspace, then $\sum_{j=1}^n e_j \tensor
\konj{e_j}$ is an element of this fixed space. 
And if $\sum_{j=1}^n x_j \tensor \konj{y_j}$ is a non-zero element
of the product fixed space 
then $\spann\{x_1, \dots, x_n\}$ contains a nontrivial invariant subspace. 

The said dichotomy is hence  a corollary of the following
purely operator-theoretic  ``key lemma''.

\renewcommand{\mytheoremname}{Key Lemma}
  \begin{theorem*}
Let $T \colon G \rightarrow \mathscr{L}(H)$ be a unitary representation of a
group $G$  on a Hilbert space $H$. 
Then 
	\begin{align*}
		\Bigl\{\sum_{j=1}^n e_j \otimes \overline{e_j}\,\Big| \,  e_1, \dots, e_n \in H \textrm{ orthonormal basis of a $T$-invariant subspace}\Bigr\}  
	\end{align*}
	spans a dense subspace of the fixed space $\fix(T \otimes \overline{T}) \subset H \otimes H^*$.
  \end{theorem*}
\renewcommand{\mytheoremname}{Theorem}

\noi
(Here, $\konj{T}$ is the contragredient representation on the dual
Hilbert space $H^*$ and $\overline{x}(y) \coloneqq (y|x)$ for $x, y \in H$. If $H = \Ell{2}(\uX)$ and $T$
comes from a measure-preserving action of $G$ on the probability space $\uX$,
then---under the natural identification $H^*  = \Ell{2}(\uX)$ and 
$H \tensor H^* = \Ell{2}(\uX \times \uX)$---the representation
$T \tensor \konj{T}$ is simply the Koopmanization of the product dynamics.)

\medskip
The proof of the key lemma 
consists in two observations. The first is the identity
\begin{equation}\label{eq.HS_T}
    \fix(T \tensor \konj{T}) = \HS_T(H), \tag{1}
\end{equation}
under the natural  identification $\HS(H) \cong H \tensor H^*$ of the space of Hilbert--Schmidt operators
on $H$ with the tensor product $H \tensor H^*$,
where 
\[   \HS_T(H) \coloneq \{ A\in \HS(H)\,|\, T_t A = AT_t \quad \text{for all $t\in G$}\}
\]
is the space of $T$-intertwining Hilbert--Schmidt operators on  $H$. The
second is
\begin{equation}\label{eq.HS_T-approx}   
\HS_T(H) = \cls{\spann}\bigl\{ A\in \HS_T(H) \, |\, \text{$A$ is of
  finite rank}\bigr\}, \tag{2}
\end{equation}
a consequence of the  spectral theorem for  self-adjoint Hilbert--Schmidt  operators.

\medskip
Apart from the dichotomy result, the key lemma also accounts (with a simple proof) for
the following  decomposition of the Hilbert space into a
``discrete-spectrum'' part and a ``weakly mixing'' part. 

\begin{corollary*}\label{cor-intro}
  Let $T \colon G \rightarrow \mathscr{L}(H)$ be a unitary group
        representation 
on a Hilbert space $H$. Then $H = H_{\mathrm{ds}} \oplus H_{\mathrm{wm}}$
orthogonally 
with closed invariant subspaces 
    \begin{align*}
        H_{\ds} &\coloneq  \cl\bigcup
\{ M\subseteq H \,|\, \text{$M$ finitely-generated, $T$-invariant
          subspace}\}, \textrm{ and }\\
 H_{\wm} &\coloneq \{ x\in H \,|\, x\tensor \konj{x} \perp \fix(T \tensor
\konj{T})\}. 
  \end{align*}
\end{corollary*}

This decomposition (in slightly different form) has been applied to 
various fields (see, e.g.,  \cite[Chap.~20]{EFHN2015}, \cite{MoreRichRobe2019},
\cite{BatkGrohKunszSchr2012}) and can be seen as a special case of
the famous Jacobs--de Leeuw--Glicksberg decomposition (see the notes to Part II on
page \pageref{s.notesII} below).

\medskip
Furstenberg in \cite{Furs1977} employed these ideas and modified them in order to cover a ``relative''
situation in the presence of an extension $\bfX \to \bfY$ 
of dynamical systems, where elements of $\Ell{\infty}(\uY)$ take the
role of the scalar field and $\Ell{2}(\uX)$ becomes a module over
it. However, the technical realization of his  program   
lacks the simple  and elegant structural divison into  a purely 
functional-analytic, representation-theoretic part and its application
to dynamical systems. Rather, it exhibits
a strong reliance on measure-theoretic tools: 
disintegration of measures, 
measurable Hilbert bundles, almost-everywhere arguments almost
everywhere.
(The same applies to Zimmer's alternative approach to the dichotomy
theorem   in \cite{Zimm1977}   as well as  to Glasner's presentation in \cite{Glas}.) 
As a consequence, the classical version of the  Furstenberg--Zimmer theorem is restricted  to actions of Borel
groups on  standard Lebesgue spaces. (We shall  call this the {\em separability
  restriction} in the following.)

\vanish{
On the other hand, the whole matter has had from
the very beginning on  also a strong functional-analytic component, e.g. 
through the concept of  {\em Hilbert modules} already used by Furstenberg. 
This component was strengthened by the work of Conze and Lesigne
\cite{ConzLesi1984},  and then by Tao in his influential blog notes
\cite[Chap.~2]{Tao2009}. Kerr and Li
in \cite{KerrLi2016} prove the  Furstenberg--Zimmer theorem in an (almost) 
purely functional-analytic way,
thereby effectively dispensing with the separability restriction. Recently, 
Duvenhage and King \cite{DuveKing2021} have established a version of the Furstenberg--Zimmer theorem in the
non-commutative setting (however, as they use direct integral theory, under the
separability restriction). 
}

\medskip
The goal of our paper is to develop a natural
functional-analytic (measure-theory free) framework that allows to derive the Furstenberg--Zimmer{} theorem 
in a manner which is {\em completely parallel}  to the non-relative
case sketched above (and, as a byproduct, to free  the theorem from the
separability restriction). It is based, firstly,  on the observation that the space
\[ \Ell{2}(\uX|\uY) \coloneq  \{ f\in \Ell{2}(\uX) \, |\, \E_\uY \abs{f}^2 \in
\Ell{\infty}(\uY)\},
\]
  already used by Tao in \cite{Tao2009}, is a so-called {\em \KH{} module} (KH-module)
over the Stone algebra $\A = \Ell{\infty}(\uY)$. And, secondly, that these modules
are the ``correct'' generalization of Hilbert spaces to
modules in the sense that each result from Hilbert space theory has its
natural analogue for \KH-modules. (A theory called  ``Boolean--valued analysis''
explains this phenomenon, see also the comments to Part I on page
\pageref{s.notesI}.)

Applied to our situation, the KH-module analogue of the Hilbert space
results from above---our main result in a nutshell---reads as follows (see Theorem
\ref{mainthm} for the ``full version'').

\renewcommand{\mytheoremname}{Theorem A}
\begin{theorem*}
Let $E$ be a \KH{} module over a Stone algebra $\A$, 
let $S \colon G \to \Aut(\A)$ be a representation of a group $G$ as
automorphisms on $\A$ and $T\colon G \to \End(E)$ an $S$-covariant unitary
representation of $G$ on $E$. Then
\begin{align*}
		\Bigl\{\sum_{j=1}^n e_j \otimes \overline{e_j} \,\Big| \, e_1, \dots, e_n \in E \textrm{ suborthonormal basis of a $T$-invariant KH-submodule}\Bigr\}  
	\end{align*}
spans an order-dense $\fix(S)$-submodule of $\fix(T \otimes \overline{T})$.
Moreover, $E$ decomposes orthogonally
into $T$-invariant KH-submodules $E = E_{\ds} \oplus E_{\wm}$, where
\[   E_{\ds} = \ocl\bigcup
\{ M\subseteq E \,|\, \text{$M$ finitely-generated, $T$-invariant
         submodule}\}
\]
and $E_{\wm} = \{ x \in E \,|\, x\tensor \konj{x} \perp \fix(T \tensor
\konj{T})\}$.
\end{theorem*}
\renewcommand{\mytheoremname}{Theorem}

\noi
Here, $\ocl$ denotes ``order-closure'' and refers---as well as the
term ``order-dense''---to
the concept of ``order-convergence'', which is the right translation of norm
convergence into the KH-module setting. 
(See Chapter \ref{c.covrepKH} and the explanations of ``Key
Concepts'' below.) 
In the case $\A =
\C$, where the Hilbert module $E$ is  just an ordinary Hilbert space,
it coincides with ordinary norm convergence. Hence, Theorem A
generalizes the Hilbert space results from above.

\medskip
Theorem A will be proved (as Theorem \ref{mainthm}) in Part II. 
The proof is completely parallel to the Hilbert space case $\A = \C$
sketched
above. In particular, it rests on KH-analogues of \eqref{eq.HS_T} and
\eqref{eq.HS_T-approx}, see \cref{invariance} and \cref{keylemma}. 
 
In Part I,  we provide the necessary background on Stone algebras and
\KH{} modules, in particular the spectral theorem for self-adjoint
Hilbert--Schmidt homomorphisms (Theorem \ref{spt-KH}). 

Part III contains the application to extensions of dynamical systems. 
Our exposition differs from conventional ones in that we, following
the approach in \cite{EFHN2015}, exclusively work
in the functional-analytic category with the corresponding notion of
(Markov) isomorphisms (cf.~the introductory remarks in Chapter
\ref{s.extensions}
and Definition \ref{d.extension}).  

The main link between the abstract KH-module results and the 
dynamical systems world
is the isomorphism 
\[     \Ell{2}(\uX|\uY) \tensor \Ell{2}(\uX|\uY) \cong\Ell{2}(\uX \times_\uY
  \uX|\uY),
\]
where $\uX\times_\uY \uX$ is the relatively independent joining (Proposition
\ref{joiningvstensor}). In this context we define 
{\em couplings} and {\em joinings} and prove their relation
to intertwining Markov operators (Proposition \ref{coupling} and
 Lemma \ref{coup=join}). This happens  
in an elegant and brief  functional-analytic way and is, possibly, 
of independent interest.

On this conceptual basis, the application to
extensions of dynamical systems is then completely analogous to the
``non-relative'' case. It results in the definition 
of the relative Kronecker factor and the
characterizations of the corresponding Kronecker subspace $\scrE(\uX|\uY)$ 
(Proposition \ref{Kronecker}) as well as 
 its orthogonal complement (Propositions \ref{Kronecker-orth}
and \ref{Kronecker-orth-amen}). Again in a nutshell, the results can
be
summarized as follows (see \cref{fixgenLinfty}, \cref{Kronecker} und
\cref{Kronecker-orth}
for the full versions).

\renewcommand{\mytheoremname}{Theorem B}
\begin{theorem*}
Let $J\colon (\uY;S) \to (\uX;T)$ be an extension of measure-preserving 
$G$-systems. Then 
\begin{align*}
		\Bigl\{\sum_{j=1}^n e_j \otimes \overline{e_j} \,\Big| \, e_1, \dots, e_n \in \mathrm{L}^2(\uX|\uY) \textrm{ suborthonormal basis of an invariant KH-module}\Bigr\}  
	\end{align*}
spans an $\Ell{2}$-dense $\fix(S)$-submodule of $\fix(T \times_\uY T)$. Moreover,  $\Ell{2}(\uX)$ decomposes orthogonally into
\[  \Ell{2}(\uX) = \scrE(\uX|\uY) \oplus\scrE(\uX|\uY)^\perp,
\]
where 
	\begin{align*}
		\dps \scrE(\uX|\uY)
\coloneqq \, \cl_{\Ell{2}} \bigcup \{ M \subseteq \mathrm{L}^2(\uX)\mid M \text{ finitely-generated, $T$-invariant $\Ell{\infty}(\uY)$-submodule}\}
	\end{align*}
and for $f\in \Ell{2}(\uX)$:
\[ f\perp \scrE(\uX|\uY) \quad\gdw\quad 
\inf_{t\in G} \max_{g\in F} \norm{\E_\uY\bigl((T_tf)
  g\bigr)}_{\Ell{2}} = 0 \quad \text{for each finite} \,\, F\subseteq \Ell{\infty}(\uX).
\]

Finally, $\scrE(\uX|\uY)$ is a $T$-invariant closed unital sublattice of $\Ell{2}(\uX)$.  \end{theorem*}
\renewcommand{\mytheoremname}{Theorem}

The Furstenberg--Zimmer dichotomy  (Theorem \ref{fuzimmer1}) as well as
the full Furstenberg--Zimmer structure theorem (Theorem \ref{fuzimmer2})
are then mere corollaries.

\medskip
Our exposition is monographic in style. 
Each part has at its end an own chapter
with notes and remarks, comprising commented references to the literature.

\medskip
\subsection*{Key Concepts}
In the following we shall discuss some of the key concepts in more detail and  highlight
some particular points. 

A (pre-)Hilbert module is, roughly speaking, a space $E$ endowed with an inner product 
$(\cdot|\cdot)$ that takes
values in a commutative unital C$^*$-algebra $\A$. By the Gelfand--Naimark
theorem, $\A \cong \Ce(\Om)$ for some compact Hausdorff space $\Om$, and
$\A$ carries a canonical lattice structure. 
As a consequence, 
one obtains a ``lattice-valued norm'' on $E$, defined by 
$\abs{x} \coloneq \sqrt{(x|x)} \in \A_+$, i.e.,   $E$ is naturally a
``lattice-normed space'' (a   concept  extensively studied by Kusraev and
his school, see \cite{Kusr2000}). Along with the  order structure on $\A_+$ 
there come natural (non-topological) notions of ``order-convergence''
and ``order-completeness'' on $\A$ as well as on $E$, in general coarser than
their norm-analogues. 

Whereas in the largest part of the literature on Hilbert modules these
order-based notions play no role, here they feature prominently. 
The reason is that if $\uY$ is a probability space then $\Ell{\infty}(\uY)$
is order-complete,  and a sequence $(f_n)_n$ in $\Ell{\infty}(\uY)$ 
order-converges if and only if it is uniformly
bounded and almost everywhere convergent (Lemma \ref{ordervsl2}). 
Hence, order-convergence
is a generalization of  bounded a.e.-convergence of sequences to convergence of nets.

Now, order-completeness of $\A= \Ce(\Om)$ for a compact space $\Om$ means that $\Om$
is  Stonean, i.e., extremally disconnected (Proposition \ref{charawstaralg}). Stonean spaces 
may appear exotic on a first encounter, but they are quite natural objects
from a functional analyst's point of view \cite{DDLS2016,GroevRoo2016}.
The same applies to  
order-convergence, which may be a cryptic notion in the beginning. 
In  $\Ce(\Om)$ it is related to pointwise convergence on the Stonean space $\Om$
in a peculiar way, involving the notion of (topological) almost everywhere
convergence (Lemma \ref{charorder}).

\medskip
Order-completeness of a Hilbert module $E$ over a Stonean algebra $\A$ 
means that $E$ is a \KH{} module (KH-module). This concept was
introduced---under the name ``AW$^*$-module'' and with a different
but equivalent definition---by Kaplansky in \cite{Kapl1953}. 
Our definition follows Kusraev \cite{Kusr2000} and has the advantage that 
(1) one can avoid the concept of a
``mixing'' and (2) the parallelism with conventional Hilbert space theory 
becomes strikingly apparent. Since, as far as we can see,  a coherent and sufficiently
complete
account of KH-module theory is missing in the literature, we 
took the opportunity to sketch this theory in Chapter \ref{c.KHM} (with
explicit proofs reduced to a minimum). 

In contrast to the cursory treatment of the general theory, 
we included a quite detailed proof of the spectral theorem
for self-adjoint Hilbert--Schmidt homomorphisms  on KH-modules
(Theorem \ref{spt-KH}),
the key auxiliary result in the paper. 
This theorem can be derived from the
spectral theorem for self-adjoint ``cyclically compact'' operators
obtained by Gönüllü in  \cite{Gonu2016}, cf.\,Remark \ref{spt-cycp}.    
However, cyclical compactness is a technically involved concept, whence
we decided to present an alternative proof, based on 
the {\em Hilbert bundle representation}.

\medskip
A Hilbert module $E$ over $\A = \Ce(\Om)$ can be 
identified with the space $\Gamma(H)$ of continuous sections of
a (canonically constructible) topological Hilbert  
bundle $H$ over $\Om$. (This establishes a categorical equivalence
reminiscent of the Serre--Swan duality, see \cref{equiv}). 
By virtue of such a Hilbert bundle 
representation, one can literally 
employ already known Hilbert space results to prove their KH-module analogues,
and we apply this method at one decisive point in the proof of the spectral
theorem (Section \ref{s.bundleview}).

Note that our use of Hilbert bundles  differs from the classical one, where only {\em measurable}
Hilbert bundles (equivalently: direct integrals) 
were used (for example by Zimmer \cite{Zimm1976} and Glasner
\cite[Chap.~9]{Glas}), of course under the  usual separability
restriction.
In our topological  approach we dispense with this
restriction and, what is more, we have to say ``almost everywhere'' almost 
nowhere.

\medskip

\medskip
\subsection*{Related Approaches and Further Applications}

In \cite{KerrLi2016}, Kerr and Li prove the  Fursten\-berg--Zimmer structure theorem effectively without
separability restriction. Following Tao \cite{Tao2009} 
they employ  the rather technical concept of conditional
compactness.  Also, instead of developing a structural
result like our Theorem \ref{mainthm}, they work 
  ad hoc with the Hilbert module  $\Ell{2}(\uX|\uY)$ which, however, is {\em not } the same as ours. (Theirs is norm-closed but not order closed.)
As a result, they only obtain  ``approximate statements'' and the 
parallels with the non-relative case are not very apparent.  

\medskip
The deeper roots of our approach lie in the conviction that  
the theory of measure-preserving systems 
often is better studied in certain categories of functional-analytic
objects (like AL-spaces
with a distinguished quasi-interior point)
than in the category  of classical ergodic theory 
(probability spaces and  measurable point mappings). This conviction is old and goes
back at least to the work of our teacher and mentor, Rainer Nagel, on
ergodic theory in the 1970s and '80s (see, e.g. \cite{DernNagePalm1987}).  
It is present in a
for a long time  neglected paper by Ellis \cite{Elli1987} and it pervades the book
\cite{EFHN2015} which, however, does not make explict use of the language of
category theory (on purpose).

In a series of papers
\cite{JamnTao2020bpre,JamnTao2019pre,JamnTao2020apre}
Jamneshan and Tao develop an alternative point-free
approach to ergodic theory based on abstract measure theory.  
Combining this perspective with conditional analysis, Jamneshan
develops an uncountable Furstenberg--Zimmer structure theory in a
parallel work \cite{Jamn2020pre}.   In a work of Jamneshan and Spaas this is extended even further to the setting of dynamical systems on von Neumann algebras (see \cite{JaSp2022}).   See also the notes  to Part I on
page \pageref{s.notesI}. We also mention that in a recent article of Eisner (see \cite{Eis}) the decomposition presented in Theorem B is stated for $\Z$-actions and used to give a new ergodic theoretic proof of Szemerédi's theorem.

\medskip

The techniques developed here form the basis for several more recent works: In \cite{EdKr2022} a theorem of Lindenstrauss on the connection between topological measurably distal systems is generalized to $\Z$-actions on general probability spaces;  and \cite{Ede2022} contains a Furstenberg--Zimmer structure theorem for stationary random walks. An \enquote{uncountable version} of Austin's Mackey--Zimmer type representation theorem for ergodic extensions with relative discrete spectrum is established in \cite{EJK2023}. Finally, a follow-up article \cite{HaKr2023} on KH-dynamical systems is in preparation,  see also the notes to Part III on page \pageref{s.notesIII}.

\medskip

\subsection*{Notation and Terminology}

The set of natural numbers, in our use, is $\N \coloneq \{1, 2, 3, \dots\}$,  and
$\N_0 \coloneq \N \cup\{0\}$. Generic probability spaces are denoted
by $\uX, \uY, \uZ$ where, e.g.,  $\uX = ( X, \Sigma_\uX,
\mu_\uX)$. Integration with respect to $\mu_\uX$ is denoted by 
\[   \int_\uX f \coloneq \int_X f \, \mathrm{d}\mu_\uX \qquad (f\in \Ell{1}(\uX)).
\]
Here, $\Ell{p}(\uX) = \Ell{p}(\uX; \C)$ denotes the space of
equivalence classes of $\C$-valued $p$-integrable functions.

If $E,F$ are normed spaces, then $\BL(E;F)$ denotes the space of all
bounded linear operators $E \to F$. We abbreviate the
linear span of some subset $M$ of a vector space $E$ with
$\spann(M)$. If $E$ is also an $\A$-module, where $\A$ is some
C$^*$-algebra, then $\spann_\A(M)$ denotes the $\A$-linear span of $M$.

We abbreviate ``compact Hausdorff topological space'' by speaking
simply of  a ``compact space''. If $\Om$ is compact, then $\Ce(\Om) =
\Ce(\Om;\C)$ is the space of $\C$-valued continuous functions on
$\Om$. The group of homeomorphisms on $\Om$ is $\Homeo(\Om)$. For
$f\in \Ce(\Om)$ we abbreviate $[\, f \neq 0\,] \coloneq \{ \vom\in \Om
\,|\, f(\vom)\neq 0\}$ and likewise for expressions as  $[\, f = 0\,]$
or $[\, f \leq 0\,]$ and so on.

The generic unit in a commutative C$ ^*$-algebra is $\car$, and the
same is used for the function constantly equal to $1$ on whatever set
is considered. More generally, a characteristic function (= indicator function) 
of a set $A$ is denoted by $\car_A$. 

The closure of a set $M$ (with
respect to some given topology) is denoted by $\cl(M)$ or $\cls{M}$.
If necessary we specify the respective topology by indexing (like $\cl_{\Ell{2}}$).
Nets are indexed like $(x_\alpha)_\alpha$ as long as the underlying 
directed index set is not needed  explicitly .

\subsection*{Acknowledgements}

We thank Rainer Nagel for long-lasting support and many inspiring discussions. 
We thank Terence Tao and Asgar Jamneshan for becoming interested in our
work and for their readiness to exchange ideas on these topics, and Patrick Hermle and Sascha Trostorff for pointing out mistakes in an earlier version of this article. We are also grateful towards the anonymous referees for their valuable feedback.

Nikolai Edeko and Henrik Kreidler are grateful to the MFO 
for providing an opportunity to work on this project in a stimulating atmosphere.
Henrik Kreidler also  acknowledges the finanical support from the DFG (project number  451698284).

Parts of this work were conceived when Markus Haase spent a research semester at
UNSW, Sydney.  He gratefully acknowledges the kind invitation by Fedor
Sukochev, inspiring discussions with Thomas Scheckter 
and the financial support by the DFG (project number 431663331). He
also thanks U\u{g}ur Gönüllü for some helpful conversations about cyclical compactness.

\part{Stone Algebras and Kaplansky--Hilbert Modules}\label{p.KHM}

Our general goal is to provide a conceptual view on the dichotomy at the heart of the Furstenberg--Zimmer structure
theory and the underlying decomposition. This part therefore gathers all the relevant
terminology and background knowledge about Hilbert modules that generalize concepts and results from regular Hilbert
space theory. Since many Hilbert
space results do not carry over to general Hilbert modules, we focus on Kaplansky--Hilbert modules which are, in many ways,
a well-behaved generalization of Hilbert spaces. Since KH-modules are defined in terms of a notion of 
order-completeness, order structures play an essential role and replace measure-theoretic almost everywhere
notions. As we will see in Part III, the conditional $\uL^2$-space
$\uL^2(\uX|\uY)$ is indeed a Kaplansky--Hilbert module over $\uL^\infty(\uY)$ and the reader is encouraged 
to think of this example along the way.

\section{Lattice-Normed Spaces and Stone Algebras}

\medskip

\subsection{Hilbert Modules}\label{s.HM}

A pre-Hilbert space is a vector space $H$ over the field $\C$ together with a mapping $(\cdot | 
\cdot) \colon H \times H \to \C$ obeying certain rules. In order to  arrive at the definition of a pre-Hilbert module, the field $\C$ is replaced by a unital commutative $\uC^*$-algebra $\A$. Recall here that, by the Gelfand--Naimark 
representation theorem, any such $\A$ may be thought of as $\uC(\Om)$  for a compact space $\Om$ (see \cite[Sec.s 1.4 and 
1.5]{Dixm1977} or \cite[Sec.~4.4]{EFHN2015}). In particular, concepts such as complex conjugation, the modulus and square 
roots 
are defined in any unital (commutative) C*-algebra via the continuous functional calculus.

We now collect some basic notions of Hilbert module theory; for more detailed information we refer to \cite{Lanc1995}.

\begin{definition}\label{HM.d.HM}
  Let $\A$ be a unital commutative $\uC^*$-algebra. A unital\footnote{ This means that $\car \cdot x = x$ for every $x \in E$. } $\A$-module $E$ equipped with a mapping
    \begin{align*}
      \left( \cdot | \cdot \right) \colon E \times E \to \A
    \end{align*}
is called  a {\emdf pre-Hilbert module over $\A$} if the following conditions are satisfied.
    \begin{itemize}
      \item[1)] For $x \in E$ we have $(x|x) \geq 0$. Moreover, $(x|x) = 0$ if and only if $x = 0$.
      \item[2)] The map $(\cdot |y)\colon  E \to \A, \, x \mapsto (x|y)$ is $\A$-linear for every $y \in E$.
      \item[3)] $\overline{(x|y)} = (y|x)$ for all $x,y \in E$.\footnote{ Complex conjugation is given by the involution in $\A$. }
    \end{itemize}    
\end{definition}

\noi
In a pre-Hilbert module $E$ the {\emdf Cauchy--Schwarz inequality}
  \begin{align*}
    |(x|y)| \leq \sqrt{(x|x)} \cdot \sqrt{(y|y)}
  \end{align*}
  holds for all $x, y \in E$.\footnote{This is due to the fact
          that $\A$ is commutative and can be proved as in \cite[page 49]{DuGi1983} by using the Cauchy-Schwarz inequality for scalar-valued positive-sesquilinear forms; in general one only has a weaker inequality, cf.\ \cite[Prop. 1.1.]{Lanc1995}.} 
As a consequence, by
  \begin{align*}
    \|x\| \coloneqq \|(x|x)^{\frac{1}{2}}\|_\A = \|(x|x)\|_\A^{\frac{1}{2}} \quad \textrm{for } x \in E
  \end{align*}
a norm $\|\cdot\|$ is defined on $E$.  The pre-Hilbert module $E$ is called a {\emdf Hilbert module}, if it is complete with respect to this norm. Note that a (pre-)Hilbert module
over $\A = \C$ is nothing but a usual (pre-)Hilbert space.

One says that  $x, y \in E$ are {\emdf orthogonal}  if $(x|y) = 0$, and for a subset $M\subset E$ we define 
the {\emdf orthogonal complement} $M^\perp$ as
\begin{align*}
  M^\perp \coloneqq \{x\in E \mid (x|y) = 0 \textrm{ for every } y \in M\}.
\end{align*}

  Given pre-Hilbert modules $E$ and $F$,   an $\A$-linear map  $T\colon E \to  F$ 
is called a {\emdf module homomorphism}. The space of
{\em bounded} module-homomorphisms is
\[  \Hom(E;F)
\]
with $\End(E) \defeq \Hom(E;E)$. Obviously, $\Hom(E;F)$ is a closed subspace of
$\mathscr{L}(E;F)$ (even with respect to the weak operator topology) and,
naturally, an $\A$-module.

A module homomorphism
$T\colon E \to F$ is called {\emdf $\A$-isometric} if
  \begin{align*}
    (Tx|Ty) = (x|y) \quad \textrm{for all } x,y \in E.
  \end{align*} 
By polarization, this is equivalent to $(Tx|Tx) = (x|x)$ for every $x \in E$. Clearly,  every $\A$-isometric homomorphism
is (norm)-isometric, and hence  bounded and injective.

\begin{examples}\label{exhilbert}\mbox{}
  \begin{enumerate}[(1)]
  \item \label{exhilbert3}
 Let $\Om$ be a compact space and $H$ a Hilbert space. The space $\uC(\Om;H)$ of continuous 
 maps from $\Om$ to $H$ equipped with the pointwise scalar product defines a Hilbert module over $\uC(\Om)$.

 \item  Let $\Om$ be a non-finite compact space and consider $\uC(\Om)$ as a Hilbert module over itself. If $\vom \in \Om$ is an 
  accumulation point of $\Om$, then 
    \begin{align*}
      I_\vom \coloneqq \{ f \in \uC(\Om)\mid f(\vom)=0\}
    \end{align*}
  is a closed submodule of $\uC(\Om)$ with $I_\vom^\perp = \{0\}$.
\end{enumerate}
\end{examples}

The last example shows that a closed submodule of a Hilbert module need not be  orthogonally complemented (see \cite[page 7]{Lanc1995}).
This is just one of many  parts of Hilbert space theory 
that, contrary to what one might think at first,  do {\em not} have immediate generalizations to Hilbert modules in general. The situation is better when one considers 
the  subclass of Kaplansky--Hilbert modules, to be introduced in 
Chapter \ref{c.KHM} below.

\medskip

\subsection{Lattice-Normed Spaces}\label{s.lns}

Any Hilbert module $E$ over a unital commutative C*-algebra $\A$ admits a \enquote{vector valued norm} 
  \begin{align*}
    |\cdot| \colon E \to \A_+, \quad x \mapsto (x|x)^{\frac{1}{2}}
  \end{align*}
  where $\A_+$ denotes the cone of the positive elements of $\A$.
        This turns $E$ into a so-called lattice-normed space,
        see \cite[Chap.~2]{Kusr2000}.

\begin{definition}\label{lns.d.lns}
  Let $\A$ be a unital commutative $\uC^*$-algebra. A vector space $E$ equipped with a mapping
    \begin{align*}
      |\cdot|\colon E \to \A_+, \quad x \mapsto |x|
    \end{align*}
  is a {\emdf lattice-normed space} (over $\A$) if the following conditions are satisfied for all $x, y \in E$ and $\lambda \in 
\C$.
    \begin{enumerate}[1)]
      \item $|x| = 0$ if and only if $x = 0$;
      \item $|\lambda x| = |\lambda| \cdot |x|$;
      \item $|x + y| \leq |x| + |y|$.
    \end{enumerate}
If, in addition,  $E$ is a unital $\A$-module, and 2) holds for all $x\in E$ and all
$\lambda\in \A$, then $E$ is called a {\emdf lattice-normed module}.  
\end{definition}

A lattice-normed space $E$, like a pre-Hilbert module, carries the natural
norm
\[  \norm{x} \coloneq \norm{\abs{x} }_\A \qquad (x\in E).
\]
If $E$ is a lattice-normed module, one has $\norm{f x} \le \norm{f}_\A
\norm{x}$ for $f\in \A$, $x\in E$, hence $E$ is a normed module. 

Note that if $\A = \C$, a lattice-normed space is nothing else than a
normed space.

\begin{examples}\mbox{}
\begin{enumerate}[(1)]  
\item Let $\A$ be a unital commutative $\uC^*$-algebra.
 As mentioned before, each pre-Hilbert module over $\A$ defines a lattice-normed module.  

\item  Every unital commutative $\uC^*$-algebra $\A$ is a lattice-normed space, where
  the lattice-norm is given by the usual modulus map $\A \to \A, \, f \mapsto |f|$ defined via functional calculus.

\item 
  Let $\Om$ be a compact space and $H$ be a Hilbert space. Consider the Hilbert module $\uC(\Om;H)$ of \cref{exhilbert}, \cref{exhilbert3}. Then 
  the vector-valued norm is given by
  \begin{align*}
    \uC(\Om;H) \to \uC(\Om),\quad F \mapsto \bigl(\vom \mapsto \|F(\vom)\|\bigr).
  \end{align*}
\end{enumerate}
\end{examples}

For vector lattices there is a well-established  concept of
\emph{order-convergence} \cite[Subsection 1.3.4]{Kusr2000} 
generalizing the dominated almost everywhere convergence 
of a sequence of integrable functions on a 
probability space \cite[Subsection 1.4.11]{Kusr2000}. 
With the  lattice-norm replacing the modulus, the notion of order-convergence can be extended to lattice-normed 
spaces as follows \cite[Subsection 2.1.5]{Kusr2000}.

\begin{definition}\label{deforderconv}
  Let $E$ be a lattice-normed space over a unital commutative $\uC^*$-algebra $\A$. 
  A net $(f_i)_{i \in I}$ in $\A$ {\emdf decreases to $0$}, if
        \[ i\le j \quad\Rightarrow\quad 0 \le f_{j} \le f_{i}\quad \text{and}\quad \inf\{ f_i \mid i \in I\} = 0.
        \]        
        A net $(x_\alpha)_{\alpha}$ in $E$      
        {\emdf order-converges} (or: is {\emdf order-convergent}) to  $x \in E$ (in symbols: $\olim_\alpha x_\alpha = x$),  if there is a net $(f_i)_{i \in I}$ in 
      $\A$ decreasing to zero 
                        and satisfying
                        \[ \forall\, i \in I\,\ \exists\, \alpha_i \colon \quad |x_{\alpha} - x|\leq f_i\qquad  (\alpha \geq \alpha_i).
                        \]
 A net $(x_\alpha)_{\alpha\in A}$ in $E$      
                        is {\emdf order-Cauchy}, if the net $(x_{\alpha} - x_{\beta})_{(\alpha,\beta)\in A\times A}$ order-converges to zero.\footnote{Here, $A \times A$ is equipped with the product direction, i.e., $(\alpha_1,\alpha_2) \leq (\beta_1,\beta_2)$ for $(\alpha_1,\alpha_2), (\beta_1,\beta_2) \in A \times A$ precisely when $\alpha_1 \leq \beta_1$ and $\alpha_2 \leq \beta_2$. }
    
A subset $M \subset E$ of $E$ is {\emdf order-bounded}, if there is $f\in \A_+$ such that $|x| \le f$ for all $x\in M$. It is
{\emdf order-closed} in $E$, if the order-limit of every
order-convergent net in $M$ is also contained in $M$. The {\emdf order-closure} 
$\ocl(M) =\overline{M}^{\mathrm{o}}$ of $M \subseteq
E$ is the smallest order-closed subset of $E$ containing $M$. (In general, it is possible that not every element of $\ocl(M)$ is the limit of an order-convergent net in $M$. However, in the situations interesting to us this is indeed the case, see \cref{order-closure} below.) We say that $M$ is
{\emdf order-dense} in $E$, if $\ocl(M) = E$.

A mapping $f\colon E \to F$ between lattice-normed spaces is {\emdf order-continuous} if\\
$\olim_{\alpha} x_\alpha = x$ in $E$ implies 
$\olim_{\alpha} f(x_\alpha) = f(x)$
whenever $(x_\alpha)_\alpha$ is a net in and $x$ an element of $E$.
\end{definition}

\begin{remarks}
  \begin{enumerate}[(1)]
    \item Recall that in a normed space $E$ a net $(x_\alpha)_{\alpha}$ converges to an element $x \in E$ if and only if for every 
$n \in \N$ there is $\alpha(n)$ such that $\|x_\alpha - x\| \leq \frac{1}{n}$ for every $\alpha \geq \alpha(n)$. Hence,  
\cref{deforderconv} is obtained by replacing the scalar-valued  norm by  a vector-valued norm and the sequence $(\frac{1}{n})_{n \in \N}$ by  a net 
decreasing to zero.

\item   The order-limit of a net is unique (if it exists). Each order-convergent net is order-Cauchy and  each  order-Cauchy net is eventually
  order-bounded.
   
\item In a lattice-normed space, the vector space operations as well as the lattice-norm are order-continuous. 
In a pre-Hilbert module, the inner product and the module product are order-continuous.

\item One has $\olim_\alpha x_\alpha =x$ in a lattice-normed space $E$ if and only if
  $\olim_\alpha |x-x_\alpha| = 0$ in $\A$, and a similar statement holds for order-Cauchy nets.
\end{enumerate}
\end{remarks}

It is natural to ask and important to understand how order-convergence and norm-convergence in a pre-Hilbert module are related.

\begin{lemma}\label{lns.l.order-norm}
  Let $E,F$ be lattice-normed modules over a unital, commutative C$^*$-algebra $\A$.
  Then the following assertions hold:
  \begin{enumerate}[(i)]
  \item If $(x_\alpha)_\alpha$ is a net in $E$ and $x\in E$, then
    $\lim_\alpha x_\alpha= x$ implies $\olim_\alpha x_\alpha = x$.

    \item Order-closed subsets of $E$ are norm-closed, norm-dense subsets are order-dense. 
 
    \item A subset of $E$ is order-bounded if and only if it is norm-bounded.
      
    \item A module homomorphism $A\colon E \to F$ is bounded  if and only if it is
      order-continuous if and only if  it is order-bounded, i.e., there exists $c > 0$ such that
    $\abs{Ax}\le c \abs{x}$ for alle $x\in E$. In this case,
    the latter estimate is true with $c = \norm{A}$.

    \item If $\A = \C$, then order-convergence is the same as norm-convergence.
       \end{enumerate}
   \end{lemma}

   \begin{proof}
     (i)\ Without loss of generality we may suppose that $x=0$. 
     By definition of the norm on $E$, $\abs{x_\alpha} \le \norm{x_\alpha}
     \car$. It follows that
     $\olim_{\alpha} \norm{x_\alpha}\car = 0$ in $\A$, and hence
     $\olim_{\alpha} \abs{x_\alpha} = 0$ as well.

       \smsk\noi
       (ii)\ follows from (i) and 
       (iii) and (v)\ are obvious.

       \smsk\noi
       (iv)\
       Suppose first that $A$ is continuous and let $c \defeq \norm{A}$. Then $\abs{x} \le \car$ implies $\abs{Ax}\le c$.
       We claim that $\abs{Ax}\le c \abs{x}$ for all $x\in E$.
       To prove this, fix $\epsilon > 0$ and define $f \defeq
       (\abs{x} + \epsilon)^{-1}$. Then $\abs{f x} \le \car$ and hence $\abs{Ax}  \le c (\abs{x} + \epsilon)$. Letting $\epsilon \to 0$ proves the claim.

       Next, suppose that there is $c\ge 0$ with $\abs{Ax}\le c\abs{x}$ for all $x\in E$. Then $\olim_\alpha x_\alpha = x$ implies $\olim_\alpha Ax_\alpha = Ax$ since
$\abs{Ax_\alpha - Ax} = \abs{A(x_\alpha - x)}
         \le c \abs{x_\alpha - x}$.
           Hence, $A$ is order-continuous.
           
           Finally, suppose that $A$ is order-continuous
           but not continuous. Then there is a sequence $(x_n)_n$ in $E$ such that $\norm{x_n} \to 0$ but $\norm{Ax_n} \to \infty$. By (i), $\olim_n x_n = 0$ and hence $\olim_n Ax_n = 0$. But order-convergent nets are eventually order-bounded and hence, by (iii), norm-bounded. This is a contradiction.  
\end{proof}

\medskip

\subsection{Stone Algebras and Stonean Spaces}\label{s.sal}

Every order-convergent net in a lattice-normed space is order-Cauchy, but the converse may fail. This leads to the notion of ``order-completeness'', to be introduced
next. 

\begin{definition}
  A lattice-normed space $E$ is {\emdf (order-)complete} 
if every order-Cauchy net in  $E$ is order-convergent in $E$.
  A commutative unital $\uC^*$-algebra $\A$ is a {\emdf
    \Stoneanalgebra}\footnote{In lattice theory
the term `` Stone algebra'' has a different meaning. Since the
Stone algebras from lattice theory do not show up in this work,
there is no danger of confusion.}
if it is order-complete (as a lattice-normed space over itself).
\end{definition}

\vanish{
\Stoneanalgebra{}s are just commutative $\mathrm{AW}^*$-algebras as
introduced and studied by I.\  Kaplansky  in \cite{Kapl1951}. The term
``Stone algebra'' was introduced by Wright in \cite{Wrig1969} and 
is also used by Kusraev in \cite{Kusr2000}.
}

\medskip
A compact space $\Omega$ is called a {\emdf Stonean
  space}\footnote{The term ``Stonian space'' is also in use.}
if it is extremally disconnected, i.e., if the closure of every
open subset is open.  And a unital commutative C$^*$-algebra (or rather: a Banach lattice) is
called {\emdf Dedekind complete} if each subset of real elements,
bounded from above, has a supremum.  
In view of the Gelfand--Naimark representation theorem, the following result is a complete characterization of \Stoneanalgebra{}s.

\begin{proposition}\label{charawstaralg}
For a compact space $\Om$ the following assertions are equivalent.
	\begin{enumerate}[label={\upshape(\alph*)}]
      \item $\uC(\Om)$ is a \Stoneanalgebra.
      \item $\uC(\Om)$ is Dedekind complete.
      \item $\Om$ is a Stonean space.
    \end{enumerate}
\end{proposition}

\begin{proof}
The equivalence of (b) and (c) is standard, see for instance
\cite[Prop.~7.7]{Schaefer1974} 
or \cite[Thm.~1.5.9]{Kusr2000}. The implication 
(c)$\dann$(a) is (admittedly not very explicitly) \cite[2.4.8]{Kusr2000}, 
applied with $\mathscr{X}=\Om\times\R$, the trivial Banach bundle over $\Om$.
(One has to know that $\Om$ is Stonean if and only if it is the Stone--\v{C}ech
compactification of each of its dense subsets, see \cite[Problem 6M]{GillmanJerisonRoCF}.) 

For the implication (a)$\dann$(b) let
 $M \subset \uC(\Om)$ be a set of real-valued functions which is bounded 
  above by an element $g \in \uC(\Om)$. We consider the set $\mathcal{P}_{\mathrm{fin}}(M)$ of all finite subsets of $M$. Then
  \begin{align*}
    \mathcal{P}_{\mathrm{fin}}(M) \to \uC(\Om), \quad N \mapsto \sup(N) 
  \end{align*}
  is an increasing net in $\uC(\Om)$ bounded by $g$. If it order-converges, then
  its order-limit is necessarily the  supremum of $M$. In view of (a)
it hence suffices to show that every increasing net $(e_\alpha)_{\alpha}$ in $\uC(\Om;\R)$ which is order-bounded from above
 is order-Cauchy. Let $(e_\alpha)_{\alpha
  \in A}$ be such a net and consider the non-empty set
    \begin{align*}
      B \coloneqq \{h \in \uC(\Om;\R)\mid  e_\alpha \leq h
                  \textrm{ for every } \alpha \in A \}.
    \end{align*}
    Then $B$ is directed when equipped with the reverse order. With the product order, we therefore obtain a directed set $I = A \times B$ and a decreasing net
      \begin{align*}
        f \colon A\times B \to \uC(\Om), \quad (\alpha,h) \mapsto 2(h-e_\alpha).
      \end{align*}       
    We show that $\inf_{(\alpha,h)} f(\alpha,h) = 0$. It is clear that $f(\alpha,h) \geq 0$ for every $(\alpha,h) \in A \times 
B$. 
Now take a real-valued function $s \in  \uC(\Om)$ with $0 \leq s \leq f(\alpha,h)$  for all $(\alpha, h) \in A \times B$. Then 
$e_\alpha \leq h - \frac{s}{2}$ for all $\alpha \in A$ and therefore $h - \frac{s}{2} \in B$ for every $h \in B$. But then $h - k\frac{s}{2} \in B$ for every $k \in \N$, which implies $s = 0$.
    
    To finish the proof, take $(\alpha,h) \in A\times B$ and observe that
      \begin{align*}
        |e_{\alpha_1} - e_{\alpha_2}| \leq |e_{\alpha_1} - h| + |h - e_{\alpha_2}| \leq 2 (h - e_\alpha)
      \end{align*}
    for $\alpha_1, \alpha_2 \geq \alpha$.
\end{proof}

As a consequence of Proposition \ref{charawstaralg} we note the following 
simplified description of order-convergence when $\A$ is a \Stoneanalgebra.

\begin{lemma}\label{lem:easynetchar}
  Let $E$ be a lattice-normed space over a \Stoneanalgebra{} $\A$, let $(x_\alpha)_\alpha$ be a net in $E$ and $x\in E$. 
  Then the following assertions hold:
  \begin{enumerate}[(i)]
  \item $\olim_\alpha x_\alpha = x$ if and only if there is an index $\alpha_0$ and a net $(f_\alpha)_{\alpha\ge \alpha_0}$ in $\A$ decreasing to $0$ with
    \[   \abs{x- x_\alpha} \le f_\alpha \qquad (\alpha \ge \alpha_0).
     \]
   \item $(x_\alpha)_\alpha$ is order-Cauchy if and only if
     there is $\alpha_0$ and a net $(f_\alpha)_{\alpha\ge \alpha_0}$ in $\A$ decreasing to $0$ with
    \[   \abs{x_\beta - x_\gamma} \le f_\alpha \qquad (\beta, \gamma \ge \alpha \ge \alpha_0).
     \]
    \end{enumerate}
    \end{lemma}

    \begin{proof}
For the proof of (i), suppose that $x= \olim_\alpha x_\alpha$. Then  choose $\alpha_0$ such that the net $(x_\alpha)_{\alpha\ge \alpha_0}$ is order-bounded and define $f_\alpha \defeq \sup_{\beta\ge \alpha} \abs{x-x_\beta}$. It is easy to see that $\inf_{\alpha \ge \alpha_0} f_\alpha = 0$. The
        converse implication is trivial. Assertion (ii) follows from
        (i). 
        \end{proof}

\medskip
How are order convergence
in a \Stoneanalgebra{} $\Ce(\Om)$  
and pointwise convergence in the function space $\Ce(\Om)$ 
related? The following is the basic observation, noted by Wright in
\cite[Lem.~1.1]{Wrig1969}.

Recall that a subset $A$ of a compact space $\Om$ is  
{\emdf residual} if it contains a countable intersection of dense open sets.
Since compact spaces have the Baire property, residual subsets are
dense.  Similarly   to the terminology in measure theory, one says that 
a statement about elements $\vom$ of $\Om$ holds {\emdf almost everywhere},
when it holds for all elements of a residual subset of $\Om$.

\begin{lemma}[Wright]\label{Wright1.1}
Let $(f_i)_i$ be a family of elements in the  \Stoneanalgebra{}  $\A
= \Ce(\Om;\R)$, bounded from below, and 
$f = \inf_{i\in
  I} f_i$ its order-infimum. Then 
\[    f(\vom) = \inf_i f_i(\vom)\quad \text{almost everywhere}.
\]
\end{lemma}

Employing \cref{Wright1.1} we obtain the following result, which entails
a characterization of order-convergence for countable nets, cf.\ \cite[Subsection 0.3.3]{Gutm1993a}.

\begin{lemma}\label{charorder}
Let $\Om$ be a Stonean space. Consider the following statements for a
net $(e_\alpha)_{\alpha \in A}$ in $\uC(\Om)$ and $e  \in \uC(\Om)$.
\begin{enumerate}[(a)]
\item $(e_\alpha)_{\alpha}$ order-converges to $e$.
\item $(e_\alpha)_{\alpha}$ is eventually bounded  and 
$\lim_{\alpha} e_\alpha(\vom) = e(\vom)$ almost everywhere.
\end{enumerate}
    Then {\rm (a)} $\Rightarrow${\rm (b)}. If the index set $A$ is countable, then {\em (a)} $\Leftrightarrow$ {\rm (b)}.
\end{lemma}

\begin{proof}
By passing to $e_\alpha - e$ we may suppose $e= 0$. 

Suppose that (a) holds and pick by Lemma \ref{lem:easynetchar} an
index $\alpha_0$ and a net $(f_\alpha)_{\alpha \ge \alpha_0}$ decreasing to $0$ with 
$\abs{e_\alpha} \le f_\alpha$ for $\alpha \ge \alpha_0$. Then 
$\norm{e_\alpha} \le \norm{f_{\alpha_0}}$ for $\alpha \ge \alpha_0$. 
Moreover, pick a residual subset  $D\subseteq \Om$  such that 
$\inf_{\alpha\ge \alpha_0} f_\alpha(\vom) = 0$ for all $\vom \in
D$ (\cref{Wright1.1}). Then 
$\limsup_{\alpha} \abs{e_\alpha (\vom)} \le \inf_{\alpha\ge \alpha_0}
f_\alpha(\vom) = 0$ for all $\vom \in D$. This proves (b).

\smsk\noi
Conversely, suppose that $A$ is countable and that (b) holds. Let $\alpha_0 \in A$ with $\sup_{\alpha \geq 
\alpha_0} \|e_\alpha\| < \infty$. For every $\alpha \in A$ with $\alpha \geq \alpha_0$ we write $f_\alpha \coloneqq \sup_{\beta
\geq \alpha} |e_\alpha|$. Then $(f_\alpha)_\alpha$ is decreasing and
it suffices to show that $f \coloneq  \inf_{\alpha \geq \alpha_0}
f_\alpha = 0$. 

By hypothesis, the set 
$D \coloneq \{ \vom \,|\, \lim_{\alpha} e_\alpha(\vom) = 0\}$ is
residual. Moreover, by \cref{Wright1.1}, the sets 
\[  D' \coloneq \Bigl\{ \vom \,\big|\, \inf_\alpha
f_\alpha(\vom)= f(\vom) \Bigr\} \quad \text{and}\quad 
D_\alpha \coloneq \Bigl\{ \vom\,\big|\, \sup_{\beta \ge \alpha}
\abs{e_\beta(\vom)} = f_\alpha(\vom)\Bigr\}\quad(\alpha \ge \alpha_0)
\]
are all residual. Obviously, $f$ vanishes on the set 
$D''\coloneq D \cap D'\cap \bigcap_{\alpha\ge \alpha_0} D_\alpha$. But since
$A$ is countable, $D''$ is residual and hence dense in $\Om$. Since
$f$ is continuous, $f= 0$ as desired.
\end{proof}

Lattice-normed modules over \Stoneanalgebra{}s have particularly nice
properties. The following is a first example.

\begin{lemma}\label{order-closure}
Let $E$ be a lattice-normed module over the \Stoneanalgebra{} $\A$, and let $M
\subseteq E$ be a submodule. Then its order-closure satisfies
\begin{align*} \ocl(M) 
= \{ x\in E\,|\, \text{there is a net $(x_\alpha)_\alpha$ in $M$
    with}\,\,
\olim_{\alpha} x_\alpha = x\}.
\end{align*}
Moreover, $\ocl(M)$ is an order-closed submodule of $E$. 
\end{lemma}

\begin{proof}
The inclusion $\supset$ is trivial. For the converse fix $x\in \ocl(M)$. 
We first claim that $\inf_{z\in M} \abs{x-z} = 0$. To prove this claim, it
suffices to show that $N \coloneq \{ x\in E\,|\, \inf_{z\in M} \abs{x-z} = 0\}$ is
order-closed (as it obviously contains $M$). 

Let $(y_\alpha)_\alpha$ be a net in $N$ and $y = \olim_\alpha y_\alpha$.
Then, for each $z\in M$ and each index $\alpha$
\[     \abs{y - z} \le \abs{y- y_\alpha} + \abs{y_\alpha - z}.
\]
This implies $\inf_{z\in M} \abs{y - z} \le \abs{y - y_\alpha}$ for each
$\alpha$, hence $y\in N$ as well. It follows that $N$ is order-closed, and 
hence $x\in N$.

Next, we claim that the module $M$ is directed with respect to
\[    z \le z'\quad \stackrel{\mathrm{def}}{\Longleftrightarrow} \quad
\abs{x- z'} \le \abs{x- z}.
\]
To see this, fix $z_1, z_2\in M$ and denote $f_j \coloneq \abs{x - z_j}$ for $j = 1,
2$.  By representing $\A = \Ce(\Om)$ for a Stonean space $\Om$, we find a 
clopen subset $A\subseteq \Om$ with  $A \subseteq [f_1 \le f_2]$ and
$A^c \subseteq [f_2 \le f_1]$. Define $z \coloneq \car_A z_1 + \car_{A^c}z_2 \in M$. 
Then 
\[ \abs{x - z} = \car_A \abs{x - z_1} + \car_{A^c} \abs{x - z_2} = f_1 \wedge
  f_2,
\]
 and hence $z\ge z_1, z_2$.  Finally, observe that $x = \olim_{z\in M} z$ with
 respect to the direction on $M$. This establishes the inclusion $\subseteq$.

The proof of the remaining assertion is straightforward.
\end{proof}

\medskip

\section{Kaplansky--Hilbert  Modules}\label{c.KHM}

Having \Stoneanalgebra{}s at our disposal, we can now introduce \KH{} modules (see also \cite[Theorem 4.1]{Frank95} for further equivalent definitions).

\begin{definition}
An order-complete lattice-normed module over a Stone algebra $\A$ is
called a {\emdf Kaplansky--Banach module} (over $\A$).   
A {\emdf Kaplansky--Hilbert module} (in short: KH-module) is  an
order-complete pre-Hilbert module $E$ over a
  \Stoneanalgebra{}.

A submodule of  a Kaplansky--Hilbert module $E$ 
is called a {\emdf Kaplansky--Hilbert submodule} (KH-submodule) of $E$ if it is order-closed in $E$.
\end{definition}

\vanish{
Kaplansky--Hilbert modules were  initially introduced as {\emdf
  $\mathrm{AW}^*$-modules} by I.\ Kaplansky in \cite{Kapl1953} 
with a slighly different, but equivalent, definition. See also \cite[Sec.~7.4]{Kusr2000}. The name itself was coined by Wright in \cite{Wrig1969KHM}.
}

By \cite[Thm.s 2.2.3 and 7.1.2]{Kusr2000} a Kaplansky--Banach
module is automatically complete with respect to the norm.
We shall focus on \KH{} modules. The only Kaplansky--Banach module
interesting to us and not being a KH-module is the space $\Hom(E;F)$
of bounded module-homomorphisms between KH-modules, see Section
\ref{s.hom} below.

\begin{examples}\label{statdardexaw}\mbox{}
\begin{enumerate}[(1)]  
\item Our standard example for a KH-module is $\Ell{2}(\uX|\uY)$, when 
$\uX \to \uY$ is an extension of probability spaces, see Definition  \ref{conditionall2}
and Proposition \ref{p.condL2-KH} below. Other examples of KH-modules are described in 
\cite[Subsection 7.4.8]{Kusr2000}.

\item \label{conterex:aw} 
Let $\Om$ be a Stonean space and $H$ a Hilbert space. Then the Hilbert
module $\uC(\Om;H)$ over $\uC(\Om)$ (see \cref{exhilbert}, 
\cref{exhilbert3}) is a Kaplansky--Hilbert module if and only if $\Om$ is finite
or $H$ is finite-dimensional \cite[Subsection 2.3.3]{Kusr2000}.
\end{enumerate}
\end{examples}

In view of the preceding example it is useful to remark that every pre-Hilbert module over a 
\Stoneanalgebra{} admits an {\em order-completion}, see \cref{s.ocp} below.

\medskip

\subsection{Support of Elements and Modules}\label{s.supp}
Let $\A$ denote a fixed \Stoneanalgebra. When convenient,
we suppose without loss of generality $\A = \Ce(\Om)$ for a fixed Stonean space $\Om$. 

The elements of $\B \coloneq \{ p\in \A \,|\, p^2 = p\}$ are called {\emdf idempotents}. 
Endowed with the lattice structure coming from $\A$, the set  $\B$ is a complete Boolean algebra
with complementation $p^c = \car - p$ and meet $p\wedge q = pq$.    
An element $x\in E$ is called {\emdf normalized} if $\abs{x}\in \B$. The {\emdf
  normalization} of $x\in E$ is 
\[    
\frac{x}{\abs{x}} \coloneq \olim_{\veps\searrow 0} \frac{1}{\abs{x} + \veps} x
\]
and its {\emdf support} is
\[  \supp(x) \coloneq   p_x \coloneq \Bigl| \frac{x}{\abs{x}} \Bigr| = \olim_{\veps\searrow 0} \frac{\abs{x}}{\abs{x} + \veps}.  
\]
The following lemma 
collects the basic properties of these concepts.

\begin{lemma}\label{supp.l.supp}
Let $E$ be a Kaplansky--Banach module over a \Stoneanalgebra{}
$\A$. Then the normalization and the support of $x\in E$ are well-defined
and have the following properties:
\begin{enumerate}[(i)]
\item $\frac{x}{\abs{x}}$ is normalized and satisfies $\abs{x} \frac{x}{\abs{x}} = x$. 

\item  $\supp(x) = \min\{ p \in \B\,|\, px= x\} = \supp(\abs{x}) =
  \supp(\frac{x}{\abs{x}})$.

\item  $\text{$x$ is normalized}
\quad\Leftrightarrow\quad \abs{x} x= x
\quad\Leftrightarrow\quad x = \frac{x}{\abs{x}} 
\quad\Leftrightarrow\quad \abs{x} = \supp(x)$.
\item If $E$ is a KH-module, then $(x|\frac{x}{\abs{x}}) = \abs{x}$.
\end{enumerate}
\end{lemma}

\begin{proof}
Observe that $\veps \mapsto \frac{\abs{x}}{\abs{x} + \veps}$ is bounded by $\car$ and increasing
as $\veps \searrow 0$. Hence, it order-converges to its supremum $p_x$. Identifying $\A= \Ce(\Om)$
we see that $p_x = \car_{\cls{[\abs{x}> 0]}} \in \B$ and $\abs{x} = p_x\abs{x}$.
Moreover, as 
\[    \Bigl|\frac{x}{\abs{x} + \veps} - \frac{x}{\abs{x} + \delta}\Bigr|
=  \Bigl|\frac{\abs{x}}{\abs{x} + \veps} - \frac{\abs{x}}{\abs{x} + \delta}\Bigr| \qquad (\delta, \veps > 0),
\]
the net $\Bigl(\frac{x}{\abs{x} + \veps}\Bigr)_\veps$ is order-Cauchy and hence order-convergent.

\noi
(i), (ii)\ Note that $\abs{x - p_xx} = (\car- p_x)\abs{x}= \abs{x} - p_x\abs{x}=0$ and hence $p_x x = x$. It follows that $\abs{x} \frac{x}{\abs{x}} = \olim_\veps \frac{\abs{x}}{\abs{x} + \veps} x
= p_x x = x$. If  $q\in \B$  with $qx = x$, then $\abs{x} = q\abs{x}$. This implies
$\frac{\abs{x}}{\abs{x}+ \veps} = q \frac{\abs{x}}{\abs{x} + \veps}$ for all $\veps > 0$ and hence
$p_x = q p_x \leq q$. 

\noi
(iii)\ If $x$ is normalized, then 
$\abs{x}x = \abs{x} \abs{x} \frac{x}{\abs{x}} = \abs{x}\frac{x}{\abs{x}} = x$. 
If $\abs{x}x = x$ then $\frac{x}{\abs{x} + \veps} = \frac{\abs{x}}{\abs{x}+\veps}x$ for each $\veps > 0$ and hence $\frac{x}{\abs{x}} = p_x x = x$. If $x = \frac{x}{\abs{x}}$, then $\abs{x} = 
\abs{\frac{x}{\abs{x}}} = p_x$ and if $\abs{x} = p_x$ then $x$ is
normalized. 

\noi
(iv)\ If $E$ is a KH-module, then $(x|\frac{x}{\abs{x}}) = (\abs{x}
  \frac{x}{\abs{x}}|\frac{x}{\abs{x}}) = \abs{x} p_x = \abs{x}$.
\end{proof}

With the help of supports one
can easily derive the following:
\[     \abs{x} \wedge \abs{y} = 0\quad \Rightarrow\quad \abs{x+y} = \abs{x} + \abs{y} \qquad(x, y\in E).
\]
(The identity  $\abs{x} \wedge \abs{y} = 0$ is equivalent to $p_x p_y
= 0$, from which it follows that $\abs{x} + \abs{y} = p_x\abs{x + y} +
p_y\abs{x+y} \le \abs{x+y}$.)

\vanish{(Proof: $\abs{p_xy} =  p_x \abs{y} =\olim_\veps \frac{\abs{x}\abs{y}}{\abs{x}+\veps} = 0$.
Hence $p_xy = 0$ and $\abs{x+y} = p_x\abs{x+ y} + p_x^c\abs{x+y} = p_x\abs{x} + p_x^c\abs{y}
= \abs{x} + \abs{y}$ as claimed.)} The idempotent
\[ \supp(E) \coloneq p_E \coloneq \sup\{p_x \,|\, x\in E\} = \inf\{ q \in \B \,|\, q^cE = \{0\}\}
\]
is called the {\emdf support} of $E$. The following is \cite[Lem.~5]{Kapl1953}, but due
to our different definition of a Kaplansky--Banach module we give a different proof here.

\begin{lemma}\label{realizesupp}
Let $E$ be a Kaplansky--Banach module. 
Then there is a normalized element $x\in E$ such that $p_x = p_E$.
\end{lemma}

\begin{proof}
Define a partial order on $\calN \coloneq \{ x\in E\,|\, \abs{x}\in \B\}$   by $x\le y \logeq x= p_xy$.
Let $K\subseteq \calN$ be any chain and $q \coloneq \sup_{a\in K} \abs{a} \in \B$. 
Then, for any $a,x,y\in K$ and $x,y \ge a$
\[   \abs{y-x}\le \abs{y-a} + \abs{x-a} = p_a^c (\abs{x} + \abs{y})\le 2 p_a^cq.
\]
This shows that the increasing net $(x)_{x\in K}$ is order-Cauchy. Hence it converges, and its limit
is an upper bound for $K$. Zorn's Lemma yields a maximal element $x$ in $\calN$. 

Let $y \in E$ be arbitrary and define $z \coloneq x + p_x^c\frac{y}{\abs{y}}$. Then 
$z\in \calN$ and $p_x z = x$, hence $z=x$ by maximality. This implies $p_x^c p_y = 0$, i.e.,
$p_y \le p_x$. It follows that $p_x= p_E$ as desired.  
\end{proof}

\medskip
If $\A = \Ce(\Om)$, then $\supp(x) = \car_{\cls{[\abs{x} \neq 0]}}$ for
$x\in E$. In particular, for $x\in \Ce(\Om)$  one has
$x\neq 0$ almost everywhere (see Section \ref{s.sal} before Lemma \ref{Wright1.1}) if and only if $\supp(x) = \car$.
Consequently, one can speak of statements holding almost everywhere
even without explicit recourse to the representation $\A
=\Ce(\Om)$.\label{almosteverywhereinalgebra}  Here is an illustration.

\begin{lemma}\label{ntf-ae}
Let $f,g\in \A$  with $fg = 0$ and $f\neq 0$ almost everywhere. Then $g=0$.
\end{lemma}

\begin{proof}
Since $fg = 0$, one has $\frac{f}{\abs{f}}g = 0$ and hence $0 = p_f \abs{g}
  = \abs{g}$. 
\end{proof}

\medskip

\subsection{Bounded Module Homomorphisms}\label{s.hom}

Let $E,F$ be lattice-normed  modules over a \Stoneanalgebra{} $\A$. 
Recall from \cref{lns.l.order-norm} that a module homomorphism $T\colon E \to F$ is
bounded if and only if it is order-continuous. The space
$\Hom(E;F)$ of all bounded module homomorphisms is a unital $\A$-module
in a canonical way. 

We now turn $\Hom(E;F)$ into a lattice-normed module. To this end, define
the {\emdf operator lattice-norm} of $T\in \Hom(E;F)$ by 
\[ \abs{T} \coloneq \sup_{\abs{x}\le \car} \abs{Tx}.
\]
Here the supremum is taken in $\A_+$, and this supremum exists 
since the set $\{ \abs{Tx}\,|\, x\in E,\, \abs{x}\le \car\}$ is bounded from above, e.g., by
$\norm{T}\:\car$. The following identities and inequalities are analogous to their well-known counterparts
in Banach space theory.

\begin{proposition}\label{completehom}
Let $E, F$ be lattice-normed  modules over a \Stoneanalgebra{} $\A$. Then
      \begin{align*}
        \left| \cdot \right| \colon \Hom(E;F) \to \A_+,
                          \quad T \mapsto \abs{T} = \sup_{|x| \leq \car} |Tx|
      \end{align*}
turns $\Hom(E;F)$ into a lattice-normed module. Moreover, $\norm{T} =
\norm{\abs{T}}_\A$ and 
\begin{equation*}
 \abs{Tx}\le \abs{T} \abs{x}\qquad (T\in \Hom(E;F),\, x\in E).
\end{equation*}
If $G$ is another lattice-normed module over $\A$ and $S\in \Hom(F;G)$, then
$\abs{ST} \le\abs{S}\abs{T}$. 
If $F$ is order-complete, then so is $\Hom(E;F)$. 
\end{proposition}

\begin{proof}
It is straightforward to show that $E$ is a lattice-normed module. 
Let $T\in \Hom(E;F)$ and $x\in E$. 
Fix $\veps > 0$ and define $f \coloneq
(\abs{x} + \veps \car)^{-1}$. Then $\abs{fx} \le \car$ 
and hence, after multiplying with $1/f$,   
$\abs{Tx} \le \abs{T} (\abs{x} + \veps \car)$. This implies    $\abs{Tx} \le \abs{T} \abs{x}$. 

Next, since  $\abs{x} \le \car$ is equivalent to $\norm{x}\le
1$,
\[  \norm{T} = \sup_{\abs{x}\le \car} \norm{ \abs{Tx} }_\A
\le \sup_{\abs{x}\le \car} \norm{ \abs{T} \abs{x} }_\A = 
\norm{\abs{T}}_\A \sup_{\abs{x}\le \car} \norm{\abs{x} }_\A \le
\norm{\abs{T}}_\A.
\]
The converse inequality follows since $\abs{T} \le \norm{T}\car$
(cf.\ the proof of Lemma \ref{lns.l.order-norm}, part (iv). 

The  inequality $\abs{ST} \le \abs{S} \abs{T}$ is now proved
exactly as in the case $\A= \C$.  The same applies to
the proof that $\Hom(E;F)$ is order-complete whenever $F$ is.
(One has to replace sequences by nets and norm convergence by order convergence.)
\end{proof}

A module homomorphism $T\colon E \to F $ between lattice-normed spaces is
called {\emdf $\A$-isometric} if $\abs{Tx} = \abs{x}$ for all $x\in
    E$.

\begin{proposition}\label{extension}
Let $E,F$ be lattice-normed modules over the \Stoneanalgebra{} $\A$ and let
$E_0$ be an order-dense submodule of $E$. Furthermore, let $F$ be  order-complete. 
Then each $T  \in
\Hom(E_0;F)$ has a unique extension to an element $T^E \in \Hom(E;F)$.

If $T$ is $\A$-isometric, then so is $T^E$.  The mapping
\[ \Hom(E_0;F) \to \Hom(E;F),\qquad T \mapsto T^E
\]
is an $\A$-isometric isomorphism of lattice-normed modules. 
\end{proposition}

The proof is a straightforward adaptation of the proof for the case
$\A = \C$. The same applies for the following result.

\begin{lemma}\label{3.4.5}
Let $E,F$ be lattice-normed spaces over a \Stoneanalgebra{} and $(T_\alpha)_\alpha$ a
uniformly bounded net in
$\Hom(E;F)$. 
Then the set $\{ x\in E \,|\, \olim_\alpha T_\alpha x = 0\}$ is an
order-closed submodule of $E$.
\end{lemma}

Finally, we note the following description of the support of a module homomorphism.

\begin{lemma}\label{lem:opsupp}
Let $E,F$ be a Kaplansky--Banach modules and $T\in \Hom(E;F)$. Then
$\supp(T) = \supp(\abs{T}) = \supp( \ocl\ran(T))$.
\end{lemma}

\begin{proof}
For $q\in \B$ one has
$q^cT = 0\,\,\gdw\,\, q^c \ran(T) = \{0\}\,\, \gdw\,\,
q^c (\ocl\ran(T)) = \{0\}$. 
Taking the minimum of these  $q$ yields the claim. 
\end{proof}

\medskip

\subsection{(Sub)Orthonormal Systems}\label{s.ons}

Let $E$ be a pre-Hilbert module over a unital commutative $\uC^*$-algebra $\A$. 
A family $(x_i)_{i\in I}$ in $E$ is called an {\emdf orthogonal system} if
$(x_i|x_j) = 0$ whenever $i \neq j$. In this case if $I$ is finite,
one has
\[    \Bigl| \sum_{i \in I} x_i \Bigr|^2 = \sum_{i\in I} \abs{x_i}^2.
\]
An orthogonal system $(x_i)_{i\in I}$ in $E$  
is called a {\emdf suborthonormal} system 
if each $x_i$ is normalized. In this case, the system is called {\emdf homogeneous}
if $\abs{x_i} = \abs{x_j}$ for all $i,j \in I$. In other words, a
suborthonormal system is homogeneous if $(x_i|x_j) = \delta_{ij} p$ for all $i, j\in I$ and some fixed idempotent $p\in \B$. If $p=\car$  here, then $(x_i)_{i\in I}$ is called
an {\emdf orthonormal} system.

A (sub)orthonormal system $(x_i)_{i\in I}$ is called a {\emdf (sub)orthonormal basis}
if $\{ x_i\, |\, i\in I\}^\perp = \{0\}$.  A subset $\mathcal{B}\subseteq E$ is
called a (sub)orthonormal subset (basis) if the family $(x)_{x\in \mathcal{B}}$ is
a (sub)orthonormal system (basis).

\begin{lemma}\label{fourier-prep}
Let $E$ be a pre-Hilbert module over a \Stoneanalgebra{} $\A$, 
and let $\calB\subseteq E$ be a
suborthonormal set in $E$. Then for each $x\in E$ the formal series
(= net of finite partial sums)
\[ \sum_{y\in \mathcal{B}} (x|y)y \coloneq \Bigl( \sum_{y\in F} (x|y)y
\Bigr)_{F \subseteq \calB\, \text{finite}}
\]
is an order-Cauchy net in $E$. If it converges, its order-limit $z$
satisfies $x-z \in \calB^\perp$,  
 and 
\[ \abs{z}^2 = \sum_{y\in \calB} \abs{(x|y)}^2 \le \abs{x}^2
\qquad \qquad   \text{{\emdf (Parseval identity/Bessel inequality)}}.
\] 
\end{lemma}

\begin{proof}
The proof (see \cite[Subsection 7.4.9]{Kusr2000})  very much follows the same steps as in case of Hilbert spaces (see,
e.g., \cite[Paragraph I.4]{Conw1985}).
\end{proof}

Contrary to Hilbert spaces, Hilbert modules need not have orthonormal bases. 
The best one can say is that a  Kaplansky--Hilbert module always has a
suborthonormal basis. This  follows from  \cite[Subsection 7.4.10]{Kusr2000}, but
we shall give a direct proof here.

\begin{proposition}\label{exorth}
  Let $E$ be a Kaplansky--Hilbert module. If $N \subset E$ is a 
suborthonormal subset, then there is a suborthonormal basis $B$ of $E$ with $N \subset B$.
\end{proposition}

\begin{proof}
Consider $\mathscr{N} \coloneqq \{M \subset E \mid M \textrm{ suborthonormal
  with } N \subset M\}$, 
  partially ordered by set inclusion. By Zorn's lemma we find a maximal
        element $B$ of $\mathscr{N}$. Now combine the maximality with Lemma
        \ref{realizesupp} to conclude that $B^\perp = \{0\}$.
\end{proof}

As a consequence, we obtain the following decomposition result, cf.\ \cite[Thm.~3]{Kapl1953}.

\begin{proposition}\label{fourier}\label{orthdecompaw}
Let $E$ be a KH-module  over a \Stoneanalgebra{} $\A$ and  $M \subseteq E$
a KH-submodule. Then $E = M \oplus M^\perp$.

Moreover, a suborthonormal system $\calB \subseteq M$ is a
suborthonormal basis for $M$ if and only if $M = \ocl\spann_\A(\calB)$, and in this case
\[   x \mapsto Px = \sum_{y \in \calB} (x|y)y
\]
is the projection of $E$ onto $M$ along $M^\perp$.
\end{proposition}

\begin{proof}
Employ  Lemma \ref{fourier-prep} and prove the
second assertion as for Hilbert spaces.
Since by \cref{exorth} one always finds a suborthonormal basis for
$M$, also the first assertion holds.
\end{proof}

\subsection{Operator Theory on KH-Modules}\label{s.otKH}

If $E$ is a pre-Hilbert module over a \Stoneanalgebra{} $\A$ and $y \in E$, then
  \begin{align*}
    \overline{y} \colon E \to \A, \quad x \mapsto (x|y)
  \end{align*}
is an element of the {\emdf dual module} $E^* \coloneqq \Hom(E;\A)$.
If $\A$ is a \Stoneanalgebra, one has the following strengthening,
comprising a version of the Riesz--Fréchet theorem for \KH{} modules.

\begin{theorem}\label{riesz}
  Let $E$ be a pre-Hilbert module over a \Stoneanalgebra{} $\A$. 
Then 
\[  \abs{\konj{y}} = \abs{y} \quad \text{for all $y\in E$}.
\]
If $E$ is a \KH{} module, then the mapping
    \begin{align*}
      \Theta \colon E \to \Hom(E;\A), \quad y \mapsto \overline{y}
    \end{align*}
is bijective.
\end{theorem}

\begin{proof}
By the Cauchy--Schwarz inequality in $E$, 
$\abs{\konj{y}} =\sup_{\abs{x}\le \car} \abs{(x|y)}\le \abs{y}$. 
Putting $x \coloneq (\abs{y} + \veps \car)^{-1}y$ yields
$\abs{y}^2 \le \abs{\konj{y}} (\abs{y} + \veps \car)$
for arbitrary $\veps > 0$, and hence $\abs{y} = \abs{\konj{y}}$.   The second part follows from \cite[Thm.~5]{Kapl1953}. 
\end{proof}

If $E$ is a KH-module, we equip the dual module $E^*$ with the structure of a Kaplansky--Hilbert module over 
$\A$ turning $\Theta$ into an $\A$-antilinear and $\A$-isometric bijection. 

The {\emdf conjugate  homomorphism} of $T \in \End(E)$ is $\konj{T}\in 
\End(E^*)$
  defined by $\konj{T} \nach \Theta = \Theta \nach T$ or, equivalently, 
$\konj{T} \konj{x} = \konj{Tx}$ for all $x\in E$.  This is not to be confused
with the classical dual operator $T'$, which would satisfy $T'\konj{x} =
\konj{T^*x}$, where  $T^*$ is the adjoint of $T$, to be introduced next.

\begin{corollary}\label{adjoint}
Let $E, F$ be Kaplansky--Hilbert modules. For
every $T \in \Hom(E;F)$ there is a unique module homomorphism $T^* \in \Hom(F;E)$ with
    \begin{align*}
      (Tx|y) = (x|T^*y) \quad \textrm{for all } x\in E,\: y\in
                  F.
    \end{align*}
Moreover, $(T^*)^* = T$, $\abs{T} = \abs{T^*}$ and  $\ran(T)^\perp = 
\ker(T^*)$.  
\end{corollary}

\begin{proof}
The existence of the adjoint is proved in \cite[Thm.~6]{Kapl1953}.
The remaining assertions are proved exactly as the analogues for
bounded operators on Hilbert spaces.
\end{proof}

If $E$ is a KH-module, then  with  the involution 
  \begin{align*}
    \End(E) \to \End(E), \quad T \mapsto T^*,
  \end{align*}
the space $\End(E)$ is a $\uC^*$-algebra (and even an $\mathrm{AW}^*$-algebra, see \cite{Kapl1953}). In 
particular, one can speak of {\emdf normal}, {\emdf self-adjoint} or {\emdf
  unitary} module homomorphisms on $E$.

\medskip
A homomorphism $T\colon E\to  F$ between  KH-modules $E,F$ is a {\emdf contraction}
if $\abs{T} \leq \car$ (equivalently:   $\norm{T}\le 1$). Recall that it is an isometry
if $\abs{Tx} = \abs{x}$ for all $x\in E$. The results about 
contractions, isometries and unitaries on Hilbert spaces 
listed in Appendix D.4 of \cite{EFHN2015}
carry over to such homomorphisms on KH-modules, as the proofs can be
repeated mutatis mutandis.
 
Similarly, the results of Appendix D.5 of \cite{EFHN2015} 
about orthogonal projections in and onto closed subspaces of Hilbert spaces,
carry over to KH-Modules. In particular, for $P \in \End(E)$ one has
\[ \abs{P}\le \car,\, P = P^2 \quad\Leftrightarrow\quad P^2= P= P^*
\quad\Leftrightarrow\quad \ran(\Id- P) \perp \ran(P).
\]
In this case, $P$  is called an {\emdf orthogonal projection}. 

As a consequence of \cref{fourier} one obtains, for a given
KH-submodule of a KH-module $E$, a unique orthogonal projection $P
\in \End(E)$ such that  $\ran(P)=M$ and $\ker(P) = M^\perp$.
 
Because of its importance, we note explicitly the following
{\emdf mean ergodic theorem}:

\begin{proposition}[Mean Ergodic Theorem]\label{met}
Let $E$ be  a KH-module $E$ and   $T\in \End(E)$ a contraction. Then 
$\fix(T) = \fix(T^*)$, $E$ decomposes orthogonally as $E = 
\fix(T) \oplus \ocl\ran(\Id - T)$, and for each $x\in E$ one has
\[    \olim_{n \to \infty} \frac{1}{n} \sum_{j=0}^{n{-}1} T^jx = Px,
\]
where $P$ is the projection onto $\fix(T)$ along $\ocl\ran(\Id - T)$.
\end{proposition}

\begin{proof}
The identity $\fix(T) = \fix(T^*)$ is proved exactly as 
\cite[Lem.~D.14]{EFHN2015}. Then the orthogonal decomposition  $E = 
\fix(T) \oplus \ocl\ran(\Id - T)$ follows from the last
assertion of \cref{adjoint}. To prove the  remaining statement, one
notes first that $\frac{1}{n} \sum_{j=0}^{n{-}1} T^jx\to 0$ for $x\in
\ran(\Id - T)$ and then passes to  the order closure by Lemma \ref{3.4.5}. 
\end{proof}

\medskip

\subsection{Modules and Homomorphisms of Finite Rank}\label{s.finrank}\label{s.HSfr}

A KH-module $E$ over a \Stoneanalgebra{} $\A$  is {\emdf  of
  finite rank} if it has a finite suborthonormal basis.
Note that
if $e_1, \dots, e_n$ is such a suborthonormal basis of $E$, then 
by \cref{fourier} 
\[     E = \spann_\A\{ e_1, \dots, e_n\} = \A e_1 \oplus \dots \oplus
\A e_n,
\]
i.e., the algebraic $\A$-span of the basis vectors is already
order-closed.

\begin{lemma}\label{lem:finrank}
Let $E$ be a KH-module and $n \in \N_0$. The following assertions are
equivalent:
\begin{enumerate}[(a)]
\item There are $x_1, \dots, x_n \in E$ such that $E =
  \ocl\spann_\A\{ x_1, \dots, x_n\}$.
\item There is a suborthonormal basis $e_1, \dots, e_n$ of $E$.
\item $\displaystyle \sum_{y\in \calB} \abs{y}^2 \le n\car$ for each
  suborthonormal set $\calB \subseteq E$.
\end{enumerate}
In {\rm (b)} one can achieve $\abs{e_1} \ge \abs{e_2} \ge \dots \ge \abs{e_n}$.
\end{lemma}

\begin{proof}
(b)$\dann$(a) is clear, and (a)$\dann$(b) follows from defining (Gram--Schmidt)
\[  e_1 \coloneq \frac{x_1}{\abs{x_1}},\quad e_k \coloneq
\frac{y_k}{\abs{y_k}},\quad y_k \coloneq x_k - \sum_{j=1}^{k-1}
(x_k|e_k)e_k\quad(1\le k\le n).
\]
(b)$\dann$(c): Let $F\subseteq E$ be any finite suborthonormal
system. Then (by Parseval and Bessel)
\[ \sum_{y \in F} \abs{y}^2 = \sum_{y\in F} \sum_{j=1}^n
\abs{(y|e_j)}^2 =  \sum_{j=1}^n \sum_{y\in F}
\abs{(e_j|y)}^2 \le \sum_{j=1}^n \abs{e_j}^2 \le  n \car.
\]
(c)$\dann$(b): We construct a suborthonormal system
$e_1, e_2, \dots$ recursively by requiring that $e_{k}$ is 
a normalized element of $\{ e_1, \dots, e_{k{-}1}\}^\perp$ with
maximal support (Lemma \ref{realizesupp}). Then obviously $p_E = \abs{e_1}\ge
\abs{e_2}  \ge \dots$. Clearly, (c) implies $e_{n{+}1}= 0$,  and
hence $\{ e_1, \dots, e_n\}$ is a suborthonormal basis.
\end{proof}

It follows from Lemma  \ref{lem:finrank} that {\em each KH-submodule of a 
KH-module of finite rank is again of finite rank.} 
A KH-submodule  is called {\emdf  homogeneous} (of rank $k$) if it has a homogeneous 
suborthonormal basis (of length $k$). 
Each finite-rank KH-module  decomposes into homogeneous submodules as follows.

The {\emdf dimension} of  a KH-module $E$ of finite-rank is
\[ \dim_E \coloneq \sup\Bigl\{ \sum_{y\in F} \abs{y}^2 \,\Big|\, \text{$F$ 
finite suborthonormal system in $E$}\Bigr\} = 
\sum_{y\in \calB} \abs{y}^2 \in \A,
\]
where $\calB$ is any finite suborthonormal basis of $E$. (The identity
follows from the proof of Lemma
\ref{lem:finrank}. Actually, $\calB$ may be any suborthonormal basis, but
we shall not use this fact.)

By identifying $\A = \Ce(\Om)$ 
we can interpret $\dim_E$ as a continuous, $\N_0$-valued function on
$\Om$. The number $N \coloneq \norm{\dim_E}_\infty$ is called the {\emdf
  maximal rank} of $E$.  Define the idempotents
\[  q_k \coloneq \car_{[\dim_E = k]} \in \B \qquad (0 \le k \le N).
\]
Then $(q_k)_k$ is a partition of unity of $\B$ with $q_N \neq 0$. We
obtain a decomposition of $E$ as
\[ E = q_0E \oplus q_1E \oplus \dots \oplus q_N E
\]
into its so-called {\emdf homogeneous components}. 

\begin{lemma}
In the situation just described, either $q_k = 0$ or $E_k = q_kE$ is
homogeneous of  rank $k$.
\end{lemma}

\begin{proof}
Suppose $q_k \neq 0$ and let $e_1, \dots, e_n$ be a suborthonormal basis for $E$  with 
$\abs{e_1} \ge \abs{e_2} \ge \dots \ge \abs{e_n}$. By definition, $kq_k = q_k
\dim_E = \sum_{j=1}^n q_k\abs{e_j}$. 
It follows that $q_k \abs{e_j} = q_k$ for $1\le j \le k$ and $q_k
\abs{e_j} = 0$ for $j > k$. Obviously, the system $q_ke_1, \dots, q_k
e_k$ is a homogeneous suborthonormal basis for $E_k$ of length $k$.   
\end{proof}

Let $E,F$ be Kaplansky--Hilbert modules over a \Stoneanalgebra{}
$\A$. For $y\in E$ and $z\in F$ define $A_{y,z} \in \Hom(E;F)$ by 
    \begin{align*}
      A_{y,z} \colon E \to F, \quad x \mapsto (x|z) y.
    \end{align*}
Any finite sum of homomorphisms of the form $A_{y,z}$ is 
  called a {\emdf homomorphism of $\A$-finite rank}. 
  Such homomorphisms   can be characterized as follows.
  
\begin{proposition}\label{charfinite}
  Let $E,F$ be a Kaplansky--Hilbert modules over a \Stoneanalgebra{}
  $\A$. For 
$A \in \Hom(E;F)$ the following 
  assertions are equivalent.
    \begin{enumerate}[(a)]
      \item $A$ is of $\A$-finite rank.
      \item The $\A$-module $\ran(A)$ is contained in a
        finitely-generated submodule.
      \item $\ocl\ran(A)$ is a KH-submodule of finite rank.
      \item There are $z_1, \dots, z_n \in E$ and a suborthonormal system $y_1, \dots, y_n \in F$  such that
        \begin{align*}
          A = \sum_{i=1}^n A_{y_i,z_i}.
        \end{align*}
    \end{enumerate}
\end{proposition}

\begin{proof}
(d)$\dann$(a)$\dann$(b) is clear. (b)$\dann$(c) follows from Lemma
\ref{lem:finrank}. (d) follows from (c) by letting  $y_1, \dots, y_n\in F$ be
a suborthonormal basis for $M \coloneq  \ocl\ran(A)$ and applying \cref{fourier}
and \cref{riesz} to find $z_1,\dots, z_n\in E$.
\end{proof}

\medskip

\subsection{Hilbert--Schmidt Homomorphisms}\label{s.HS}

Let $E$ and $F$ be Kaplansky--Hilbert modules over a \Stoneanalgebra{}
$\A$.  Moreover, let $\mathscr{F}$ be the 
family of all finite suborthonormal subsets of $E$. 
A homomorphism  $A\in \Hom(E;F)$ is called a {\emdf Hilbert--Schmidt
homomorphism} if
\[ \abs{A}_{\mathrm{HS}} \coloneqq \sup
                          \Bigl\{\Bigl({\sum}_{x \in \calB}
                          |Ax|^2\Bigr)^{1/2} \,\,\big|\,\, \calB \in
                          \mathscr{F}\Bigr\}
\]  
exists in $\A_+$. We write $\HS(E;F)$ for the $\A$-module of all
$\A$-Hilbert--Schmidt homomorphisms from $E$ to $F$ and $\HS(E)$ if
$E= F$.

\begin{proposition}
Let $\calB$ be a fixed suborthonormal basis of $E$. Then
for $A\in \End(E)$ the following assertions are equivalent.
\begin{enumerate}[(a)]
\item $A\in \HS(E)$.
\item $A^*\in \HS(E)$.
\item $\sum_{x\in \calB} \abs{Ax}^2$ order-converges in $\A$.
\item $\sum_{x \in \calB} \abs{A^*x}^2$ order-converges in $\A$.
\end{enumerate}
In this case, $\abs{A}_{\HS}^2 = \sum_{x\in \calB} \abs{Ax}^2$ and for
each $T\in \Hom(E)$ one has $TA, AT\in \HS(E)$ with  
$\abs{AT}_{\HS}, \abs{TA}_{\HS} \le \abs{T} \abs{A}_{\HS}$.
\end{proposition}

\begin{proof}
The proof is analogous   to that of the case   $\A= \C$, see \cite[Prop.~3.2]{Gonu2014}.
\end{proof}

The space $\HS(E)$ carries a natural KH-module structure:

\begin{proposition}\label{HSops}\label{HS.p.char}
Let $E$ be a Kaplansky--Hilbert module over a
                \Stoneanalgebra{} $\A$. Then the mapping
\begin{align*}
\HS(E) \times \HS(E)  \to \A, \quad (A,B) \mapsto (A|B)\coloneq \olim_{\calB \in \mathscr{F}} \sum_{x \in \calB} 
(Ax|Bx)
\end{align*}
turns $\HS(E)$ into a Kaplansky--Hilbert module over $\A$, and
\[ (A|B)_{\HS} = \sum_{x\in \calB} (Ax|Bx) \qquad (A, B \in \HS(E))
\]
as an order-convergent series for each suborthonormal basis $\calB$ of $E$.  
\end{proposition}

\begin{proof}
Again, the proof is analogous   to that of the case   $\A = \C$, see 
\cite[Prop.~3.3 and Thm.~3.4]{Gonu2014}.
\end{proof}

If  $y,z \in E$, then the finite-rank homomorphism $A_{y,z}$
  (cf.\ Section \ref{s.finrank}) is Hilbert--Schmidt with 
\[ |A_{y,z}|_{\mathrm{HS}} =   |y| \cdot |z|. 
\]
Indeed, if $\calB$ is any suborthonormal basis of $E$, then
\[    \abs{A_{y,z}}_{\HS}^2 = \sum_{e\in\calB} \abs{ (e|z) y}^2
= \abs{y}^2 \sum_{e\in\calB} \abs{ (z|e)}^2 = \abs{y}^2 \abs{z}^2
\]
by Parseval. 

\begin{lemma}\label{lem:approxHS}
  Let $E$ be a Kaplansky--Hilbert module over a \Stoneanalgebra{} $\A$. Then 
  the space of $\A$-finite-rank homomorphisms 
  is order-dense in $\HS(E)$.
\end{lemma}

\begin{proof}
This is again analogous to the Hilbert space case. Let $\calB$ be any
suborthonormal basis of $E$. Define, for each $G \subseteq \calB$
finite, $A_G \coloneq \sum_{e\in G} A_{Ae,e} = \sum_{e\in G} (\cdot |e)Ae$. Then 
\[  \abs{A - A_G}^2_{\HS} = \sum_{e\in \calB \ohne G} \abs{Ae}^2 \to
0\quad \text{as $G \nearrow \calB$}.\qedhere
\]
\end{proof}

\medskip

\subsection{Order-Completion and Tensor Products}\label{s.ocp}

Just as each pre-Hilbert space has a Hilbert space completion, 
each pre-Hilbert module over a \Stoneanalgebra{} has an
``order-completion'' in the following sense.

\begin{proposition}\label{completion}
   Let $E$ be a pre-Hilbert module over a \Stoneanalgebra{} $\A$.
   Then there is a Kaplansky--Hilbert module $E^\sim$ and an
         $\A$-isometric homomorphism $\iota \colon E \to 
E^\sim$ with order-dense range. 
\end{proposition}

\begin{proof}
The double dual $E^{**}$ of $E$ is an
order-complete lattice-normed module. The mapping
\[ \iota \colon  E \to E^{**},\qquad   x \mapsto \iota(x) \coloneq (x^* \mapsto x^*(x))
\]
is $\A$-isometric. Indeed, $\abs{\iota(x)}  = \sup_{\abs{x^*}\le \car}
\abs{ x^*(x)} \le \abs{x}$ and, by \cref{riesz}, 
\begin{align*}
              \abs{\iota(x)} & = \sup_{\abs{x^*}\le \car} \abs{ x^*(x)} 
\ge \sup_{\abs{y}\le \car} \abs{\konj{y}(x)} 
= \sup_{\abs{y}\le \car} \abs{(x|y)} = \abs{x} 
\end{align*}
for each $x\in E$. Let $E^\sim$ be the order-closure of $\iota(E)$
within $E^{**}$.  Then $\iota(E)$ is an order-dense submodule of
$E^\sim$. Finally, transport  the $\A$-valued inner product from $E$
to $\iota(E)$ and then extend it (e.g. by a twofold application of
\cref{extension}) to all of $E^\sim$. 
\end{proof}

The order-completion 
of a pre-Hilbert module $E$ over a  \Stoneanalgebra{} $\A$ is unique up to a
canonical $\A$-isometric isomorphism. We will therefore speak of \emph{the}
order-completion of $E$ in the following. 

\begin{example}\label{comp-COm}
Let $\Om$ be a Stonean space and let $H$ be a Hilbert space. The 
order-completion of the Hilbert module
$\Ce(\Om; H)$ (see \cref{exhilbert}, \cref{exhilbert3}) can be identified with
\[ \Ce_{\#}(\Om;H) \defeq  \ell^\infty_c(\Om; H) / \ell^\infty_0(\Om; H).
\]
Here, $\ell^\infty_c(\Om; H)$ is the space of all bounded functions $\Om \to H$
which are continuous on a residual set, and $\ell^\infty_0(\Om; H)$ is the
subspace of functions that vanish almost everywhere. See \cite[2.3.3]{Kusr2000}. 
\end{example}

Using the order-completion, one can 
construct the tensor product of Kaplansky--Hilbert modules. 

\begin{definition}
  Let $E$ and $F$ be Kaplansky--Hilbert modules over a \Stoneanalgebra{} $\A$. The algebraic
  tensor product $E \otimes_\mathrm{alg} F$  as $\A$-modules is equipped with the $\A$-valued inner product
  \begin{align*}
    (\cdot | \cdot) \colon (E \otimes_\mathrm{alg} F) \times (E \otimes_\mathrm{alg} F) \to \A
  \end{align*}
  defined on elementary tensors by 
  \[ (x \otimes u|y \otimes v) \coloneqq 
  (x|y) \cdot (u|v) \qquad (x \otimes u, y\otimes v \in E \otimes_{\mathrm{alg}} F),
\]
see  \cite[pages 40--41]{Lanc1995}. Its order-completion is called
  the {\emdf tensor product} $E \otimes F$. 
\end{definition}

See  \cite[Chap.~4]{Lanc1995} for 
the tensor product of  general Hilbert modules, but note that in our
construction the {\em order}-completion is used instead of the norm-completion.

\medskip

As in the case of ordinary Hilbert--Schmidt operators, 
the space $\HS(E)$ can be identified with a tensor product.

\begin{proposition}\label{tensordescription}
  Let $E$ be a Kaplansky--Hilbert module over a \Stoneanalgebra{} $\A$. Then there is a 
  unique $\A$-isometric isomorphism
    \begin{align*}
      V \colon E \otimes E^* \to \HS(E)
    \end{align*}
  with $V(y \otimes \overline{z}) = A_{y,z}$ for all $y,z \in E$.
\end{proposition}
    \begin{proof}
      The mapping 
        \begin{align*}
          E \times E^* \to \HS(E), \quad (y,\overline{z}) \mapsto  A_{y,z}
        \end{align*}
      is $\A$-bilinear, and therefore induces an $\A$-linear map $W\colon E \otimes_{\mathrm{alg}} E^* \to \HS(E)$ 
on the algebraic tensor product $E \otimes_{\mathrm{alg}} E^*$. Moreover, for $\sum_{i=1}^n y_i \otimes \overline{z_i} \in E 
\otimes_{\mathrm{alg}} E^*$ and any suborthonormal basis $B \subset E$
    \begin{align*}
      \sum_{x \in B} \Bigl|\sum_{i=1}^nA_{y_i,z_i}x\Bigr|^2 &= \sum_{x \in B}\sum_{i=1}^n\sum_{j=1}^n (x|z_i) \overline{(x|z_j)} 
(y_i|y_j)
      = \sum_{i=1}^n \sum_{j=1}^n \Bigl(\sum_{x \in B} (z_j|x) \,\Big|\,  z_i\Bigr)  (y_i|y_j) \\
      &= \sum_{i=1}^n\sum_{j=1}^n (z_j|z_i) (y_i|y_j)
      = \Bigl|\sum_{i=1}^n y_i \otimes \overline{z_i}\Bigr|^2.
    \end{align*}
  By \cref{exorth}, this means
    \begin{align*}
      \Bigl|\sum_{i=1}^n y_i \otimes \overline{z_i}\Bigr| = \Bigl|\sum_{i=1}^nA_{y_i,z_i}\Bigr|_{\mathrm{HS}} 
    \end{align*}
  and therefore $W$ induces an $\A$-isometric map $V \colon [E \otimes_{\mathrm{alg}} E^*] \to \HS(E)$. 
  By \cref{extension}, $V$ extends uniquely to an $\A$-isometry
  $V \colon E\otimes E^* \to \HS(E)$. Since the range of $V$ is order-dense by \cref{lem:approxHS}, this extension 
  is surjective and hence an isomorphism.
\end{proof}

It is common to identify   $y \tensor \konj{z}$ with $A_{y,z}$ via $V$.

\section{Hilbert Bundle Representation}\label{s.hbrep}

Similar to the representation of a unital commutative $\Ce^*$-algebra as
the space $\Ce(\Om)$ of continuous functions on a compact Hausdorff space $\Om$, 
there is a representation of Hilbert modules over such algebras
as the space of {\em continuous  sections} on a {\em Hilbert bundle} 
over $\Om$. This chapter is devoted to these notions
and describing the correspondence.

\begin{definition}\label{defhilbertbundle}
  Let $\Om$ be a Hausdorff topological space. A {\emdf
          (continuous) Hilbert bundle} over $\Om$ is a topological
        space 
$H$ together with a  continuous, 
open surjection $p \colon H \to \Om$ with the following properties.
     \begin{enumerate}[1)]
    \item Every fiber $H_\vom \coloneq p^{-1}\{\vom\}$, $\vom \in H$,  is a Hilbert space.
    \item The mappings
      \begin{alignat*}{3}
        + &\colon H \times_\Om H \to H, \quad &(e,f) &\mapsto e+f\\
        \cdot &\colon \C \times H \to H, \quad &(\lambda,e) &\mapsto \lambda e\\
        ( \cdot | \cdot ) &\colon H \times_\Om H \to \C, \quad &(e,f) &\mapsto (e|f) 
      \end{alignat*}
      are continuous. Here, $H \times_\Om H \defeq \{ (e,f)\in H \times H \mid p(e)= p(f)\}$.

    \item For each $\vom \in \Om$, the sets
      \begin{align*}
        \{e \in H\mid pe \in U, \|e\| < \epsilon\}\qquad (\omega\in U \subset \Om\,\text{ open},\, \epsilon > 0)
      \end{align*}
      constitute  an open neighborhood base of 
      $0_\vom \in H_\vom$.
  \end{enumerate}
\end{definition}

A simple, but important example for a Hilbert bundle is the following.

\begin{example}
   Let $\Om$ be a Hausdorff space and $H$ be a Hilbert space. Then the product $\Om \times H$ equipped with the product topology and the 
projection onto the first component is a continuous Hilbert bundle over $\Om$ called the {\emdf trivial bundle with fiber $H$}. 
\end{example}

Let $p\colon H \to \Om$ be a Hilbert bundle.  Each (continuous) mapping
$x \colon O \to H$ with  $p \circ x = \mathrm{id}_O$ for some open
subset $O \subseteq \Om$ is called is called a 
{\emdf local (continuous) section} of $H$.  The space of all local continuous sections
of $H$ on $O$ is denoted by $\Gamma(O;H)$. Local sections with  $O =
\Om$ are called {\emdf global} sections, and we write
$\Gamma(H)\coloneq \Gamma(\Om;H)$.  One can show that for every $e \in H$ there always exists a global section $x \in \Gamma(H)$ with $x(p(e)) = e$ (see, e.g., \cite[Theorem 3.2]{Gierz1982}).

It is easy to see that $\Gamma(H)$ is a Hilbert $\Ce(\Om)$-Module.
Conversely, the following proposition---a module version 
of the Gelfand--Naimark representation theorem---tells that 
each Hilbert module is isomorphic to the module of continuous sections 
in a (naturally constructed) Hilbert bundle. See \cite[Chap.~2]{DuGi1983} for a proof.

\begin{proposition}\label{equiv}
  Let $\Om$ be a compact space and $E$ be a Hilbert module over
        $\uC(\Om)$. 
Then the disjoint union $H$ of the Hilbert spaces
    \begin{align*}
      H_\vom \coloneqq E/\{x \in E\mid |x|(\vom) =
                  0\}\qquad (\vom \in \Om)
    \end{align*}
equipped with the base point map $p \colon H \to \Om$ and the topology generated by the sets
    \begin{align*}
      V(x,U,\varepsilon) \coloneqq \left\{e \in p^{-1}(U) \mid \bigl\|e - [x]_{p(e)}\bigr\| < \varepsilon\right\}
    \end{align*}
  for $x \in E$, $U \subset \Om$ open and $\varepsilon > 0$ is a
        Hilbert bundle over $\Om$. Moreover, the mapping
    \begin{align*}
      E \to \Gamma(H), \quad x \mapsto [\vom \mapsto [x]_\vom]
\end{align*}     
  is a $\uC(\Om)$-isometric isomorphism of Hilbert modules over $\uC(\Om)$.
\end{proposition}

Suppose that $E, F$ are Hilbert modules over $\A= \Ce(\Om)$ and $T\in \Hom(E;F)$. 
Interpreting $E = \Gamma(H)$ and $F = \Gamma(K)$ for Hilbert modules $H, K$
one obtains an induced {\emdf bundle map} 
\[  T^\wedge \colon H \to K,
\]
which is overall continuous and restricts to uniformly bounded linear {\emdf fiber maps}
$T_\vom \in \BL(H_\vom; K_\vom)$ for all $\vom \in \Om$ by 
\[  T^\wedge e = (Tx)(\vom) \qquad \text{whenever} \quad x\in E,\, x(\omega) = e.
\]
Conversely, if $S\colon H \to K$ is a bundle map, then by 
\[  (S^\vee x)(\vom) \coloneq  S(x(\vom))\qquad (x\in E,\, \vom \in \Om)
\]
an element $S^\vee \in \Hom(E;F)$ is defined. The mappings $T \mapsto T^\wedge$ and  
$S \mapsto S^\vee$ are mutually inverse. (See \cite[Summary 10.18]{Gierz1982})

One can frame this in the language of category theory: The assignments $H \mapsto \Gamma(H)$ and $E \mapsto H_E$ establish
an  equivalence of the category of Hilbert bundles and bundle maps over $\Om$ on one
side and of Hilbert modules and bounded module homomorphisms over $\uC(\Om)$ on
the other  (see  \cite[Chap.~2]{DuGi1983} for more details). 
Consequently, as every unital commutative $\uC^*$-algebra is isomorphic 
to a space $\uC(\Om)$, considering Hilbert bundles over compact spaces is, at
least from a categorial perspective,  equivalent 
to considering Hilbert modules over unital commutative $\uC^*$-algebras. In particular, we can reformulate properties of 
and theorems on bundles in terms of modules and vice versa.

\section{Spectral Theory of Hilbert--Schmidt Homomorphisms on KH-Modules}

In this chapter we prove the spectral theorem for self-adjoint
Hilbert--Schmidt
homomorphisms on KH-modules.

\medskip

\subsection{Review of the Spectral Theorem on Hilbert Spaces}\label{s.sptrev}

In this section we review the classical spectral theorem for 
self-adjoint Hilbert--Schmidt operators on a Hilbert space $H$. 
We formulate and prove it in a way that can be transferred, more or
less
straightforwardly, to the module setting.

\medskip
\noindent
Let $H$ be a Hilbert space, $\HS(H)$ the space of Hilbert--Schmidt operators 
on $H$, and $A = A^* \in \HS(H)$ a
self-adjoint Hilbert--Schmidt operator on $H$. The 
following procedure which is commonly used in the proof of the spectral theorem 
for compact, self-adjoint operators shall be called the {\emdf ``spectral algorithm''} 
applied to $A$.

\smallskip

Define $\lambda \coloneq \norm{A}$ and
\[   A^\sharp  \coloneq \begin{cases} \frac{A}{\norm{A}} & (A \neq 0)\\
  0 & (A=0).
\end{cases}
\]
Then $\lambda  A^\sharp = A$. Next, let
\[  P^+ \coloneq P_{\fix(  A^\sharp )},\quad P^-\coloneq P_{\fix(-  A^\sharp )}
\]
be the orthogonal projections onto the fixed spaces of $  A^\sharp $ and $-  A^\sharp $,
respectively.  Then
\[  
  AP^+ = \lambda P^+,\quad AP^- = -\lambda P^{-}
\] 
and so one sees that $P^+$ and $P^{-}$ are also the projections onto the   eigenspaces $\ker(\lambda -A)$ and $\ker(-\lambda - A)$ of $A$ 
corresponding to $\lambda$ and $-\lambda$, respectively .
Since $A$ is self-adjoint, $P^+ P^- = P^- P^+ = 0$ \cite[Lemma D.25]{EFHN2015}. Set 
$B_1 \coloneq A$, $\lambda_1\coloneq \lambda$, $P_1^\pm \coloneq P^\pm$ and 
\[ 
  B_2 \coloneq B_1 - \lambda_1(P_1^+ - P_1^-).
\] 
Then $B_2$ is again self-adjoint and Hilbert--Schmidt,
and satisfies $B_2 \perp A(P^+ + P^-)$. 
Now repeat the procedure with $B_2$ in place of $A$, and iterate. 

\medskip
This ``spectral algorithm'' yields a sequence of pairs of projections $(P^+_j, P^-_j)_j$,  a
sequence of operators $(B_j)_j$ and a scalar sequence $(\lambda_j)_j$  with the following properties.
\begin{enumerate}[(i)]
\item All the occurring projections $P^+_j, P^-_k$ are pairwise
orthogonal.
\item $P_j^{\pm}A = AP_j^{\pm} = \pm \lambda_j P_j^\pm$ for all $j\ge
1$.
\item All operators $B_j$ are self-adjoint and Hilbert--Schmidt.
\item $\lambda_{j} = \norm{B_{j}}$ for all $j \ge 1$ (with $B_1 = A$).
\item $\lambda_1 \ge \lambda_2 \ge \dots$
\item For all $j\ge 1$: $\, \lambda_j \neq 0\,\, \dann\,\, \lambda_j > \lambda_{j{+}1}$ 
and $P_j^++  P_j^-\neq0$.
\item $\displaystyle A = B_{n{+}1} + \sum_{j=1}^n A(P_j^+ + P_j^-) = B_{n{+}1} +
\sum_{j=1}^n \lambda_j (P_j^+ - P_j^-) \qquad (n\in \N)$.
\end{enumerate}

Observe that this sum is an orthogonal decomposition of $A$ within
$\HS(H)$. Hence,
\[   \abs{A}_{\HS}^2 = \abs{B_n}_{\HS}^2 + \sum_{j=1}^n \lambda_j^2 (\abs{P_j^+}_{\HS}^2 + \abs{P_j^-}_{\HS}^2) \qquad (n\in \N).
\]
It follows that $\sum_{j=1}^\infty \lambda_j^2 < \infty$ and hence
$\norm{B_n} \to 0$. This, in turn,  implies 
\[ A = \sum_{j=1}^\infty \lambda_j (P_j^+ - P_j^-),
\]
the series being  convergent  within $\HS(H)$.

\medskip
Before we can turn to the module analogue, we need to examine more
closely how the projections $P_j^\pm$ can be obtained from the
respective operator $B_j$.  We only consider $P^+$, the case of $P^-$ is
analogous. 

Suppose that $T\in \BL(H)$ is a contraction.
Hence, by the mean ergodic theorem, $P^+ = P_{\fix(T)}$ is the strong
limit of the ergodic averages $C_n \coloneq \frac{1}{n} (\Id + T +
\dots + {T}^{n-1})$.  If $T$ is compact, then $C_nT= TC_n \to TP^+
= P^+$ in operator norm.

To sum up: If we let 
\[ p_n(z) \coloneq  \frac{z}{n} (1 + z + \dots + z^{n{-}1}),
\]
then $p_n(T) \to P_{\fix(T)}$ in $\BL(H)$ whenever $T\in \BL(H)$ is compact with
$\norm{T}\le 1$. Applied to our situation from above we obtain
\[   P_j^+ = \lim_n p_n(B_{j}^\sharp)\quad \text{and}\quad P_j^- = \lim_{n} p_n(-B_{j}^{\sharp}),
\]
and the convergence is in the operator norm.

\medskip

\subsection{Spectral Theorem on KH-Modules}\label{s.spt}

We now pass to \KH{}  modules and obtain the following 
analogue. Recall the definition of Hilbert--Schmidt homomorphisms
on KH-modules from \cref{s.HS}.

\begin{theorem}[Spectral Theorem]\label{spt-KH}
Let $E$ be a KH-module over a 
\Stoneanalgebra{} $\A$ and let $A\in \HS(E)$ be 
a self-adjoint Hilbert--Schmidt homomorphism on $E$. Then there is a
sequence $(\lambda_j)_j$ in $\A_+$ and  orthogonal projections $P_j^+,
P_j^{-}$ ($j\in \N$) in $\End(E)$ such that 
\[   A = \sum_{j=1}^\infty \lambda_j  (P_j^+ -P_j^{-})
\]
in $\HS(E)$. Moreover, for each $j \ge 1$ the following assertions hold.
\begin{enumerate}[(i)]
\item $\lambda_j P_j^\pm \in \HS(E)$.
\item $\ran P^+_j \perp \ran P_j^-$ and $\ran P_j^\pm \perp \ran
  P^\pm_k$  for each $k\neq j$.

\item $P_j^{\pm}A = AP_j^{\pm} = \pm \lambda_j P_j^\pm$.

\item $\lambda_j \ge \lambda_{j{+}1}$.

\item $\lambda_j >
  \lambda_{j{+}1}$ almost everywhere on $\supp(\lambda_j)$.\footnote{ This means that $\supp(\lambda_j - \lambda_{j+1}) = \supp(\lambda_j)$, cf. page  \pageref{almosteverywhereinalgebra}.}
\item $\supp(\lambda_j) = \supp(P_j^++ P_j^-)$.

\end{enumerate}
\end{theorem}

The proof is completely analogous to the Hilbert space 
case. We fix  a KH-module $E$ over the \Stoneanalgebra{}
$\A$ and  a Hilbert--Schmidt homomorphism $A\in \HS(E)$ and 
perform the following ``spectral algorithm'':

Define $B_1 \coloneq A$,  $\lambda_1 \coloneq \abs{B_1}$, and $ B_1^\sharp\coloneq
\frac{B_1}{\abs{B_1}}$ (see Section \ref{s.supp}). Next, let
\[  P_1^+ \coloneq P_{\fix(B_1^\sharp)} \quad \text{and} \quad P_1^- \coloneq
P_{\fix(-B_1^\sharp)}
\]
be the orthogonal projections onto the respective fixed spaces, see \cref{met}.
Then let 
\[ B_2 \coloneq B_1 - \lambda_1(P_1^+ - P_1^-)
\]
and iterate the procedure. The resulting sequences $(\lambda_j)_j$ and
$((P_j^+, P_j^-))_j$  are called the {\emdf spectral decomposition} of $A$.

\medskip
We shall now establish the claimed properties. First, note
that $\lambda_{n} = \abs{B_n}$ and
\[   A = B_{n} + \sum_{j=1}^{n-1} \lambda_j (P_j^+ - P_j^-)\qquad (n \in \N).
\]
Moreover, since the mean ergodic theorem also holds
for contractions on KH-modules (\cref{met}),
\[ P_j^+x = \olim_{n \to \infty} p_n(B_{j}^\sharp)x,\qquad P_j^-x = \olim_{n \to
  \infty} 
p_n(-B_{j}^\sharp)x \qquad (x\in E),
\]
where $p_n(z) = \frac{1}{n}(z + z^2 + \dots + z^{n})$. It follows
that 
$\supp(P_j^\pm) \subseteq \supp(B_{j}) = \supp(\abs{B_{j}}) = \supp(\lambda_j)$.

\begin{lemma}\label{lem:commuting}
All the operators $B_j, P_j^+, P_j^-$, $j \in \N$,  are self-adjoint 
homomorphisms
and contained in the double commutator $\{A\}''$ of $A$ in $\End(E)$.
In particular, they all commute with each other and with $A$.  
\end{lemma}

\begin{proof}
Let $T\in \End(E)$ commute with $A$. Then it commutes with $B_1$. 
Suppose $T$ commutes with $B_j, P_j^\pm$ for $j< n$. Then $T$ commutes
with $B_n$, hence with $B_n^\sharp$, hence with $P_n^\pm$. 
\end{proof}

It is obvious that $\ran P^+_j \perp \ran P_j^-$, and hence $Q_j \coloneq
P_j^+ + P_j^-$ is again an orthogonal projection for each $j \ge 1$. 
(Indeed, it is the
projection onto $\fix( (B_j^\sharp)^2)$.)  Note that 
\[   B_{j+1} = B_{j}  - \lambda_j(P_j^+ - P_j^-)= B_j(\Id - Q_j).
\]
This yields $\lambda_{j{+}1} = \abs{B_{j+1}} \le \abs{B_{j}} =
  \lambda_j$ and, by induction, 
\[  B_{n} = A \prod_{j< n} (\Id - Q_j).
\]
Since all the operators on the right-hand side commute, we obtain
\[ \ran B_{n} \perp \ran(Q_j) \quad \text{for all $j < n$.}
\]
This implies $\ran(B_n') \perp \ran(Q_j)$ and hence also $\ran(Q_n)
\perp \ran(Q_j)$ for all $j < n$. It follows that 
all the  involved projections $P_j^\pm$ are pairwise orthogonal.

Next, observe that from the representation
\[ B_n = A - \sum_{j< n} \lambda_j (P_j^+ - P_j^-)
\]
and the orthogonality of the projections it follows that
\[  \pm \lambda_n P_n^\pm = \pm B_n P_n^\pm = \pm AP_n^\pm.
\]
Since $A$ is Hilbert--Schmidt, also $\pm
\lambda_j P_j^\pm$ is Hilbert--Schmidt.  
Thus, properties (i)--(iv) are established. It remains to show
(v) and (vi) and  that $\abs{B_n}_{\HS} \to 0$ (in order).

\begin{lemma}\label{HS-proj}
Let $0 \le \lambda \in \A$ and $Q\in \End(E)$ a projection such that 
$\lambda Q\in \HS(E)$ and 
$\supp(\lambda) = \supp(Q)$. Then $\lambda \le \abs{\lambda Q}_{\HS}$.
\end{lemma}

\begin{proof}
Since $\ran(Q)$ is order-closed, we have $\supp(Q) = \supp(\ran(Q))$ by 
Lemma \ref{lem:opsupp}. Take a normalized element $e\in \ran(Q)$ of maximal
support  $\abs{e}  = \supp(Q)$ (Lemma \ref{realizesupp}). Then 
$\abs{\lambda Q}_{\HS} \ge \abs{\lambda Q e} = \lambda \abs{e} = \lambda$ as
claimed.  
\end{proof}

\medskip

Now, note that $B_n \perp \sum_{j< n} \lambda_j (P_j^+ - P_j^{-})$ in $\HS(E)$. This implies
\[ \abs{A}_{\HS}^2 = \abs{B_n}_{\HS}^2 +\sum_{j< n}\abs{\lambda_j (P_j^+ - P_j^{-})}^2_{\HS}.
\]
Since $\HS(E)$ is complete,   
$C \coloneq  \sum_{j=1}^\infty \lambda_j (P_j^+ - P_j^{-})$
exists (in order) in $\HS(E)$ and $\olim_{n \to \infty} \abs{B_n}_{\HS} = \abs{A-C}_{\HS}$. 
Furthermore, suppose that (vi) is already established. Then, 
by Lemma \ref{HS-proj}  we can conclude that 
\[ \abs{A}^2_{\HS} \ge \sum_{j=1}^\infty \abs{\lambda_j(P_j^+ - P_j^{-})}^2_{\HS} = \sum_{j=1}^\infty \abs{\lambda_j Q_j}^2_{\HS}  
\ge \sum_{j=1}^\infty \lambda_j^2.
\]
It follows that $(\lambda_j)_j$ decreases to  $0$. But this implies  
$C= A$ and hence that $\abs{B_n}_{\HS}$ decreases to $0$ as desired.

\medskip
Hence, only (v) and (vi) remain to be shown. 
  This will be
done in the next section, where we employ the ``bundle view''
to reduce the problem to the Hilbert space case.

\medskip

\subsection{The ``Bundle View''}\label{s.bundleview}

As described in Section \ref{s.hbrep} we may (and do henceforth) suppose
that $\A = \Ce(\Om)$ for some Stonean space $\Om$ and 
$E = \Gamma(H)$ for some Hilbert bundle $H$ over $\Om$.  Each
homomorphism $T \in \End(E)$ induces fiber maps  $T_\omega \in
\BL(H_\vom)$ for each $\vom \in \Om$.

We intend to show that if one starts with a self-adjoint Hilbert--Schmidt homomorphism
$A$ on $E$ and performs the ``spectral algorithm''  to it yielding the
projections
$P_j^\pm$ and the functions $\lambda_j$, then for
almost every $\vom$ the fibre  operator $A_\vom$ is Hilbert--Schmidt in
$H_\vom$ and 
\[ A_\vom = \sum_{j=1}^\infty \lambda_j(\vom)  \bigl((P_j^+)_\vom - (P_j^-)_\vom\bigr)
\]
is precisely the spectral decomposition of $A_\vom$ arising from
the ``spectral algorithm'' applied to $A_\vom$. (Recall the topological definition of \enquote{almost 
everywhere} from the remarks preceeding Wright's Lemma \ref{Wright1.1}.) From this insight we shall 
then infer the yet remaining statements to conclude the proof of 
Theorem \ref{spt-KH}.

\medskip

The following lemma will be of utmost use in understanding Hilbert--Schmidt homomorphisms
on a Hilbert module $E \cong \Gamma(H)$ in terms of a representing Hilbert bundle. The 
Hilbert bundles corresponding to KH-modules are so-called \textbf{complete}
Hilbert bundles, see \cite[Section 2.1]{Gutm1993a}. However, it is enough for 
us to know that $H$ is complete if and only if $\Gamma(H)$ is a KH-module and the reader
may take this as a definition.

\begin{lemma}\label{fiberops}
    Let $E= \Gamma(H)$ for some complete Hilbert bundle
                $H$ over a Stonean
                space $\Om$ and $A\in \End(E)$. Then the following
                assertions hold:
\begin{enumerate}[(i)]
\item $(A_\vom)^* = (A^*)_\vom$ for every $\vom \in \Om$.\\  In
  particular, if $A$ is self-adjoint, then so is each $A_\vom$.
\item If $A \in \End(E)$ is an orthogonal projection, then so is each
  $A_\vom$. 
\item $\abs{A}(\vom) = \norm{A_\vom}$ almost everywhere.
\item If $A \in \HS(E)$, then 
$A_\vom \in \HS(H_\vom)$ and $\|A_\vom\|_{\mathrm{HS}} =
|A_\vom|_{\mathrm{HS}}(\vom)$
almost everywhere.
\end{enumerate}
\end{lemma}
  
\begin{proof}
(i) and (ii) are straightforward.

\noi
(iii)\ Let $\vom \in \Om$ and $v\in H_\vom$. Pick $x\in  \Gamma(H)$ with
$x(\vom) = v$. Then
\[  \abs{A_\vom v}_\vom  = \abs{(Ax)(\vom)} = \abs{Ax}(\vom)
\le \abs{A}(\vom) \abs{x}(\vom) = \abs{A}(\vom) \abs{v}_\vom.
\]
It follows that $\norm{A_\vom} \le \abs{A}(\vom)$ for every $\vom \in
\Om$. On the other hand, as $\abs{A} = \sup_{\abs{x}\le 1} \abs{Ax}$,
by Wright's Lemma \ref{Wright1.1} one has 
\[ \abs{A}(\vom) = \sup_{\abs{x}\le 1} \abs{Ax}(\vom) = 
\sup_{\abs{x}\le 1} \abs{A_\vom x(\vom)} \le \norm{A_\vom}
\qquad 
\text{almost everywhere}.
\]
(iv)\  Let $\mathscr{F}$ be the family of all finite suborthonormal subsets of
$E$.  By \cref{Wright1.1} 
we find a residual subset $D \subset \Om$ such that for all $\vom \in D$
    \begin{align*}
      |A|_{\mathrm{HS}}^2(\vom) = \sup \Bigl\{\sum_{x \in F} |Ax|^2(\vom)\,\Big|\, F \in \mathscr{F}\Bigr\}.
    \end{align*}
Hence, for each $F\in \mathscr{F}$ and $\vom \in D$
\begin{align*}
      \sum_{x \in F} |Ax|^2(\vom) \leq \|A_\vom\|_{\mathrm{HS}}^2,
    \end{align*}      
since $\{x(\vom)|\mid x \in F\}\setminus\{0\}$ is an orthonormal subset of
$H_\vom$. It follows that 
    \begin{align*}
      |A|_{\mathrm{HS}}^2(\vom) = \sup
                  \Bigl\{\sum_{x \in F} |Ax|^2(\vom)
\, \Big|\,  F \in \mathscr{F}\Bigr\} \leq \|A_\vom\|^2_{\mathrm{HS}}
\qquad \text{for every $\vom \in D$.}
    \end{align*}
Conversely, fix $\vom \in D$ and let $v_1, \dots, v_n \in H_\vom$ be
orthonormal.  Then there are sections
$x_1, \dots, x_n$ such that $x_i(\vom) = v_i$ for $i = 1,\dots,
n$ (see Section \ref{s.hbrep}). By continuity,  the Gramian matrix
$\bigl((x_i(\cdot)|x_j(\cdot))\bigr)_{i,j}$ is invertible in a neighborhood of
$\vom$. Hence,  one can apply the Gram--Schmidt procedure and suppose without
loss of generality that $x_1, \dots, x_n$ is a suborthonormal system. Then
\vanish{We then find local orthonormal sections $x_1, 
\dots, x_n \in \Gamma(O; H)$ defined on a clopen neighborhood $O \subset \Om$ of $\vom$ such that $x_i(\vom) = v_i$ for $i=1, 
\dots, n$. Identifying $x_i$ with the canonical continuous extension by zero to the whole space $\Om$ for every $i=1,\dots, n$,
one has}
  \begin{align*}
    \sum_{i=1}^n \|A_\vom v_i\|^2 = \sum_{i=1}^n \|(Ax_i)(\vom)\|^2 \leq |A|_{\mathrm{HS}}(\vom)^2.
  \end{align*} 
  Therefore, $A_\vom\in\HS(H_\vom)$ and $\|A_\vom\|_{\mathrm{HS}} \leq |A|_{\mathrm{HS}}(\vom)$.
\end{proof}

Now let us  start again with the Hilbert--Schmidt operator $A \in \HS(E)$
and see what the algorithm yields. By (iii) of the preceding lemma,
with $\lambda \coloneq \abs{A}$ we have 
\[ \lambda(\vom) = \norm{A_\vom}\quad \text{and}\quad 
  (A^\sharp)_\vom = \Bigl(\frac{A}{\abs{A}}\Bigr)_\vom=  (A_\vom)^\sharp
\]
(almost everywhere). 
Although $A^\sharp$ may not be a Hilbert--Schmidt
operator, $(A_\vom)^\sharp$ is Hilbert--Schmidt (a.e.). Then $p_n(A^\sharp)_\vom =
p_n(A_\vom^\sharp) \to P_{\fix(A^\sharp_\vom)}$ in operator norm
(a.e.). By (iii) of Lemma \ref{fiberops} again and by Lemma \ref{charorder}, $p_n(A^\sharp)$ is convergent in
$\End(E)$, and since the limit must be $P^+$, it follows that 
\[      (P^+)_\vom = P_{\fix(A^\sharp_\vom)} \qquad \text{(a.e.)}.
\]
Likewise,  $(P^-)_\vom = P_{\fix(-A^\sharp_\vom)}$ and $Q_\vom =
P_{\fix((A^\sharp)^2_\vom)}$ (a.e.). 

Since the procedure only involves countably many steps and
countable intersections of residual sets in $\Om$ are still residual, 
we conclude what we aimed at: There is a residual set $D$
in $\Om$ such that for each $\vom \in D$ one has
\[ \lambda_n(\vom) = \abs{(B_n)_\vom},\quad 
(B_n)_\vom = A_\vom - \sum_{j<n} \lambda_j(\vom)  \bigl((P_j^+)_\vom - (P_j^-)_\vom\bigr)
\]
for all $n\in \N$ and 
\[   A_\vom = \sum_{j=1}^\infty \lambda_j(\vom)  \bigl((P_j^+)_\vom - (P_j^-)_\vom\bigr)
\]
is the spectral decomposition of the Hilbert--Schmidt operator $A_\vom$.

\medskip
By what we have seen in Section \ref{s.sptrev}, $\lambda_j(\vom) >
\lambda_{j{+}1}(\vom)$ whenever $\lambda_j(\vom) \neq 0$ and $\vom \in D$. 
But this just means $\lambda_j > \lambda_{j{+}1}$ almost everywhere 
on $\supp(\lambda_j)$, i.e., assertion (v) of Theorem \ref{spt-KH}.
Finally, again by what we have seen in Section \ref{s.sptrev}, 
if $\vom \in D$ and  $\lambda_j(\vom)  = \norm{(B_n)_\vom} > 0$, then
$(Q_n)_\vom \neq 0$. And this is precisely assertion (vi) of Theorem \ref{spt-KH}.

\medskip
\noi
With this observation, the proof of Theorem \ref{spt-KH} is complete.\qed

\begin{remark}\label{spt-cycp}
Theorem \ref{spt-KH} can be recognized to be a special case of the spectral theorem
for self-adjoint {\em cyclically compact} homomorphisms as proved by Gönüllü in
\cite{Gonu2016}. In fact, it follows from \cite[Thm. 2.2,
(iii)$\Rightarrow$(i)]{Gonu2014} that 
Hilbert--Schmidt homomorphisms (in our sense) are cyclically compact.
(Gönüllü himself operates with a different definition of a
Hilbert--Schmidt homomorphism that presupposes cyclical
compactness and rests on the spectral theorem, cf. \cite[Def. 3.1]{Gonu2014}.) 

Our approach to Theorem \ref{spt-KH} has the advantage that it
completely avoids the notion of cyclical compactness. 
Moreover, having made explicit the steps of the ``spectral algorithm'' in the proof
is useful for proving an ``equivariant'' version  in Part II (Proposition \ref{spt-cov}). 

Finally, it seems worthwhile noting  that our proof can easily be modified 
to cover all self-adjoint cyclically compact homomorphisms. Indeed,
one only needs to replace $\abs{\cdot}_{\HS}$ by the operator lattice-norm
and to employ that if $A \in \End(E)$ is cyclically compact then
almost every fiber operator $A_\vom$ is compact. (The latter 
follows from the fact that each cyclically
compact homomorphism  is the order-limit of a net of finite-rank homomorphisms, 
see \cite[Thm.\,8.5.6]{Kusr2000}). As a consequence we find that 
a homomorphism $A\in \End(E)$ is cyclically compact {\em if and only if}
almost every fiber operator $A_\vom$ is compact---a quite satisfying
result.
\end{remark}

\section*{Notes and Comments}\label{s.notesI}

Stonean (sometimes: Stonian) spaces are as old as Stone's famous representation
theorem from \cite{Ston1937}. Their central role in the representation theory
of vector lattices was recognized soon after, e.g. in representation theorems 
by Kakutani \cite{Kaku1941L,Kaku1941M}. See also Schaefer's monograph \cite[Chap.~2]{Schaefer1974} and the
more recent book \cite{GroevRoo2016}. 

In \cite{Kapl1951}, Kaplansky introduces the concept of an (in general
non-commutative) AW$^*$-algebra. For commutative algebras, this means that 
the Boolean algebra of idempotents is complete and generates it as a
C$^*$-algebra. Representing $\A$  as $\Ce(\Om)$ for some compact space it is
routine to check that $\A$ is a AW$^*$-algebra in this sense if and only if $\Om$ is
Stonean, i.e., $\A$ is a Stone algebra.   (As far as we know, the term
``Stone algebra'' was introduced by Wright in \cite{Wrig1969}. It 
is also used by Kusraev in \cite{Kusr2000}.)

\smallskip
In \cite{Kapl1953} Kaplansky introduced the concepts of Hilbert C$^*$-modules and
Hilbert AW$^*$-modules, nowadays   also   called Kaplansky--Hilbert modules, and proved
several analogues of Hilbert space results for them.  
In \cite{DeckPear1963} Deckard and Pearcy observed that matrices with entries in a Stone
algebra $\A$ can be $\A$-unitarily triangularized. This and its sequel
\cite{DeckPear1964} 
appear to be the first 
results on the spectral theory of homomorphisms on KH-modules. Wright
continued this work in \cite{Wrig1969KHM} and coined their modern name.  

Kusraev and Kutadeladze incorporated the theory of KH-modules into their systematic 
investigation of ordered functional-analytic structures, in particular of
{\em lattice-normed spaces},  and 
operators thereon. Chapter 7 of \cite{Kusr2000}  contains a brief introduction to
Kaplansky--Hilbert modules. The definition there is ours (with  ``{\em
  bo}-convergence'' being  our ``order-convergence''). 

\smallskip

During the 1970s, the known results about Kaplansky--Hilbert modules were
recognized as being embeddable into a theory called ``Boolean-valued analysis'' (BVA),
which is the application of Boolean-valued models of set theory to analysis.
Originally conceived by 
Takeuti \cite{Take1978,Take1979b,Take1979a} it was continued by 
Ozawa \cite{Ozaw1983,Ozaw1990} and, among others, by 
Kusraev and Kutadeladze, see \cite[Chap.~8]{Kusr2000} and \cite{KusrKuta1999}.  

The upshot of this theory, which has a major model-theoretic thrust, is that 
Kaplansky--Hilbert modules are simply the Hilbert spaces in a Boolean-valued
universe. As the Boolean-valued universe is a model of ZFC, Hilbert space
results can be re-interpreted in this new model and yield in this way valid
statements about KH-modules (``transfer principle'').

Another, but essentially equivalent approach, is more recent and comes
under the name of conditional analysis or conditional set theory, see, e.g.,  
\cite{DrapJamnKarlKupp2016,CherKuppVoge2015,FiliKuppVoge2009}.
It is model-theory free but works with slightly different objects, essentially
completions with respect to ``mixing''.

\smallskip
We acknowledge these approaches, but we also feel that for most mathematicians
(including ourselves), the model-theoretic path remains 
arcane and---in our topological set-up---the conditional
world approach would require working with rather unfamiliar objects. 
Unfortunately, nowhere in the literature could we find a coherent account of the
KH-module theory that would collect all the facts we needed and  with  minimal  
background requirements. Therefore, we decided to give such an account here.

In particular, we present a new proof of the spectral theorem 
for self-adjoint Hilbert--Schmidt homomorphisms on KH-modules. This
theorem can  also be derived from  the spectral theorem
for self-adjoint {\em cyclically compact} operators, cf. Remark \ref{spt-cycp}. 
Such operators were
introduced  by Kusraev in 1983, see \cite{Kusr1983}, and are called
``stably compact'' operators in conditional analysis,  see \cite{JamnZapa2017pre}.
In \cite{Kusr2000b},  a singular-value representation was established by BVA-techniques.  
An explicit formulation and a model
theory-free proof of the spectral theorem for self-adjoint cyclically compact
operators were given 
by Gönüllü \cite{Gonu2016}. However, cyclical compactness is a technically involved
concept which fully reveals its naturality only
in the light of the  Boolean/conditional version of the notions of
``natural number'' and ``sequence''. Our alternative proof of the spectral
theorem avoids cyclical compactness altogether and is instead based on 
the representation of a KH-module as the space of continuous sections
of a 
Hilbert bundle.

\smallskip
The duality between modules and bundles goes back to the famous result of Serre--Swan, see 
\cite{Swan1962}. Godement (in \cite{Gode1951}) as well as   Douady and Dixmier (in \cite{DixmDoua1963}) seem to be among the first to systematically study continuous Hilbert bundles, although similar ideas date back further. Such bundles then became increasingly
important in representation theory during the 1970s.  
In \cite{HoKe1977}, Hofmann and Keimel summarize the connections and 
equivalences between the  different approaches to ``relative functional
analysis''  based on  Banach bundles, modules, and sheaves, respectively. The
monographs by Gierz \cite{Gierz1982} and Dupré--Gillette \cite{DuGi1983} appear to be the
best references on Banach bundles and their connection with modules, the latter
stressing the categorial equivalence of both worlds. 
Gutman \cite{Gutm1993a} investigates this correspondence in the context of lattice-normed
spaces and, in particular, characterizes those Banach 
bundles that correspond to Kaplansky--Banach modules \cite[Thm.~2.1.1]{Gutm1993a}.
Proposition \ref{equiv} for KH-modules has been obtained by Takemoto in
\cite{Takemoto1973}.

It turns out that when dealing with KH-modules it is often easier to prove things
directly than appealing to a representation as sections in a bundle. We
use the latter only at one point in the proof of the spectral theorem (Section \ref{s.bundleview}).

Besides topological Hilbert bundles, there are of course also measurable
Hilbert bundles and a connection of modules with spaces of measurable sections. 
Such Hilbert bundles were used, e.g. by Zimmer \cite{Zimm1976} and Glasner \cite{Glas}
in the context of the Furstenberg--Zimmer theory. They also appear, as direct
integrals of Hilbert spaces,  in the representation theory of von Neumann algebras
\cite[Chap.~VI.8]{Take2002}.  Typically, this happens under the separability restriction
and for Borel measure spaces only.  By working in a topological setting, we avoid this restriction.

\part{Covariant Unitary Group Representations on KH-Modules}\label{p.DEC}

In this part, we define (covariant) unitary group representations
on Kaplansky--Hilbert modules and then prove a version of the 
Jacobs--de Leeuw--Glicksberg decomposition for such representations.
This is done by using the spectral theorem to establish a correspondence 
between finitely-generated 
invariant submodules and intertwining Hilbert--Schmidt homomorphisms.

\smallskip
As stated in the introduction,  $G$ is an arbitrary, but fixed group in this (and the next) chapter.

\section{Covariant Group Representations on Hilbert Modules}\label{c.covrepKH}

\medskip

\subsection{KH-Dynamical Systems}

 Let $\A$ a fixed commutative unital $\Ce^*$-algebra.
A {\emdf $G$-representation on $\A$} is a unital representation
\[ 
  S\colon G \to \Aut(\A),\qquad t \mapsto S_t
\]
of $G$ as unital $*$-automorphisms of $\A$ (as a $\Ce^*$-algebra). 
It then follows that
\[  
  \abs{S_tf} = S_t\abs{f} \qquad \text{for each $f\in \A$ and $t\in G$},
\]
i.e., each $S_t$ is a lattice homomorphism. Indeed, if we represent $\A= \Ce(\Om)$
for some compact space $\Om$, then there is a unique group homomorphism
\[ 
  \vphi\colon G \to \Homeo(\Om), \qquad t \mapsto \vphi_t
\]
such that $S_tf =f\nach \vphi_{t}^{-1}$, see \cite[Thm.~4.13]{EFHN2015}. It
follows that $S_t$ commutes with all pointwise defined operations,
like taking square roots for instance.

\medskip
Let $S\colon G \to \Aut(\A)$ be a $G$-representation on $\A$, and let $E$ be a Hilbert module
over $\A$. We want to consider $S$-covariant unital representations of $G$ on
$E$. For this, we need the following concept.

\begin{definition}
Let $E$ be a pre-Hilbert module over the commutative unital $\Ce^*$-algebra $\A$,
and let $S\in \Aut(\A)$. A bounded linear operator $T\in \BL(E)$ is called
an {\emdf $S$-homomorphism} if 
\[        T(f x) = Sf \cdot Tx\qquad (x\in E,\, f\in \A).
\]
An $S$-homomorphism $T$ is called an {\emdf $S$-isometry} if
\[     
    (Tx|Ty) = S(x|y) \quad \textrm{for all } x,y \in E,
\]
and it is called {\emdf $S$-unitary} if it is $S$-isometric and invertible.
\end{definition}

\begin{remark}\label{regularmod}
An $S$-homomorphism $T \in \mathscr{L}(E)$ is simply a
        homomorphism from $E$ to $E_S$, where $E_S$ is the pre-Hilbert
        module that arises from $E$ by changing the module
        multiplication to
    \begin{align*}
      \A \times E \to E, \quad (f,x) \mapsto 
f\cdot_S x \coloneq  Sf \cdot x
    \end{align*}
and the $\A$-valued inner product to 
\[   E\times E \to \A,\qquad (x,y) \mapsto (x|y)_S \coloneq S^{-1} (x|y).
\]
Likewise, an $S$-isometry $T$ is an isometry from $E$ to $E_S$.          
This change of perspective allows to transfer statements about
bounded $\A$-homo\-morphisms between pre-Hilbert modules to statements about $S$-homomorphisms.
\end{remark}

We are now able to define covariant representations.

\begin{definition}\label{def:covrep}
Let $S\colon G \to \Aut(\A)$ be a $G$-representation on $\A$, and let $E$ be a Hilbert module
over $\A$. Then a {\emdf (covariant) unitary $G$-representation} on $E$ over $S$ is a
mapping  
\[   T\colon G \to \BL(E), \qquad t \mapsto T_t
\]
with the following properties.
\begin{enumerate}[(i)]
\item $T_e = \Id$ and $T_{ts} = T_t T_s$ for all $s,t\in G$ \quad (group homomorphism).
\item $T_t$ is an $S_t$-unitary $S_t$-homomorphism for each $t\in G$\quad ($S$-covariance).
\end{enumerate}
\end{definition}

If $T$ is an $S$-covariant representation on $E$, then 
the tuple $(\A,S; E;T)$ is called a {\emdf $G$-system}, for short. 
 If, in addition, $\A$ is a \Stoneanalgebra{} and $E$ is a
Kaplansky--Hilbert module over $\A$, then the $G$-system is called  
a {\emdf \KH{} $G$-system}. And if $G$ is understood,  a 
\KH{}  $G$-system  is simply called a {\emdf KH-dynamical system}.

\begin{remarks}\label{remstoreps}
\begin{enumerate}[(1)]
\item If $S\colon G \to \Aut(\A)$ is a $G$-representation on $\A$, then the triple
$(\A,G, S)$ is commonly called a {\em $\Ce^*$-dynamical system}, see 
\cite[7.4.1]{Pedersen}. 
This is one reason for using the term ``KH-dynamical system''.
The second is the  application
to extensions of $G$-dynamical systems, see Chapter \ref{c.dicmps} below
and confer also \cite[Def's  3.5 and 3.12]{KrSi2020}.

\item We recall from \cite[7.4.8]{Pedersen} that an 
{\em $S$-covariant unital representation} of a $\Ce^*$-dynamical system $(\A,
G,S)$  on a Hilbert space $H$ is
a unitary representation $T$ of $G$ on $H$ together with a unital
$*$-representation
$\pi\colon \A \to \BL(H)$ 
such that 
\[    \pi(S_t f) T_t  = T_t \pi(f) \qquad (t\in G,\, f\in \A).
\]
Then $H$ is an $\A$-module via $f \cdot x \coloneq \pi(f)x$. Using module
notation, the identity  above becomes
\[     S_tf \cdot T_t x = T_t(f\cdot x)\qquad (t\in G, f\in\A, x\in H),
\]
i.e., $T_t$ is an $S_t$-homomorphism. It is trivially $S_t$-isometric, as the
inner product maps into $\C \car\subseteq \A$.
Hence our notion of covariant representation generalizes
the classical one from Hilbert spaces to Hilbert modules. 
\end{enumerate}
\end{remarks}

Let us list some examples. 

\begin{examples}\label{simpleex}\mbox{}
\begin{enumerate}[(1)]

\item\label{simpleex1}
  Let $\phi \colon \Om \to \Om$ be a homeomorphism on a compact space $\Om$, $H$ a Hilbert space and $\Phi \in \mathscr{L}(H)$. 
  Then the Koopman operator $S_\phi \in \mathscr{L}(\uC(\Om))$ given by $S_\phi f \coloneqq f \circ \phi^{-1}$ for every $f 
  \in 
  \uC(\Om)$ is a  *-automorphism and
  \begin{align*}
    \uC(\Om;H)\rightarrow \uC(\Om;H),\quad f \mapsto \Phi \circ f \circ \phi^{-1}
  \end{align*}
  is an $S_\phi$-homomorphism.

\item\label{simpleex2} 
More generally, suppose that  $\phi\colon G \to \mathrm{Homeo}(\Om)$ is a representation
of a  group $G$ as homeomorphisms of $\Om$, and $\Phi\colon G \to \BL(H)$ is a unitary
representation of $G$ on $H$. Then 
\[ S\colon G \to \Aut(\A), \qquad S_tf \defeq f\circ \phi_t^{-1}   = f\circ \phi_{t^{-1}} 
\]
is a representation of $G$ on $\A$ as $*$-automorphisms, and with 
\[    T\colon G \to \BL( \Ce(\Om;H)), \qquad T_tf \defeq  \Phi_t \circ f \circ \vphi_t^{-1}
\]
we obtain a $G$-system $(\Ce(\Om), S; \Ce(\Om;H),T)$. 

\item If $\Om$ is a Stonean space, then from (2) one obtains a KH-dynamical system by passing to the completion
(see below). Other examples of KH-dynamical systems will be studied in
Part III,  see in particular Section \ref{s.akhds}.
\end{enumerate}
\end{examples}

\medskip

\subsection{Standard Constructions}\label{stacon}

Suppose that $S \in \mathscr{L}(\A)$ is a $*$-automorphism of a \Stoneanalgebra{} $\A$
 and $T \in \mathscr{L}(E)$   an $S$-unitary on a   pre-Hilbert module
 $E$ over $\A$. Then there is a unique $S$-unitary operator 
$T^\sim \in \mathscr{L}(E^\sim)$ on the order-completion $E^\sim$ 
with $T^\sim|_{E} = T$. 

More generally, let  $(\A,S; E,T)$ be a $G$-system and $\A$ be a
\Stoneanalgebra. Then by extending each operator $T_t$ to
the order completion $E^\sim$ of $E$ we obtain a KH-dynamical system
  $(\A,S; E^\sim,T)$, called the {\emdf (order-)completion}  of the original
  system.

\begin{example}
In the situation of \cref{simpleex}\ref{simpleex2}, suppose that $\Om$ is a Stonean
  space. Then the $G$-system  $(\Ce(\Om), S; \Ce(\Om;H),T)$  determines a 
  KH-dynamical system 
\[ (\Ce(\Om), S; \Ce_{\#}(\Om;H),T),
\]
(cf.\ \cref{comp-COm}) where $T_t$ is again
  defined as  $T_tf \defeq  \Phi_t \circ f \circ
  \vphi_t^{-1}$ (on  representatives $f$). 
\end{example}

Let $(\A,S;E,T)$ be a KH-dynamical system.
A subset $F \subset E$ is {\emdf $T$-invariant} if $T_t(F) 
\subset F$ for every $t \in G$. If $F$ is a $T$-invariant
KH-submodule of $E$, then $(\A,S;F,T|_F)$ with
    \begin{align*}
      T|_F \colon G \to \mathscr{L}(F), \quad t \mapsto T_t|_F
    \end{align*}
is a KH-dynamical system. It is called a {\emdf subsystem} of $(\A,S;E,T)$.

Let $(\A,S;E,T)$ be a KH-dynamical system. Then 
the system $(\A,S;E^*,\overline{T})$ with 
 \[  \overline{T}_t \overline{x} \coloneqq \overline{T_t x}\qquad (\overline{x} \in
  E^*,\, t\in G)
\]
defined using the Riesz--Fr\'echet identification from \cref{riesz} is called the 
{\emdf dual system}.\footnote{Also the term {\em contragredient system} is used.}

\medskip

Let $(\A,S;E,T)$ and $(\A,S;F,R)$ 
be KH-dynamical systems.
Then their  {\emdf tensor product} is $(\A,S;E \otimes F, T \otimes R)$, where
\begin{align*}
  T \otimes R \colon G \to \mathscr{L}(E \otimes F), \qquad t \mapsto T_t \otimes R_t,
\end{align*}
with $T_t \otimes R_t$ being the unique extension (\cref{extension}) of the linear operator on $E \otimes_{\mathrm{alg}} F$ given by
\begin{align*}
(T_t \otimes R_t)(x \otimes y) = T_t x \otimes R_ty
\end{align*}
for $x \in E, y \in F$ and $t \in G$.

\vanish{
\medskip
Finally, let $(\A,S;E;T)$ a KH-dynamical system and let $\lambda \in
\A$ be an unimodular element, i.e., $\abs{\lambda} = \car$.
Then we obtain the {\emdf $\lambda$-twisted system} $(\A,S;E;T^\lambda)$
by 
\[     T^\lambda_t x \coloneq \lambda (S_t\konj{\lambda}) T_t x\qquad
(t\in G,\, x\in E).
\] 
It is routine to show that this is indeed a $G$-system. 
}

\section{The Decomposition Theorem}\label{c.mainthm}

In this chapter, we prove the main decomposition theorem
for $G$-systems. Central to the proof is a
re-examination of  the spectral theorem  (Theorem
\ref{spt-KH}) for  Hilbert--Schmidt homomorphisms that {\em
  intertwine} the dynamics.

Before we go in medias res, let us state some properties of
$G$-systems.

\begin{lemma}\label{covrep-basic}
Let $(\A, S; E,T)$ be a \KH{} $G$-system.  Then the following
assertions hold:
\begin{enumerate}[(i)]
\item $S_t\abs{x}^2 = \abs{T_tx}^2$ and $S_t\abs{x} = \abs{T_tx}$ for
  all $t\in G$ and $x\in E$.
\item If $x\in E$ is normalized, then so is $T_tx$ for all $t\in G$.
\item Each $T_t$ maps the unit ball $\{ x\in E \,|\, \abs{x}\le \car\}$
bijectively onto itself. 

\item $T_t M^\perp = (T_t M)^\perp$ for all $t\in G$ and $M
  \subseteq E$.

\item If $\calB$ is a suborthonormal subset (basis) of $E$, then
  so is $T_t\calB$ for each $t\in G$.

\item If $M\subseteq E$ is $T$-invariant, then so are  $M^\perp$ and 
  $\ocl\spann_\A(M)$. 
\end{enumerate}
\end{lemma}

\begin{proof}
The first assertion of (i) is clear and the second follows from the
first since $S_t \abs{x}^2 = (S_t \abs{x})^2$.
(ii) follows from (i) as $S_t$ must map idempotents to 
idempotents.  (iii) follows since, by (i), $\abs{T_tx} = S_t\abs{x}\le S_t \car = \car$ 
whenever $\abs{x}\le \car$.  (iv) is obvious
and (v) follows from (ii) and (iii). 
Finally, (v) follows from (iv) and \cref{order-closure}.
\end{proof}

\medskip

\subsection{Equivariant Spectral Theory}\label{s.est}

Let  $(\A, S; E,T)$ be a \KH{} $G$-system. The {\emdf fixed algebra}
\[  \fix(S) \coloneq \fix_\A(S) \coloneq  \bigcap_{t\in G} \fix(S_t) = \{ f\in \A\,|\, S_t f =
f\,\text{for all $t\in G$}\}
\]
is an order-closed unital $*$-subalgebra of $\A$, and hence a Stone
algebra in its own right. Its complete Boolean algebra of idempotents is
denoted by $\B_S$. The suprema (infima) in $\fix(S)$ and $\A$ of a
bounded subset of $\fix(S)$ coincide (by \cite[Thm.~7.23]{EFHN2015}
and order-closedness).
 
On the other hand, the {\emdf  fixed module}
\[  \fix(T) \coloneq \fix_E(T)\coloneq  \bigcap_{t\in G} \fix(T_t) = \{ x\in E\,|\, T_t x =
x\,\text{for all $t\in G$}\}
\]
is an order-closed $\fix(S)$-submodule of $E$. Even more, it is a
KH-module over $\fix(S)$ with the induced inner product. 
(Indeed, if $x,y\in \fix(T)$, then $S_t(x|y) = (T_tx|T_ty) = (x|y)$ and
hence $(x|y)\in \fix(S)$.) The notions of order-convergence in
$\fix(T)$ and in $E$ coincide.

\vanish{
\begin{remark}\label{fix-twisted}
Let  $(\A, S; E,T)$ be a \KH{} $G$-system, $\lambda \in \A$
unimodular, and $(\A, S; E,T^\lambda)$ the associated
$\lambda$-twisted system (cf.\ Section \ref{stacon}). Then
\[    \fix(T^\lambda) = \lambda \fix(T)
\]
as is easily shown. 
\end{remark}
}

\begin{definition}
Let $(\A, S; E,T)$ be a \KH{} $G$-system. A homomorphism
$A\in \End(E)$ is {\emdf $T$-intertwining} if  $T_t A = AT_t$ for all
$t\in G$. 
\end{definition}

The set of $T$-intertwining homomorphisms is denoted by $\End_T(E)$. 
It is a (even strongly)  order-closed $*$-subalgebra\footnote{The $*$-invariance
  requires a moment's thought.}  and a $\fix(S)$-submodule of
$\End(E)$.  

Actually, it
is  a Kaplansky--Banach module over $\fix(S)$ and the notions of order-convergence
in $\End_T(E)$ and in $\End(E)$ coincide. This is a consequence of 
the following lemma.

\begin{lemma}\label{lem:abs-fix}
Let $(\A, S; E,T)$ be a \KH{} $G$-system and $A\in \End_T(E)$. Then 
$\abs{A} \in \fix(S)$.
\end{lemma}

\begin{proof}
Let $A \in \End_T(E)$ and $t\in G$. Since, by Lemma \ref{covrep-basic},
$T_t$ maps the (order) unit ball $\{ x\in E\, |\, \abs{x}
\le \car\}$ bijectively onto itself, 
\begin{align*}   S_t\abs{A} &= S_t \sup\{ \abs{Ax} \, |\, \abs{x}\le \car\}
= \sup\{ S_t\abs{Ax} \, |\, \abs{x}\le \car\}
\\ & = \sup\{ \abs{AT_tx} \, |\, \abs{x}\le \car\}
= \sup\{ \abs{Ay} \, |\, \abs{y}\le \car\} = \abs{A}
\end{align*}
and hence $\abs{A}\in \fix(S)$. 
\end{proof}

Here is another interpretation of $\End_T(E)$. It is easily checked
that $T_t A T_{t^{-1}}$ is a bounded homomorphism on $E$ whenever $A$
is. Hence we obtain an {\emdf implemented} (covariant) representation
\[ \calT\colon G \to \End(\End(E)), \qquad \calT_t(A) \coloneq T_t A
T_{t^{-1}} \qquad (t\in G). 
\]
Clearly, $\End_T(E) = \fix(\calT)$. 

\medskip 
Recall that $\HS(E)$ is a KH-module in its own right. 
The implemented representation restricts to a
covariant representation on $\HS(E)$, as the following lemma shows. Its proof is straightforward.

\begin{lemma}
Let $(\A, S; E,T)$ be a \KH{} $G$-system and let $t\in G$. Then
$\calT_t(A) = T_t A T_{t^{-1}} \in \HS(E)$  and 
$(\calT_t(A)| \calT_t(B))_{\HS} = S_t (A|B)_{\HS}$ 
for all $A, B \in \HS(E)$. 
\end{lemma}

Note that the fixed space (in $\HS(E)$) of this implemented representation
is simply
\[  \fix_{\HS(E)}(\calT) = \HS_T(E) \coloneq \End_T(E) \cap \HS(E).
\]

We now look at the ``equivariant version'' of the spectral theorem.

\begin{proposition}\label{spt-cov}
Let $(\A, S; E,T)$ be a \KH{} $G$-system and let $A\in \HS_T(E)$
be self-adjoint. Furthermore, let
\[  A= \sum_{j=1}^\infty \lambda_j (P_j^+ - P_j^-)
\]
be the spectral decomposition of $A$ (resulting from the ``spectral algorithm''
applied to $A$).  Then $\lambda_j \in \fix(S)$ and  
$P_j^\pm \in \End_T(E)$ for  each $j \in \N$. 
\end{proposition}

\begin{proof}
We use the notation of \cref{s.spt}. Since $B_1 = A \in \End_T(E)$ we obtain $\lambda_1 =  |B_1| \in \fix(S)$ by \cref{lem:abs-fix}. But then also $B_1^\sharp = \frac{B_1}{\abs{B_1}} \in  \End_T(E)$ since $\frac{B_1}{\abs{B_1} + \veps} \in  \End_T(E)$ for each $\veps > 0 $ and $\End_T(E)$ is order-closed in $\End(E)$. This shows that $\fix(B_1^\sharp)$ and $\fix(-B_1^\sharp)$ are $T$-invariant, and hence also the orthogonal complements  $\fix(B_1^\sharp)^\perp$ and $\fix(-B_1^\sharp)^\perp$ are $T$-invariant, see \cref{covrep-basic} (vi). This shows $P_1^+, P_1^- \in \End_T(E)$. But then also $B_2 = A - \lambda_1 (P_1^+ - P_1^-) \in \End_T(E)$, and the claim follows by induction.
\end{proof}

\medskip

\subsection{Main Theorem}

Let $(\A,S;E,T)$ be a KH-dynamical system. From this we build the
dual system $(\A,S;E^*,\konj{T})$ and then the tensor product system
$(\A,S;E\tensor E^*,T \tensor \konj{T})$ as described in Section \ref{s.ocp}. 

The following lemma relates the tensor product system with the
implemented system on $\HS(E)$.

\begin{lemma}\label{invariance}
Let $(\A,S;E,T)$ be a KH-dynamical system. 
Then the  tensor product system $(\A,S;E\tensor E^*,T \tensor \konj{T})$ and the
implemented system  $(\A,S; \HS(E), \calT)$ are isomorphic via the unitary
isomorphism 
\[ V\colon E \tensor E^* \to \HS(E),\qquad V(y \tensor \konj{z}) = ( x \mapsto (x|z)y)
\]
described in Proposition \ref{tensordescription}. Via this isomorphism,\quad 
$\fix(T\tensor \konj{T}) \cong \HS_T(E)$.
\end{lemma}

\begin{proof}
For $t\in G$ and  $x, y,z\in E$ one has
\begin{align*}    T_tV(y \tensor \konj{z})x & = T_t( (x|z)y ) = S_t(x|z)\cdot T_ty
=  (T_tx | T_t z) T_t y = V( T_t y \tensor \konj{T_t z}) T_t x
\\ & = V( (T \tensor \konj{T})(y \tensor \konj{z})) \, T_t x.
\end{align*}
This shows $T_tV(u)= V((T_t \tensor \konj{T_t})u) T_t$ for $u = y \tensor \konj{z}$.  
The claim then follows by order-density of $E \otimes_{\mathrm{alg}} E^*$ in
$E \otimes E^*$. 
\end{proof}

The next lemma tells that finitely-generated invariant KH-submodules 
are in one-to-one correspondence with $T$-intertwining orthogonal
projections of finite rank.

\begin{lemma}\label{lem:proj-invar}
Let $(\A,S;E,T)$ be a KH-dynamical system, $M = \ran(P) \subseteq E$ a
KH-submodule with associated orthogonal projection $P$. Then 
$M$ is $T$-invariant if and only if $P$ is $T$-intertwining, and 
$M$ is of finite rank if and only if $P$ is a finite-rank homomorphism if and only if $P \in \HS(E)$.
\end{lemma}

\begin{proof}
Apply Lemma \ref{covrep-basic} for the first assertion. The first part of the
second is obvious. For the final equivalence, suppose that $P \in\HS(E)$. Then,
for every finite suborthonormal system $\calB \subseteq M$,
$\sum_{e\in \calB} \abs{e}^2 = \sum_{e\in \calB} \abs{Pe}^2 \le
\abs{P}_{\HS}^2$. By Lemma \ref{lem:finrank}, $M$ is of finite rank.
\end{proof}

Recall from \cref{fourier} that if $M$ is a finite-rank KH-submodule of $E$
and $\calB= \{e_1, \dots, e_n\}$ is a suborthonormal basis for $M$, then
the  associated orthogonal projection is
\[  P = Vu_\calB,\quad \text{where}\quad  u_\calB \coloneq \sum_{j=1}^n e_j \tensor \konj{e_j}.
\]
By the previous two lemmas,
$M$ is $T$-invariant if and only if $u_\calB \in \fix(T\tensor \konj{T})$.  
 
Recall from Section \ref{s.finrank} that a suborthonormal system $e_1,\dots, e_n$ is {\em
  homogeneous} if $\abs{e_1} = \dots = \abs{e_n}$; 
and KH-submodule $M$ of finite rank is
called {\em homogeneous} if it has a homogeneous suborthonormal
basis.  Finally, recall that a KH-submodule $M$ of finite rank has a 
canonical decomposition into homogeneous KH-submodules.

\begin{lemma}\label{hom-invar}
Let $(\A,S;E,T)$ be a KH-dynamical system and $M\subseteq E$ a
$T$-invariant KH-submodule of finite maximal rank $N \in \N$.
\begin{enumerate}[(i)]
\item If $M$ is homogeneous with homogeneous suborthonormal basis
  $e_1, \dots, e_N$. Then $\abs{e_1} = \dots = \abs{e_N} = \supp(M)
  \in \fix(S)$.
\item Let $M= q_0M \oplus q_1M \oplus \dots \oplus  q_NM$ be the
canonical decomposition of $M$ into homogeneous KH-submodules. Then
each $q_k\in \fix(S)$ and each $q_kM$ is $T$-invariant. 
\end{enumerate}
\end{lemma}

\begin{proof}
(i) As each $T_t$ is bijective, $\supp(M) \in \fix(S)$. But
$\abs{e_1}= \dots = \abs{e_N} = \supp(M)$ is clear.

\noi 
(ii) Recall from Section \ref{s.finrank} the definition of the {\em dimension} of $M$ as
$\dim_M \coloneq \sum_{y \in \calB} \abs{y}^2$, 
where $\calB$ is any finite suborthonormal basis for $M$. As $M$ is
$T$-invariant and  each $T_t$ maps a suborthonormal basis onto another
suborthonormal basis, $\dim_M \in \fix(S)$. It follows that each $q_k \in
\fix(S)$ (cf.\ its definition in Section \ref{s.finrank})
and hence $q_kM$ is $T$-invariant.
\end{proof}

We are now ready to state and prove the main result of this
article. (Recall the definition of $u_\calB$ from above.)

 \begin{theorem}[Decomposition Theorem]\label{mainthm}
Let $(\A,S;E,T)$ be a KH-dynamical system. 
Then
\[ \fix(T\tensor \konj{T}) = \ocl\spann_{\fix(S)}(U).
\]
where $U= \{ u_\calB \,|\,  
\text{$\calB \subseteq E$ finite, homogeneous, suborthonormal}\} \cap 
\fix(T\tensor \konj{T})$.

\smallskip
Furthermore, $E$ decomposes orthogonally as $E = 
E_{\ds} \oplus E_{\wm}$ into $T$-invariant KH-submodules, where 
\begin{align*}
E_{\ds} & = \mathrm{ocl}\bigcup
\{ M\subseteq E \,|\, \text{$M$ finitely-generated, $T$-invariant submodule}\}\\
& = 
\ocl\sum
\{ M\subseteq E \,|\, \text{$M$ homogeneous $T$-invariant KH-submodule
  of finite rank
}\}\\
& = \ocl\sum\{ \ran(A) \,|\, A\in \HS_T(E)\}\qquad \text{and}\\
E_{\wm} & = \left\{x \in E \mid x \otimes \overline{x} \perp \fix(T \otimes \konj{T})\right\}.
\end{align*}
\end{theorem}

The spaces $E_{\ds}$ and  $E_{\wm}$  are called the {\emdf
  discrete spectrum part} and {\emdf weakly mixing part}, respectively. 

\medskip
The first statement of Theorem \ref{mainthm} is the KH-analogue
of the ``Key Lemma'' of the Introduction. 
By Lemma \ref{invariance}, it is equivalent to the following result.

\begin{lemma}\label{keylemma}
Let $(\A,S;E,T)$ be a KH-dynamical system. 
Then  
\begin{align*}
 \HS_T(E) & = \ocl\spann_{\fix(S)} \bigl\{ P \in \HS_T(E)\,|\,
  P\,\, \text{\rm orthogonal projection  onto} \\
& \qquad \qquad \qquad \qquad \qquad\text{\rm a homogeneous finite-rank
                                                    submodule}\bigr\}.
\end{align*}
\end{lemma}

\begin{proof}
Fix  $A\in \HS_T(E)$. We need to show that one can 
approximate $A$ in order by sums of operators of the form $\lambda P$
where $\lambda \in \fix(S)$ and $P\in \HS_T(E)$ is an orthogonal projection
onto a $T$-invariant, homogeneous, finite-rank submodule.
This is done in several reduction steps. 

Firstly, by the decomposition $A = \frac{1}{2}(A+ A^*) + \ui
\frac{1}{2\ui}(A-A^*)$ we may suppose that $A$ is self-adjoint. 
Secondly,  by virtue of the spectral decomposition (\cref{spt-cov}),
we may further suppose that $A = \lambda P$ where $P$ is a $T$-intertwining 
 orthogonal projection and $0 \le \lambda \in \fix(S)$.

Next, abbreviate $p \coloneq \supp(\lambda) \in \fix(S)$ and
define $p_\veps \coloneqq \supp( (\lambda - \veps\car)^+)$.
Then $p_\veps\in\B\cap \fix(S)$, $\veps p_\veps \le \lambda$ and    $p_\veps \nearrow
p$ as $\veps \searrow 0$. 
It follows that $p_\veps\lambda P \to \lambda P = A$ in order, hence we may
suppose $A = q\lambda P$ and 
$\veps q \le \lambda$ for some $\veps > 0$ and some $q\in \B \cap \fix(S)$.

Define $Q \coloneq qP$. Then $A = \lambda Q$ and $Q$ is a $T$-intertwining orthogonal projection.
Moreover, $Q$ is $\A$-Hilbert--Schmidt, as
\[    
\sum_{y \in \calB} \abs{Qy}^2 \le \frac{1}{\veps^2} 
\sum_{y \in \calB} \lambda^2\abs{Qy}^2  \le  \frac{1}{\veps^2} \abs{A}_{\HS}^2
\]
for each finite suborthonormal set $\calB\subseteq E$. By Lemma \ref{lem:proj-invar}, $Q$
is of finite rank. 

Finally, by applying Lemma \ref{hom-invar} we may suppose that $\ran(Q)$ is homogeneous.
\end{proof}

\begin{proof}[Proof of \cref{mainthm}.]
Each finitely-generated $T$-invariant submodule of $E$ is contained in 
a $T$-invariant KH-submodule of finite rank (e.g., its order-closure),
and
each $T$-invariant KH-submodule of finite rank is contained in a
sum of homogeneous $T$-invariant KH-submodules of finite rank (Lemma
\ref{hom-invar}). Hence, the first two alternative definitions of $E_{\ds}$
coincide.  The third characterization follows directly from Lemma
\ref{keylemma}. 

It remains to prove $E_{\ds}^\perp = E_{\wm}$. 
To this end, let  $x\in E$ and  $\calB$ be a suborthonormal set. Then 
\[ (x \tensor \konj{x}| u_\calB) = \sum_{e\in \calB} (x\tensor \konj{x}|e\tensor
  \konj{e})
=     \sum_{e\in \calB} \abs{(x|e)}^2,
\]
hence   $x\perp \calB$ if and only if $x\tensor \konj{x} \perp u_\calB$. 
It follows that $x\perp E_{\ds}$ if and only if $x\tensor \konj{x} \perp
u_\calB$ for all $u_\calB \in U$. And, by Lemma  \ref{keylemma}, this is the
case if and only if $x\tensor \konj{x} \perp \fix(T \tensor \konj{T})$.
\end{proof}

\begin{remark}\label{main-cycp}
Employing the fact that the spectral theorem (Theorem \ref{spt-KH}) 
also holds for  self-adjoint cyclically compact operators (cf.\ Remark \ref{spt-cycp}) one can prove
without difficulty that
\[ E_{\ds} = \ocl\sum\bigl\{ \ran(A) \,|\, A \in \End_T(E)\,\,
\text{cyclically compact}\bigl\}.
\]
\end{remark}

\vanish{
\subsection{Special  Cases}

We cover the ``ergodic case''  $\fix(S) = \C\car$ and the other extremal case
$\fix(S) = \A$ (classical representation). 
({\em To be filled in.})

}

\section*{Notes and Comments}\label{s.notesII}

The main example of a ``KH-dynamical system'' comes from
an extension $\bfX \to \bfY$ of measure-preserving systems, where $T$ and $S$ are the
induced Koopman representations on $\Ell{2}(\uX|\uY)$ and $\Ell{\infty}(\uY)$,
respectively. Of course, this example motivated our abstract notion. On the other
hand, as demonstrated in Remarks \ref{remstoreps}, the concept assorts well with 
the established theory of covariant representations of  C$^*$-algebras.

\medskip
In the whole theory, the case $\A = \C$ is the classical Hilbert space
situation, discussed in the Introduction. Indeed, as 
on $\C$ there is only a trivial dynamics, a ``KH-dynamical
system'' in the case $\A= \C$ is nothing but a unitary representation $T\colon G \to
\BL(H)$ of a group $G$ on  a Hilbert space $H$.   
Theorem \ref{mainthm} then reduces to the Hilbert space
decomposition  $H = H_{\ds} \oplus H_{\wm}$ mentioned in the corollary on page
\pageref{cor-intro}. 

It turns out that this decomposition coincides
with the so-called {\em Jacobs--deLeeuw--Glicksberg} decomposition $H =
H_{\rev} \oplus H_{\mathrm{aws}}$, which
arises whenever a semigroup acts contractively on a Hilbert space
$H$. Here, on $H_{\rev}$ the semigroup embeds into a strongly
continuous action of a
compact group, and on $H_{\mathrm{aws}}$ the zero vector is contained in  the weak closure of each
orbit, see \cite[Chap.~16]{EFHN2015}. 

In order to see that indeed $H_{\rev} = H_{\ds}$ (and hence also $H_{\wm} = H_{\mathrm{aws}}$) 
one needs an essential part of the Peter--Weyl theorem, namely that
$H= H_{\ds}$ whenever $G$ is compact and 
$T$ is strongly continuous, see \cite[Thm.~16.31]{EFHN2015} and its proof.  

On the other hand, that  part of the Peter--Weyl theorem
is actually a straightforward consequence of the decomposition result of Theorem \ref{mainthm} for $\A = \C$ (i.e., the corollary stated on page \pageref{cor-intro} of the introduction).  
Indeed, suppose that $G$ is a compact group and the unitary representation $T$ of $G$
on $H$ is strongly continuous. Then $P \coloneq \int_G T_t \tensor \konj{T_t} \, \ud{t}$ is the orthogonal
projection onto $\fix(T \tensor \konj{T})$. Hence, for $x\in H_{\wm}$  one has
\[ 0 = \sprod{x \tensor \konj{x}}{P(x\tensor \konj{x})} = 
\int_G \abs{\sprod{ T_t x }{x}}^2 \, \ud{t},
\]
and this implies $x=0$ since the Haar measure on $G$ has
full support. It follows that $H = H_{\ds}$ as claimed.

These findings lead to the conclusion that although the decomposition
of
Jacobs--deLeeuw--Glicksberg and the discrete spectrum/weakly mixing decomposition coincide
{\em extensionally}, they are {\em intensionally} different and the link between them 
is furnished only by a corollary of the latter.

\medskip
It can be expected that there is an  analogue of the
Jacobs--deLeeuw--Glicksberg theory for semigroups on KH-modules
(simply because in the Boolean/conditional 
world there is an analogue for every classical result, see the notes to Part I on
page \pageref{s.notesI}), which then relates to  the ds/wm-decomposition  
just as in the Hilbert space case. This is the object of future work.

\medskip
Since the spectral theorem also holds for cyclically compact operators
(see Remark \ref{spt-cycp} and the 
 notes to Part I on page \pageref{s.notesI}), one can extend the description
of the discrete spectrum part of $E$ involving cyclically compact operators
(see Remark \ref{main-cycp}).

\part{The Furstenberg--Zimmer Structure Theorem}\label{p.FZ}

In this part we apply the Decomposition Theorem (\cref{mainthm})
to extensions of measure-preserving systems in order 
to establish the Furstenberg--Zimmer structure theorem. Again, 
$G$ denotes an arbitrary, but fixed group.

\section{Extensions of Measure-Preserving Systems}\label{s.extensions}

Classically, a {\em measure-preserving $G$-system} is a pair $\bfX = (\uX;\phi)$ of a probability space 
$\mathrm{X} = (X,\Sigma_{\uX},\mu_{\uX})$ and a family $(\phi_t)_{t \in G}$ of  measurable and measure-preserving maps $\phi_t \colon X \to X$ 
for $t \in G$ such that $\phi_{ts} = \phi_t \circ \phi_s$ holds almost
everywhere for $t,s \in G$ and $\phi_0 = \mathrm{id}_X$.

Such a dynamical system always induces a group $(T_t)_{t \in G}$ of
(``Koopman'') operators on the corresponding 
$\uL^2$-space via 
$T_tf \coloneqq f \circ \phi_t^{-1}$ for $f \in \uL^2(\uX)$. The operators $T_t$
are examples of {\emdf Markov embeddings}, i.e., isometries $T \in \mathscr{L}(\uL^2(\uX))$ satisfying 
  \begin{itemize} 
    \item $|Tf| = T|f|$ for every $f \in \uL^2(\uX)$,
    \item $T\car = \car$.                
  \end{itemize}
(See \cite[Chap.~13]{EFHN2015} for a detailed discussion of general Markov
operators. Bijective  Markov embeddings are also called {\em Markov isomorphisms}.)

In the case of a {\em standard}  probability space $\mathrm{X}$, every unitary group representation $G \to 
\mathscr{L}(\uL^2(\uX))$ as Markov isomorphisms is the ``Koopmanization'' of 
an underlying representation by means of point transformations.
This follows from an important  result of von Neumann (see 
\cite[Prop.~7.19 and Thm.~7.20]{EFHN2015}). 
If the probability space is not standard, however, such an operator group is induced only by 
transformations of the associated measure algebra (see \cite[Thm.~12.10]{EFHN2015}), but not necessarily by underlying point transformations.  

Giving up point transformations and passing to the functional analytic side
amounts to a change of category. This change involves a reversing of arrows, to
the effect that  {\em factors} and {\em factor maps} in the classical framework are replaced
by {\em extensions} and {\em embeddings} in the functional analytic framework.
To wit, where in classical ergodic theory one would write $\bfX \to \bfY$ 
for an extension of systems, with $\bfY$ being the factor and $\bfX$ the
extension, we write $J\colon (\uY;S) \to (\uX; T)$, where $J$ is the
Markov embedding of $\Ell{2}(\uY)$ into $\Ell{2}(\uX)$ and $S$ and $T$ are
the respective Koopman representations. Actually, taking the change of category 
seriously amounts to requiring only the functional-analytic properties of such
representations and dispensing with the requirement that they arise
from underlying  point transformations. 
However unfamiliar this step might be at first, 
it is quite natural and advantageous for
structural considerations, e.g., because  the new category  is free from countability
restrictions and  better suited for universal constructions.  

Whereas the book \cite{EFHN2015} can be seen as a bridge between the two worlds,
here  we do not consider point transformations at all and work exclusively on 
the functional-analytic side. Hence the following definitions.

\begin{definition}\label{d.extension}
  A {\emdf measure-preserving $G$-system} is a pair $(\uX;T)$ with a
        probability space $\uX$ and a  
representation 
    \begin{align*}
      T \colon G \to \mathscr{L}(\uL^2(\uX)), \quad t \mapsto T_t
    \end{align*}
  of $G$ as Markov lattice isomorphisms on $\uL^2(\uX)$.

\smallskip
\noindent
An {\emdf extension} (or {\emdf morphism}) $J \colon (\uY;S) \to (\uX;T)$ of 
measure-preserving systems
 is a Markov embedding $J 
\in \mathscr{L}(\uL^2(\uY), \uL^2(\uX))$ such that the diagram
  \begin{align*}
      \xymatrix{
          \uL^2(\uX) \ar[r]^{T_t}& \uL^2(\uX)\\
            \uL^2(\uY) \ar[r]^{S_t} \ar[u]^J& \uL^2(\uY)  \ar[u]_J
        }
    \end{align*}
  is commutative for every $t \in G$. 

\smallskip
\noindent
Two extensions $J_1 \colon (\uY_1;S^1) \to (\uX;T)$ and $J_2 \colon (\uY_2;S^2) \to
(\uX;T)$ are  {\emdf equivalent} if 
there is a Markov lattice isomorphism $V \colon \uL^2(\uY_1) \rightarrow
\uL^2(\uY_1)$ intertwining the dynamics and such that the diagram
  \begin{align*}
    \xymatrix{
          & (\uX;T)& \\
            (\uY_1;S^1) \ar[ru]^{J_1} \ar[rr]^V &  &(\uY_2;S^2) \ar[lu]_{J_2}
        }
  \end{align*}  
is commutative.
\end{definition}

We start from an extension $J \colon (\uY;S) \to (\uX;T)$ such that the range $\ran J \subseteq \uL^2(\uX)$ is 
an invariant and closed  \emph{unital vector sublattice} of $\uL^2(\uX)$. (The latter
means that $\ran J$ is a closed subspace containing $\car$  and closed with
respect to taking   real and imaginary parts as well as   the modulus.) Each unital closed vector sublattice is of the form
$\uL^2(X,\Sigma',
\mu_\uX)$ for some (usually non-unique) $\sigma$-subalgebra $\Sigma'$ of
$\Sigma_\uX$ \cite[Thm.~13.19]{EFHN2015}.  Moreover, 
two extensions $J_1$ and $J_2$ as above  are equivalent precisely when $\ran
J_1 = \ran J_2$. It follows that---{\em up to
  equivalence}---extensions into $(\uX;T)$  are in one-to-one correspondence
with the invariant and closed  unital vector sublattices of $\uL^2(\uX)$,   
and each extension can be seen as a proper inclusion.
(See \cite[Thm.s 13.19 and 13.20]{EFHN2015} for these assertions.)

\medskip

\subsection{Inductive limits}

For later use, let us recall here the concept of 
inductive limits of measure-preserving systems (see
\cite[Sec.~13.5]{EFHN2015}). 

  \begin{definition}
An {\emdf inductive system} is a net 
$(\uX_\alpha;T_\alpha)_\alpha$ of measure-preserving systems together with
an extension $J_\alpha^\beta \colon (\uX_\alpha;T_\alpha) \to
(\uX_\beta;T_\beta)$ whenever $\alpha \le \beta$ and with the following properties:
    \begin{itemize}
      \item $J_\beta^\gamma J_\alpha^\beta =
                          J_\alpha^\gamma$\quad  for $\alpha \leq \beta \leq \gamma$ and
      \item $J_\alpha^\alpha = \mathrm{Id}$\quad  for each $\alpha$.  
    \end{itemize}

\noi
An {\emdf inductive limit} of such an inductive system is 
a  measure-preserving system $(\uX;T)$ together with extensions $J_\alpha \colon (\uX_\alpha;T_\alpha) \to (\uX;T)$ such that $J_\alpha = J_\beta J_\alpha^\beta$ for 
$\alpha \leq \beta$ and with the following universal property:
    \begin{itemize}
      \item Whenever $(\uY;S)$ is a measure-preserving system and $I_\alpha \colon (\uX_\alpha;T_\alpha) \to 
(\uY;S)$ are extensions with $I_\alpha =  I_\beta J_\alpha^\beta$ for $\alpha \leq \beta$, then there is a unique extension $J 
\colon (\uX;T) \to (\uY;S)$ such that for each $\alpha$ the diagram
        \begin{align*}
          \xymatrix{
            & (\uY;S)\\
            (\uX_\alpha;T_\alpha)  \ar[ru]^{I_\alpha} \ar[r]_{J_\alpha}  & (\uX;T)\ar[u]_{J}  \\    
            }
        \end{align*}
      is commutative.
    \end{itemize}
    In this case, we write 
      \begin{align*}
        (\uX;T) = \lim_{\substack{\longrightarrow\\ \alpha}} (\uX_\alpha;T_\alpha).
      \end{align*}
  \end{definition}
  By \cite[Thm.~13.38]{EFHN2015} (which readily extends to the case of arbitrary group actions), every inductive system has an 
inductive limit. Moreover, by the universal property, an inductive limit
 is unique up to a natural isomorphism.

\medskip

\subsection{The Associated KH-Dynamical System}\label{s.akhds}

Let   $J \colon \Ell{2}(\uY) \to \Ell{2}(\uX)$ be a Markov embedding and
abbreviate $\A \coloneqq \Ell{\infty}(\uY)$. 
The adjoint $J^*$ of $J$ is a Markov operator and hence 
extends uniquely to a Markov operator 
\[ \E_{\uY}\colon \uL^1(\uX) \to \uL^1(\uY),
\]
see \cite[Prop.~13.6]{EFHN2015}. We think of $\E_{\uY}$ as  a {\em conditional expectation}.
(Actually, the  true conditional expectation is the orthogonal projection $J^* J$ 
onto $\ran J$, see \cite[Sec.~13.3]{EFHN2015}).

Since $J$ is a Markov embedding, one has
\[  J(fg) = J(f) \cdot J(g) \quad \text{for all $f,g\in \A$},
\]
see \cite[Sec.~13.2]{EFHN2015}. Hence, the product
\[   \A \times \Ell{1}(\uX) \to \Ell{1}(\uX),\quad    
(f,g) \mapsto (Jf)g
\]
turns $\uL^1(\uX)$ into an $\A$-module.  
The conditional expectation 
$\E_{\uY}\colon \uL^1(\uX) \to \uL^1(\uY)$ is a module homomorphism
\cite[Thm.~13.12]{EFHN2015}. Define
\[ (f|g)_\uY \coloneq \E_{\uY}(f \konj{g}) \quad \text{and}\quad 
 \abs{f}_{\uY} \coloneq \sqrt{ (f|f)_{\uY}} = \sqrt{ \E_{\uY}\abs{f}^2}
\quad \text{for $f,g \in \Ell{2}(\uX)$}.
\] 
Then the mapping
\[    (\cdot\,|\, \cdot)_{\uY} \colon \Ell{2}(\uX)\times \Ell{2}(\uX) \to \Ell{1}(\uY)
\]
is $\A$-sesquilinear, positive definite and bounded.
Hence, one has
the inequalities
\[  \abs{(f|g)_\uY} \le \abs{f}_{\uY} \abs{g}_{\uY}\quad \text{and}\quad
  \abs{f+g}_{\uY}\le
\abs{f}_{\uY}+ \abs{g}_{\uY}
\]
for all $f, g\in \Ell{2}(\uX)$. (The second follows
from the first and the first needs only be proved for $f,g\in \Ell{\infty}(\uX)$,
where it follows from standard arguments.) Consequently, 
\[ \abs{\abs{f}_\uY - \abs{g}_\uY } \le \abs{f-g}_{\uY} \qquad (f,g\in \Ell{2}(\uX))
\]
and hence, taking $\Ell{2}$-norms,
\[      \norm{\abs{f}_\uY - \abs{g}_{\uY}}_2 \le \norm{f-g}_2
\qquad (f,g\in \Ell{2}(\uX)).
\]
In order to obtain a Hilbert module, we need to consider a submodule of 
$\uL^2(\uX)$.

\begin{definition}\label{conditionall2}
Let $\uX$ and $\uY$ be probability spaces and
$J\colon \uL^2(\uY)\to \uL^2(\uX)$
 a Markov embedding. 
The {\emdf conditional $\uL^2$-space} is
    \begin{align*}
      \uL^2(\uX|\uY) \coloneqq \{f \in \uL^2(\uX)\mid \E_{\uY}|f|^2 \in \uL^\infty(\uY)\},
\end{align*}
equipped with the $\uL^\infty(\uY)$-valued inner product $(\cdot|\cdot)_\uY$
and  the $\uL^\infty(\uY)$-valued norm $\abs{\cdot}_{\uY}$.
\end{definition}

By what we have seen above, $\Ell{2}(\uX|\uY)$ is a pre-Hilbert module
over the unital commutative $\mathrm{C}^*$-algebra $\Ell{\infty}(\uY)$.
Note that our definition of $\uL^2(\uX|\uY)$  coincides with Tao's
\cite[Sec.~2.13]{Tao2009} 
but differs from the one given in \cite[p.51]{KerrLi2016}.

\begin{remark}\label{innerproducts}
Apart from the structures just introduced,  
the pre-Hilbert module $E = \uL^2(\uX|\uY)$ is equipped also with
the usual inner product and norm inherited from $\uL^2(X)$, and the
usual modulus. These are denoted by  $(f|g), \norm{f}_2, \abs{f}$, respectively,
in order to avoid  confusion with $(f|g)_\uY, \norm{f}_\uY, \abs{f}_\uY$.

Order-related notions like order-convergence, order-Cauchy
property or order-closure will always refer 
to $\uL^2(\uX|\uY)$ as a lattice-normed module over 
$\uL^\infty(\uY)$,  and not in the sense of a sublattice of $\uL^2(\uX)$, unless otherwise noted.
\end{remark}

Since $\uY$ is a probability space, $\A \coloneq \uL^\infty(\uY)$ is
order-complete, i.e.{} a \Stoneanalgebra, see \cite[Cor.~7.8 and
Rem.~7.11]{EFHN2015}. We intend to  show that $\uL^2(\uX|\uY)$ is
a Kaplansky--Hilbert module over $\A$. 
To this aim, and also for later use, we need to relate the different convergence notions.

\begin{lemma}\label{ordervsl2}
Let $\uX$ and $\uY$ be probability spaces and $J\colon \uL^2(\uY)\to \uL^2(\uX)$ a Markov embedding. 
\begin{enumerate}[(i)]

\item 
For a bounded sequence $(f_n)_n$ in $\Ell{\infty}(\uY)$ one has
$\olim_{n\to \infty} f_n = 0$ if and only if $f_n \to 0$ $\mu_\uY$-almost everywhere.

\item Each order-bounded subset of $\uL^2(\uX|\uY)$ is
  $\norm{\cdot}_2$-bounded. Each order-convergent net in $\Ell{2}(\uX|\uY)$ is 
$\Ell{2}$-convergent (to the same limit).

\item If $M$ is an $\uL^\infty(\uY)$-submodule of $\uL^2(\uX|\uY)$, then within
 $\uL^2(\uX|\uY)$  the $\norm{\cdot}_2$-closure of $M$ 
  coincides with its order-closure. Moreover, 
$f \in \ocl(M)$ if and only if there is a sequence $(f_n)_n$ in $M$ such that 
for each $\veps > 0$ there is a $\uY$-measurable set $A$ with
$\mu_\uY(A^c) \le \veps$ and $\norm{ \car_A\abs{f_n - f}_\uY}_{\Ell{\infty}} \to 0$.
\end{enumerate}
\end{lemma}

\begin{proof}

(i) This follows from the fact, that if $(f_n)_n$ is a bounded
sequence, then the superior limits $\limsup_{n\to \infty}
\abs{f_n}$ in the order-sense and in the a.e.-sense coincide.

\smallskip
\noi
(ii) This follows from $\norm{f}_2^2 = \int_{\uY} \E_{\uY}\abs{f}^2 = 
\int_{\uY} \abs{f}^2_{\uY}$  for $f\in \Ell{2}(\uX|\uY)$ and the fact that 
if $(u_\alpha)_\alpha$ decreases to $0$ in $\Ell{\infty}(\uY)$, then 
$\norm{u_\alpha}_2 \searrow 0$, see \cite[Thm.~7.6]{EFHN2015}.

\smallskip
\noi
(iii) By (ii) and Lemma \ref{order-closure}, 
the order-closure is contained in the  $\Ell{2}$-closure.  Conversely, suppose
that  $f_n \in M$, $f\in \Ell{2}(\uX|\uY)$  and $\norm{f_n - f}_2 \to 0$. 
Passing to a subsequence we may suppose $\abs{f_n - f}_\uY \to 0$ almost
everywhere. 
Fix $\veps > 0$.
Then by Egoroff's theorem there is a $\uY$-measurable set $A$ such that
$\mu_\uY(A^c) \le \veps$ and 
$\norm{ \abs{\car_Af_n - \car_Af}_\uY}_{\Ell{\infty}} \to 0$. It follows that 
$\car_A f\in \ocl(M)$.  By (i), $f\in \ocl(M)$ as desired.
\end{proof}

\begin{proposition}\label{p.condL2-KH}
  Let $\uX$ and $\uY$ be probability spaces and $J\colon \uL^2(\uY)\to \uL^2(\uX)$ a Markov embedding. Then 
$\uL^2(\uX|\uY)$ is a Kaplansky--Hilbert module over $\uL^\infty(\uY)$. 
\end{proposition}

\begin{proof}
Let $(f_\alpha)_\alpha$ be an order-Cauchy net in $\Ell{2}(\uX|\uY)$.
Then the net $(f_\alpha -f_\beta)_{(\alpha, \beta)}$ order-converges to $0$. 
By Lemma \ref{ordervsl2}, $\norm{f_\alpha - f_\beta}_2 \to 0$ as $(\alpha, \beta) \to
\infty$, i.e., $(f_\alpha)_\alpha$ is $\norm{\cdot}_2$-Cauchy.
As $\Ell{2}(\uX)$ is complete, there is $f\in \Ell{2}(\uX)$ 
with $\norm{f- f_\alpha}_2 \to 0$. 

Since $(f_\alpha)_\alpha$ is order-Cauchy, we find an index $\alpha_0$ 
and a net $(u_\alpha)_{\alpha \ge \alpha_0}$ decreasing to $0$ in
$\Ell{\infty}(\uY)$
such that 
\[  \abs{f_\alpha - f_\beta}_\uY \le u_\gamma \qquad (\alpha, \beta \ge \gamma
  \ge \alpha_0).
\]
(Use Lemma \ref{lem:easynetchar}.)
Since the mapping $g \mapsto \abs{f_\alpha - g}_\uY$ is continuous with respect
to $\Ell{2}$-norms, by letting $\beta \to \infty$ we obtain
\[  \abs{f_\alpha - f}_\uY \le u_\gamma \quad (\alpha \ge \gamma
  \ge \alpha_0).
\]
And this just means that $\olim_{\alpha} f_\alpha = f$.
\end{proof}

\medskip

Let $J \colon 
(\uY;S) \to (\uX;T)$ be any extension of dynamical systems and, as before,
$\A \coloneq \Ell{\infty}(\uY)$.  Then 
for every $t\in G$ the operator $S_t$ restricts to an 
automomorphism $S_t \in \Aut(\A)$ of the \Stoneanalgebra{} $\A$. Moreover, 
$T_t$ restricts to an $S_t$-unitary operator (again denoted by $T_t$) on the KH-module 
$\uL^2(\uX\mid \uY)$. In this manner we obtain the {\emdf associated
KH-dynamical system} $(\A,S;\uL^2(\uX|\uY),T)$.

\medskip

\subsection{The ``Dual Extension''}

As is common, we identify the dual of $\Ell{2}(\uX)$ with $\Ell{2}(\uX)$
via the duality $\dprod{f}{g} \coloneq\int_\uX f g$. Then 
$\dprod{f}{\konj{g}} = (f|g)$ for all $f,g\in \Ell{2}(\uX)$. 

Likewise, if $J \colon \Ell{2}(\uY) \to \Ell{2}(\uX)$ is  a Markov embedding and
$\Ell{2}(\uX|\uY)$ the associated \KH{} module over 
$\A = \Ell{\infty}(\uY)$,  we identify the dual module  $\Ell{2}(\uX|\uY)^*$
with $\Ell{2}(\uX|\uY)$ via the duality
\[  \dprod{f}{g}_\uY  \coloneq \E_\uY( f g)\qquad (f,g \in \Ell{2}(\uX|\uY)).
\]
With this identification, the  {\em formal} conjugation mapping 
$\Theta\colon \Ell{2}(\uX|\uY) \to \Ell{2}(\uX|\uY)$ is just ordinary conjugation,
 i.e.,
$\dprod{f}{\konj{g}}_\uY = (f|g)_\uY$ for all $f,g\in \Ell{2}(\uX|\uY)$. 

\medskip
Let $T\colon \Ell{2}(\uX|\uY) \to \Ell{2}(\uX|\uY)$  be any   linear   operator. 
Recall from Section \ref{s.otKH} that the {\em conjugate operator} is defined
through $\konj{T}\:\konj{f} = \konj{Tf}$. Hence, if  $T$ is {\em real}, i.e., if $T$ maps
real functions to real functions, then $\konj{T} = T$. This applies,
in particular, to Markov homomorphisms, and hence to measure-preserving systems. 

Given an extension $J\colon (\uY;S) \to (\uX;T)$, the \enquote{dual extension} is therefore
just $J$ again, and the dual $G$-system of the associated 
system $(\A,S;\uL^2(\uX|\uY),T)$ is precisely this system again.

\medskip

\subsection{Relatively Independent Joining}

Suppose  
\[ J_\uX \colon \Ell{2}(\uY) \to \Ell{2}(\uX)\quad  \text{and}\quad
 J_{\uX'} \colon \Ell{2}(\uY) \to 
\Ell{2}(\uX')
\]
are Markov embddings. Then we can form the associated
KH-modules $\Ell{2}(\uX|\uY)$ and $\Ell{2}(\uX'|\uY)$. We aim at showing that 
their  \KH{} tensor product can be written as 
\[ \Ell{2}(\uX|\uY) \tensor \Ell{2}(\uX'|\uY) 
= \Ell{2}(\uZ|\uY)
\]
for certain Markov embeddings  $I_{\uX} \colon \Ell{2}(\uX) \to \Ell{2}(\uZ)$ and
$I_{\uX'} \colon \Ell{2}(\uX') \to \Ell{2}(\uZ)$.

\medskip
In general, a pair of Markov embeddings $I_{\uX}\colon\Ell{2}(\uX) \to \Ell{2}(\uZ)$ and
$I_{\uX'}\colon \Ell{2}(\uX') \to \Ell{2}(\uZ)$ is called a {\emdf coupling}, if 
the ranges $\ran(I_\uX)$ and $\ran(I_{\uX'})$ together generate $\Ell{2}(\uZ)$ as a
Banach lattice. This is equivalent to saying that the  $*$-subalgebra 
\[ \spann\left\{I_\uX f \cdot I_{\uX'} g\mid f \in \uL^\infty(\uX),\, g \in \uL^\infty(\uX')\right\}
\]
is dense in  $\uL^2(\uZ)$.

Given such a coupling, one can form the Markov 
operator 
\[ P(I_\uX, I_{\uX'}) \coloneqq  \E_{\uX'} \nach I_\uX \colon \Ell{2}(\uX)\to
  \Ell{2}(\uX').
\]
This so-called {\emdf coupling operator} determines the coupling up to a canonical isomorphism. And even
more is true:

\begin{proposition}\label{coupling}
Let $P\colon  \Ell{2}(\uX)\to
  \Ell{2}(\uX')$ be a Markov operator. Then there is a coupling $(\uZ,
  I_\uX,I_{\uX'})$ such that $P = P(I_{\uX}, I_{\uX'})$.  
This coupling is unique in the sense that any two couplings with this property
are canonically isomorphic.
\end{proposition}

\begin{proof}
By passing to Stone topological models as in \cite[Chap.~12]{EFHN2015}, we may
suppose that $\Ell{\infty}(\uX) = \Ce(X)$ and $\Ell{\infty}(\uX') = \Ce(X')$
with Stonean spaces $X, X'$.  Then  $P$ restricts to a Markov
operator\footnote{A {\emdf Markov operator} between $\Ce(K)$-spaces is a
positive linear operator that maps $\car$ to $\car$.}
$P\colon \Ce(X) \to \Ce(X')$. Moreover, $Z \coloneq X \times X'$ is also
compact and we define the operator $Q\colon \Ce(Z) \to \Ce(X')$ by 
$Qh(x') \coloneq    (Ph(\cdot, x'))(x')$. A moment's thought reveals that $Q$ is a Markov operator
and 
\[     Q(f \tensor g) = Pf \cdot g \qquad (f\in \Ce(X),\, g \in \Ce(X')).
\]
Endow $Z$ with the probability measure $\mu_\uZ \coloneq Q'\mu_{\uX'}$, i.e., through 
\[     \int_{\uZ} h = \int_{\uX'} Qh \qquad (h\in \Ce(Z)).
\]
The mapping $I_\uX\colon \Ce(X) \to \Ce(Z)$, $I_\uX f  = f \tensor \car$ 
is the Koopman operator of the projection $(x,x') \mapsto x'$. It
extends uniquely to a Markov embedding,
\[   I_\uX\colon \Ell{2}(\uX) \to \Ell{2}(\uZ).
\]
Analogously we obtain the Markov embedding 
$I_{\uX'}\colon \Ell{2}(\uX') \to \Ell{2}(\uZ)$. As $\Ce(X) \tensor \Ce(X')$ is dense
in $\Ce(Z)$, it follows that $(I_\uX, I_{\uX'}, \uZ)$ is a coupling. 

Next, it
is routine to show that $Q$, or rather its extension to a Markov operator
$\Ell{2}(\uZ) \to \Ell{2}(\uX')$, is precisely the adjoint of $I_{\uX'}$.
(It suffices to test on elements of $\Ce(X) \tensor \Ce(X')$.) Given $f\in
\Ce(X)$ we finally obtain
\[    Q (I_\uX f) = Q (f \tensor \car) = Pf ,
\]
hence  $P = P(I_\uX, I_{\uX'})$. This concludes the existence proof. The
uniqueness proof is technical but routine, hence we omit it. 
\end{proof}

Now suppose that, as before,  we start with two Markov embeddings
\[ J_\uX \colon \Ell{2}(\uY) \to \Ell{2}(\uX)\quad  \text{and}\quad
 J_{\uX'} \colon \Ell{2}(\uY) \to  \Ell{2}(\uX').
\]
Define $P \coloneq J_{\uX'} \nach \E_\uY^\uX$, where $\E_{\uY}^\uX =J_{\uX}^*$ is the 
``conditional expectation'' from $\Ell{2}(\uX)$ onto $\Ell{2}(\uY)$. 
Then $P\colon \Ell{2}(\uX) \to \Ell{2}(\uX')$ is a Markov operator. 
The corresponding coupling is denoted by $\uZ = \uX \times_\uY \uX'$
and called the {\emdf relatively independent coupling} (associated with the
embeddings $J_{\uX}, J_{\uX'}$). 

Note that Markov operators between $\Ell{2}$-spaces uniquely extend
to the $\Ell{1}$-spaces \cite[Prop.~13.6]{EFHN2015}. If $f\in \Ell{2}(\uX)$ and 
$g \in \Ell{2}(\uX')$ then 
\[     f \tensor g \coloneq (I_\uX f) \cdot (I_{\uX'} g)\in \Ell{1}(\uX
  \times_\uY \uX').
\]
For $h\in \Ell{\infty}(\uY)$, $f\in \Ell{\infty}(\uX)$ and $g\in
  \Ell{\infty}(\uX')$ we have
\begin{align*}
(f \tensor g| J_\uX h \tensor \car) 
& = \int_{\uZ} (f \konj{J_\uX h}) \tensor g = \int_{\uX'} P(f J_\uX\konj{h})g
= \int_{\uY} \E^{\uX}_{\uY}(f J_\uX\konj{h}) 
 \cdot \E^{\uX'}_{\uY}g 
\\ & = (\E^{\uX}_{\uY}f 
 \cdot \E^{\uX'}_{\uY}g\, |\, h)_{\Ell{2}(\uY)}= \dots =  (f \tensor g| \car
     \tensor J_{\uX'} h) .
 \end{align*}
This means $I_{\uX}  \nach   J_\uX =  I_{\uX'}  \nach   J_{\uX'}$ or, when we suppress
explicit
reference to the embeddings $J_{\uX}$ and $J_{\uX'}$,
\[  h \tensor \car = \car \tensor  h\qquad
\text{for each $h\in \Ell{2}(\uY)$}.
\]
So the relatively independent coupling is a coupling over
  $\uY$. 
Furthermore, the computation from above also yields
\[   \E_\uY (f\tensor g)  = \E_{\uY} f \cdot \E_\uY g \qquad (f\in \Ell{2}(\uX),
  \, g\in \Ell{2}(\uX')).
\]
This explains the name of the coupling.  Observe that this property
uniquely determines the coupling in the sense of Proposition \ref{coupling}.

\medskip
Couplings in the presence of dynamics are called {\emdf joinings},
cf.{} \cite[Chap.~6]{Glas}. 
The following lemma describes under which condition
a coupling can be turned into a joining.

\begin{lemma}\label{coup=join}
Let $(\uZ, I_{\uX}, I_{\uX'})$ be a coupling with associated Markov operator
$P= P(I_{\uX}, I_{\uX'})$, 
and let $L,R$ be Markov embeddings on
$\Ell{2}(\uX)$ and $\Ell{2}(\uX')$, respectively. Then there is a Markov
embedding $T$ on $\Ell{2}(\uZ)$ with $ I_{\uX} L = T I_{\uX}$ and
$    I_{\uX'} R = T I_{\uX'}$ if and only if $ R^* P L = P$.  
In this case, $T$ is unique and satisfies
\begin{equation}\label{join.eq.mult}
    T( f \tensor g) = Lf \tensor Rg \qquad 
(f\in \Ell{\infty}(\uX),\, g\in \Ell{\infty}(\uX')).
\end{equation}
\end{lemma}

\begin{proof}
Since $T$ is supposed to be a Markov embedding and hence multiplicative on 
$\Ell{\infty}$, the commutation conditions
$ I_{\uX} L = T I_{\uX}$ and $I_{\uX'} R = T I_{\uX'}$ are equivalent to 
\eqref{join.eq.mult}. Since the products $f\tensor g$ span a dense subspace,
uniqueness is clear. The computation
\[    \int_{\uX'} (R^*PLf) g =  \int_{\uX'} (PLf) \, (Rg)_{\uX'} = \int_\uZ Lf \tensor Rg
\]
shows that the condition $R^*PL = P$ is necessary for the existence of $T$. 
Conversely, suppose that this condition holds. Then, since 
$R,L$ are multiplicative and real, 
\begin{align*}
    \Bigl\| & \sum_{j=1}^n Lf_j \tensor Rg_j \Bigr\|^2_{\Ell{2}(\uZ)}
  = \sum_{j,k= 1}^n \int_\uZ (Lf_j \konj{Lf_k}) \tensor (Rg_j\konj{Rg_k})
= \sum_{j,k= 1}^n \int_\uZ L(f_j\konj{f_k}) \tensor R(g_j\konj{g_k})
\\ & = \sum_{j,k= 1}^n \int_{\uX'} R^*PL (f_j\konj{f_k}) \cdot g_j\konj{g_k}
= \sum_{j,k= 1}^n \int_{\uX'} P(f_j\konj{f_k}) \cdot g_j\konj{g_k}
= \dots = \Bigl\| \sum_{j=1}^n f_j \tensor g_j \Bigr\|^2_{\Ell{2}(\uZ)}.
\end{align*}
This shows that through \eqref{join.eq.mult} an isometric operator $T$ on $\Ell{2}(\uZ)$
is defined. It is multiplicative on $\Ell{\infty}(\uX) \tensor
\Ell{\infty}(\uX')$ and hence a Markov embedding.   
\end{proof}

If, in the situation of the previous lemma, $R$ is a Markov isomorphism, then 
$P = R^*PL$ is equivalent with $PL = RP$. Hence, the following corollary holds.  

\begin{corollary}
Let $(\uX; L)$ and $(\uX'; R)$ be measure-preserving $G$-systems. 
Let $(\uZ, I_{\uX}, I_{\uX'})$ be a coupling such that its coupling operator
$P$ intertwines the dynamics, i.e., satisfies $PL_t = R_tP$ for all $t\in G$.
Then there is a unique $G$-dynamics $T$ on $\uZ$ such that 
\[  I_{\uX}\colon (\uX; L) \to (\uZ;T) \quad \text{and}\quad  
      I_{\uX'}\colon (\uX'; R) \to (\uZ;T)
\]
are extensions of $G$-systems. The $G$-dynamics $T$ is characterized by 
\[        T_t( f \tensor g) = L_tf \tensor R_tg \qquad 
(f\in \Ell{\infty}(\uX),\, g\in \Ell{\infty}(\uX'),\, t\in G).
\] 
\end{corollary}

The system $(\uZ; T)$ is called the {\emdf joining} of the systems
$(\uX; L)$ and $(\uX'; R)$ via the coupling operator $P$.

\medskip
Let $J_{\uX} \colon (\uY;S) \to (\uX;L)$ and $J_{\uX'} \colon (\uY;S) \to
(\uX';R)$ be extensions and $\uZ = \uX\times_\uY \uX'$ the relatively
independent coupling. Then the coupling operator is $P = J_{\uX'}^{\phantom{*}}J_{\uX}^*$
and hence intertwines the dynamics. It follows that 
there is a unique $G$-dynamics, denoted by $L \times_{\uY} R$, 
 on $\uX\times_\uY \uX'$  such that the coupling becomes a joining of
the original systems. 

The system $(\uX\times_\uY \uX'; L \times_\uY R)$ is called the
{\emdf relatively independent joining} of the extensions $(J_{\uX}, J_{\uX'})$.

\begin{examples}\mbox{}
\begin{enumerate}[(1)]
\item 
\label{product}
Let $(\uX;L)$ and $(\uX';R)$ be  measure-preserving systems. Then the relatively independent joining of the trivial 
extensions $(\{\mathrm{pt}\};\mathrm{Id}) \to (\uX;L), (\uX'; R)$ is 
simply the product system $(\uX \times \uX';L \times R)$ 
with the natural embeddings $(\uX;L), (\uX';R) \to (\uX
\times \uX' ;L \times R)$.

\item 
Let $(\uX;L)$, $(\uY;M)$ and $(\uZ;R)$ be measure-preserving systems. We consider the product systems $(\uX \times \uY;L 
\times M)$ and $(\uY \times \uZ;M \times R)$ and the canonical extensions $J_1 \colon (\uY;M) \to  (\uX \times \uY;L \times M)$ 
and $J_2 \colon (\uY;M) \to (\uY \times \uZ;M \times R)$. Then the relatively independent joining of $J_1$ and $J_2$ is given by 
the system $(\uX \times \uY \times \uZ;L \times M \times R)$.
\end{enumerate}
\end{examples}

\medskip

Finally, we realize that the KH-dynamical system associated with
the extension $(\uY; S) \to (\uX\times_\uY \uX'; L \times_\uY R)$ is precisely
the tensor product system arising from the KH-dynamical systems associated
with the original extensions.

\begin{proposition}\label{joiningvstensor}
Let $J_{\uX} \colon (\uY;S) \to (\uX;L)$ and $J_{\uX'} \colon (\uY;S) \to
(\uX';R)$ be extensions. Then there is a unique isomorphism
of KH-dynamical systems
\begin{align*}     W\colon (\uL^2(\uX|\uY) \otimes  \uL^2(\uX'|\uY);L\tensor R) \quad
  \cong \quad 
 (\uL^2(\uX \times_{\uY} \uX'|\uY);L \times_{\uY} R)
\end{align*}
with $W( f\tensor g) = (I_\uX f) (I_{\uX'} g)$ for all $f \in \mathrm{L}^2(\uX|\uY)$ and $g \in \mathrm{L}^2(\uX'|\uY)$.
\end{proposition}

\begin{proof}
For the time being, we need to distinguish $f\tensor g \in
\Ell{2}(\uX|\uY)\tensor  \Ell{2}(\uX'|\uY)$
from  $f\tensor_\uY g \coloneq (I_\uX f) (I_{\uX'} g) \in \Ell{2}(\uX
\times_\uY \uX')$. 
Uniqueness of $W$ is clear.
 
To show existence we start with the (obviously well-defined) operator
\[ W\colon \Ell{\infty}(\uX) \tensor_{\alg}  \Ell{\infty}(\uX') 
\to \Ell{2}(\uX \times_\uY \uX'| \uY), \qquad W(f\tensor g) \coloneq f\tensor_\uY g.
\]
By employing the definition of the relatively independent coupling, the
identity
\[    (f \tensor g| u \tensor v)_\uY  = (f \tensor_\uY g | u \tensor_\uY v)
\]
is easily established.    Since $\Ell{\infty}(\uX)$ is,  by 
\cref{ordervsl2} (iii), order-dense in $\Ell{2}(\uX|\uY)$ (and the analogous
statement is true for $\uX'$), $W$ extends to an isometric homomorphism
\[ W \colon  \uL^2(\uX|\uY) \otimes  \uL^2(\uX'|\uY)
\to \Ell{2}(\uX \times_\uY \uX'| \uY)
\]
of KH-modules. 

It is clear that $W$ intertwines the dynamics. To show that $W$ is surjective, it
suffices to prove that the range of $W$ is order-dense in $\uL^2(\uX_1 \times_{\uY}
\uX_2|\uY)$. However, this is an immediate consequence of  \cref{ordervsl2} (iii).

	Finally, take $f \in \mathrm{L}^2(\uX|\uY)$ and $g \in \mathrm{L}^2(\uX'|\uY)$. We then find sequences $(f_n)_{n \in \N}$ in $\mathrm{L}^\infty(\uX)$ and $(g_n)_{n \in \N}$ in $\mathrm{L}^\infty(\uX')$ order-converging to $f$ and $g$, respectively. Then $\olim_{n \rightarrow \infty} f_n \otimes g_n = f \otimes g$ in $\mathrm{L}^2(\uX|\uY) \otimes \mathrm{L}^2(\uX'|\uY)$ and hence $((I_{\uX}f_n) (I_{\uX'}g_n)) _{n \in \N}$ order-converges to $h \coloneqq W(f \otimes g) \in \mathrm{L}^2(\uX \times_{\uY} \uX'|\uY)$. In particular, $\lim_{n \rightarrow \infty} (I_{\uX}f_n) (I_{\uX'}g_n) = h$ in $\mathrm{L}^1(\uX \times_{\uY} \uX)$ by \cref{ordervsl2} (ii). On the other hand, since $\lim_{n \rightarrow \infty} f_n = f$ in $\mathrm{L}^2(\uX)$ and $\lim_{n \rightarrow \infty} g_n = g$ in $\mathrm{L}^2(\uX')$ again by \cref{ordervsl2} (ii), we have $\lim_{n \rightarrow \infty} (I_{\uX}f_n)(I_{\uX'}g_n) = (I_{\uX}f)(I_{\uX'}g)$ in $\mathrm{L}^1(\uX \times_{\uY} \uX')$. This shows $W(f \otimes g) = h = (I_{\uX}f)(I_{\uX'}g)$.
\end{proof}

\section{Splitting and Dichotomy for Extensions of Measure-Preserving Systems}\label{c.dicmps}

In this chapter we apply the Decomposition Theorem \ref{mainthm}
to extensions of measure-preserving $G$-systems. Unless otherwise stated, 
$G$ is   an arbitrary but fixed   group, and $J\colon (\uY;S) \to (\uX;T)$ is a fixed 
extension of a measure-preserving $G$-systems.

\medskip
We usually suppress reference to $J$ and consider $\Ell{2}(\uY)$ as a closed
subspace  of $\Ell{2}(\uX)$.  (Nevertheless, we distinguish notationally
the dynamics $S$ on $\uY$ and $T$ on $\uX$.) Furthermore, we 
use the prefix ``$\uY$-'' whenever we 
we use notions pertaining to the associated KH-dynamical system on
$\Ell{2}(\uX|\uY)$, so that, for instance, we speak of
$\uY$-suborthonormal sets or $\uY$-homogeneous finite-rank submodules.

\begin{remark}\label{L2-L2cond}
The following observations will be helpful when
transferring  results from $\Ell{2}(\uX|\uY)$
to the whole of $\Ell{2}(\uX)$. 
\begin{enumerate}[(1)]
\item 
If $f\in \Ell{2}(\uX)$, then multiplying by $\rho_n \coloneq \car_{[\abs{f}_{\uY} \le n]} \in \Ell{\infty}(\uY)$
yields $\rho_n f \in \Ell{2}(\uX|\uY)$. As $\rho_n \nearrow \car$, we obtain
$\rho_n f \to f$ in $\Ell{2}(\uX)$. 
\item 
If $M\subseteq \Ell{2}(\uX|\uY)$ is a $\Ell{\infty}(\uY)$-submodule,
then $\cl_{\Ell{2}} M \cap \Ell{2}(\uX|\uY) = \ocl(M)$  (by Lemma
\ref{ordervsl2}). In particular, $f \in \cl_{\Ell{2}}(M)$ if and only if
$\rho_n f\in \ocl(M)$ for all $n\in \N$.
\item 
If $f,g\in \Ell{2}(\uX)$, then $f\perp_\uY g $ is equivalent to 
$f \perp_{\Ell{2}} \Ell{\infty}(\uY) g$. 
\end{enumerate}
\end{remark}

\medskip

\subsection{The Fixed Space in the Relatively Independent Joining}

Let $\calB \subseteq \Ell{2}(\uX)$ be a finite set. Similar as in  Chapter \ref{c.mainthm}
we write
\[  u_\calB \coloneq \sum_{g\in \calB} g\tensor \konj{g} \in
  \Ell{1}(\uX\times_\uY \uX).
\]
In general, $u_\calB$ may be  just an element of $\Ell{1}$, but
if $\calB \subseteq \Ell{2}(\uX|\uY)$, then $u_\calB \in \Ell{2}(\uX\times_{\uY}
\uX|\uY)$ and if $\calB  \subseteq \Ell{\infty}(\uX)$, then $u_\calB \in
\Ell{\infty}(\uX \times_\uY \uX)$.

Note that if  $f\in \fix(T\times_\uY T)$, then $\E_\uY\abs{f} \in \fix(S)$, and hence
by Remark \ref{L2-L2cond}(1)
\[ \fix(T\times_\uY T) = \cl_{\Ell{2}} \Bigl( 
\Ell{2}(\uX\times_\uY \uX|\uY) \cap \fix(T\times_\uY T)\Bigr).
\]
Therefore, Theorem \ref{mainthm} together with Remark \ref{L2-L2cond}(2) 
imply that $\fix(T \times_\uY T)$ is generated (in the $\Ell{2}$-sense) 
by those of its elements that have the form $\lambda u_\calB$,
where $\lambda \in \fix(S) \cap \Ell{\infty}(\uY)$ and $\calB\subseteq
\Ell{2}(\uX|\uY)$ is a finite, homogeneous, suborthonormal set 
generating a $T$-invariant $\Ell{\infty}(\uY)$ submodule. 
 Unfortunately, it seems impossible to add the requirement $\calB \subseteq
\Ell{\infty}(\uX)$ here. The best we can say is the following:

\begin{proposition}\label{fixgenLinfty}
Let $J\colon (\uY;S) \to (\uX;T)$ be an extension of measure-preserving 
$G$-systems. Then, in $\Ell{2}(\uX\times_\uY \uX)$,  
\[ \fix(T \times_\uY T) = \cl_{\Ell{2}}\spann\bigl\{ 
u_\calB  \, |\, \calB \subseteq \Ell{\infty}(\uX)\, \text{\rm finite},\,
u_\calB \in \fix(T\times_\uY T) \bigr\}. 
\]
\end{proposition}

In order to prove  \cref{fixgenLinfty} we need the following lemma.

\begin{lemma}\label{sons-linfty}
Let $J\colon (\uY;S) \to (\uX;T)$ be an extension of measure-preserving 
$G$-systems and let $e_1, \dots, e_n$ be a $\uY$-suborthonormal
system in $\Ell{2}(\uX|\uY)$ satifsying that $\spann_{\Ell{\infty}(\uY)}\{ e_1, \dots,
e_n\}$ is $T$-invariant. Then\footnote{Recall that  $|\cdot|$ denotes the usual modulus mapping on 
$\uL^2(\uX)$, see \cref{innerproducts}.} 
$\sum_{j=1}^n \abs{e_j}^2 \in \fix(T)$.
\end{lemma}

\begin{proof}
Define $e \coloneq \sum_{j=1}^n e_j \tensor \konj{e_j} \in \fix(T
\times_{\uY} T)$. Then, within $\Ell{1}(\uX \times_\uY \uX)$,  
\[  \abs{e}^2 = \sum_{j,k=1}^n e_j\konj{e_k} \tensor \konj{e_j}e_k \in
\fix(T\times_\uY T). 
\]
Hence, also $\E_\uX \abs{e}^2 \in \fix(T)$, where $\E_\uX$ denotes the
conditional expectation onto the second factor. It is easy to see that 
$\E_\uX(f \tensor g) = (\E_\uY f) g$ in general, hence
\[  \fix(T) \ni \E_\uX \abs{e}^2 = \sum_{j,k=1}^n \E_\uY(e_j\konj{e_k})
\konj{e_j}e_k
= \sum_{j=1}^n \abs{e_j}^2_{\uY} \abs{e_j}^2 = 
\sum_{j=1}^n   \abs{\abs{e_j}_{\uY} e_j}^2  = \sum_{j=1}^n \abs{e_j}^2,
\]
cf.{} Lemma \ref{supp.l.supp}(iii). 
\end{proof}

\begin{proof}[Proof of \cref{fixgenLinfty}]
Let $\calB = \{ e_1, \dots, e_d\}$ be a suborthonormal subset of $\Ell{2}(\uX|\uY)$
such that $u_\calB \in \fix(T\times_\uY T)$ and let $\lambda  \in \fix(S)\cap
\Ell{\infty}(\uY)$. We may assume that $\lambda \geq 0$.  Let $e \coloneq \sum_{j=1}^d \abs{e_j}^2 \in \fix(T)$ by Lemma \ref{sons-linfty} and define
$\eta_n\coloneq \frac{\sqrt{e} \wedge n}{\sqrt{e} + \frac{1}{n}}$. Then $0\le \eta_n \le 1$, 
$\eta_n  \in \fix(T)$ and  $\eta_n e_j \to e_j$ in $\Ell{2}$ as $n\to
\infty$ for each $j=1, \dots, d$. Define 
\[ \calB_n \coloneq\sqrt{\lambda} \eta_n  \calB = \{ \sqrt{\lambda} \eta_n e_1, \dots
\sqrt{\lambda} \eta_n e_d\} \subseteq \Ell{\infty}(\uX).
\] 
Then $\calB_n$ is finite,  
$u_{\calB_n} \in \fix(T \times_\uY T)$ for each $n\in \N$ and $u_{\calB_n} \to
\lambda u_\calB$ in $\Ell{2}$-norm as $n\to \infty$.  
\end{proof}

\begin{remark}
 Lemma   \ref{sons-linfty} is due to
Furstenberg \cite[p.231]{Furs1977} and reproduced in \cite[Thm.~9.13]{Glas}. 
However, Glasner seems to claim that the scaled elements $\eta_n e_j$ 
still form a  $\uY$-suborthonormal system, which, as far as we can see, need
not be the case.\footnote{Glasner uses different scaling functions
  $\eta_n$, but that is inessential here.}  
\end{remark}

\medskip

\subsection{Kronecker Subspace}

Theorem \ref{mainthm} states that the  associated KH-dynamical system
on $\Ell{2}(\uX|\uY)$ induces a $\uY$-orthogonal decomposition into 
KH-submodules
\[   \Ell{2}(\uX|\uY) = \Ell{2}(\uX|\uY)_{\ds}  \oplus \Ell{2}(\uX|\uY)_{\wm},
\]
into the discrete spectrum part and the weakly mixing part. 
The space
\[    \scrE(\uX|\uY) \coloneq \cl_{\Ell{2}} \Ell{2}(\uX|\uY)_{\ds}
\]
is called the {\emdf (relative) 
discrete spectrum part} or the {\emdf Kronecker subspace}  of the extension.

\begin{proposition}\label{Kronecker}
Let $J\colon (\uY;S) \to (\uX;T)$ be an extension of measure-preserving 
$G$-systems. Then $\scrE(\uX|\uY)$ 
is the $\Ell{2}$-closure of each of the following sets:
\begin{enumerate}[(1)]
\item $\dps
\bigcup \{ M \, |\, M \text{ \rm  finitely-generated, $T$-invariant KH-submodule of
                 $\Ell{2}(\uX|\uY)$ }\}$;

\item $\dps \bigcup
\{M\, |\, M \text{ \rm  finitely-generated, $T$-invariant $\Ell{\infty}(\uY)$-submodule of
                 $\Ell{2}(\uX)$  }\}$;
\item $\dps \bigcup
\{ M\, |\,  M \text{ \rm  finitely-generated, $T$-invariant $\Ell{\infty}(\uY)$-submodule of
                 $\Ell{\infty}(\uX)$ }\}$;
\item $\dps  \bigcup\{ \spann_{\Ell{\infty}(\uY)}\calB\, |\, \calB \subseteq
  \Ell{\infty}(\uX)\, \text{ \rm finite},\,  
u_\calB \in \fix(T \times_\uY T)\}$.  
\end{enumerate}
Moreover, $\scrE(\uX|\uY)$ is a $T$-invariant closed unital sublattice of
$\Ell{2}(\uX)$.  
\end{proposition}

\begin{proof}
Let us denote the sets listed in (1)--(4) by $E_1$--$E_4$. Then 
$\cl_{\Ell{2}}E_1 =  \scrE(\uX|\uY)$  by definition of $\scrE(\uX|\uY)$ and Remark \ref{L2-L2cond}(2).
The inclusions  $E_4 \subseteq E_3 \subseteq E_2$ are clear. We shall show
$E_2 \subseteq \cl_{\Ell{2}}E_1$ and $E_1 \subseteq \cl_{\Ell{2}}E_4$.  

\smallskip
\noi
For the proof of the first inclusion, suppose that $M$ is a $T$-invariant
$\Ell{\infty}(\uY)$-submodule of $\Ell{2}(\uX)$ generated by the functions
$f_1, \dots , f_n$. Let $\lambda \coloneq \car + \sum_{j=1}^n \abs{f_j}_\uY \in
\Ell{2}(\uY)$ and $g_j \coloneq \frac{1}{\lambda} f_j \in \Ell{2}(\uX|\uY)$ for $1,
\dots, n$.   We claim that $N \coloneq \spann_{\Ell{\infty}(\uY)}\{ g_1, \dots, g_n\}$
is $T$-invariant.  Fix $t\in G$ and write
\[   T_t f_j = \sum_{k=1}^n a_{jk} f_k,\quad T_{t^{-1}} f_j = \sum_{k=1}^n
  b_{jk} f_k
\]
for $\Ell{\infty}(\uY)$-elements $a_{jk}, b_{jk}$. Then
\begin{align*}
T_t g_i = S_t(\tfrac{1}{\lambda}) \sum_k a_{jk} f_k 
        = \sum_k a_{jk}   \lambda  S_t(\tfrac{1}{\lambda}) g_k
        = \sum_k a_{jk}   S_t \bigl(\tfrac{S_t^{-1}\lambda}{\lambda}\bigr) g_k.  
\end{align*}
Next, note that 
\[   S_t^{-1}\lambda = \car + \sum_{j=1}^n \abs{T_t^{-1} f_j}_\uY
\le \car + \sum_{j=1}^n \sum_{k=1}^n \abs{b_{jk}} \abs{f_k}_{\uY} 
\le c \lambda
\]
for some real number $c > 0$. This proves the claim. 
Finally, note that for each $j=1, \dots, n$ 
\[ f_j = \lambda g_j \in \cl_{\Ell{2}}
\spann_{\Ell{\infty}(\uY)}\{ g_1, \dots, g_n\} = \cl_{\Ell{2}} \ocl(N)
\subseteq \cl_{\Ell{2}} (E_1). 
\]
Hence,  $M \subseteq \cl_{\Ell{2}}(E_1)$ as desired.

\smallskip
\noi
For the second inclusion let $M\subseteq \Ell{2}(\uX|\uY)$ be a
$T$-invariant, finite-rank KH-submodule. Pick a $\uY$-suborthonormal basis 
$\calB = \{ e_1,
\dots, e_d\}$ of $M$ and find, as in the proof of Proposition \ref{fixgenLinfty},
functions $\eta_n \in \fix(T)$ with $0 \le \eta_n \le 1$ and $\calB_n \coloneq
\eta_n \calB \subseteq  \Ell{\infty}(\uX)$ and  
$\eta_n e_j \to e_j$ in $\Ell{2}$-norm for each $j =1, \dots, d$. Then 
$u_{\calB_n} \in \fix(T\times_\uY T)$ and 
\[ M \subseteq \cl_{\Ell{2}} \bigcup_{n \in
  \N} \spann_{\Ell{\infty}(\uY)}\calB_n \subseteq \cl_{\Ell{2}} E_4.
\]
Finally, let $M$ be the norm-closure in $\Ell{\infty}(\uX)$ 
of the union of all finitely-generated
$T$-invariant $\Ell{\infty}(\uY)$-submodules of $\uL^\infty(\uX)$. 
Then $M$ is a unital $\uC^*$-subalgebra of $\uL^\infty(\uX)$ and therefore a 
unital sublattice of $\uL^\infty(\uX)$ (see \cite[Thm.~7.23]{EFHN2015}). This
implies that its $\uL^2$-closure $\scrE(\uX|\uY)$ is a unital Banach sublattice of $\uL^2(\uX)$.
\end{proof}

By \cref{Kronecker} and general
theory \cite[Prop.~13.19]{EFHN2015},  $\scrE(\uX|\uY) \cong
\Ell{2}(\uZ)$ for some probability space $\uZ$ (determined up to a natural
isomorphism). Moreover, as $\scrE(\uX|\uY)$ is $T$-invariant, 
by restriction we obtain a system $(\uZ; T)$.  

We write  $\Kro(\uX|\uY)  \coloneq \uZ$ and call 
the system $(\Kro(\uX|\uY); T)$ the {\emdf relative Kronecker factor} of 
the extension.   The original
extension then factors as
\[    (\uY; S) \to (\Kro(\uX|\uY); T) \to (\uX; T)
\]
which amounts to a sequence 
\[     \Ell{2}(\uY)  \to \Ell{2}( \Kro(\uX|\uY) ) \to  \Ell{2}(\uX)
\]
of intertwining Markov embeddings
on the level of function spaces.

\begin{definition}
An extension $J\colon (\uY;S) \to (\uX;T)$  of measure-preserving systems 
has {\emdf relative discrete spectrum} (or: is a discrete spectrum extension)
if $\scrE(\uX|\uY) = \Ell{2}(\uX)$.
\end{definition}

If $(\uX;T)$ is a measure-preserving system, then the trivial extension $J \colon (\{\mathrm{pt}\};\mathrm{Id}) \to (\uX;T)$ 
has relative discrete spectrum if and only if the system $(\uX;T)$ has discrete
spectrum, i.e., the finite-dimensional invariant  subspaces of $\uL^\infty(\uX)$
are dense in $\uL^2(\uX)$.

\begin{example}\label{skewtorus}
  Consider the torus $\T \coloneqq \{y \in \C\mid |y| = 1\}$ equipped with the Haar measure. For fixed $a \in \T$ we consider the 
measure-preserving system $(\uY;S)$ of the group $\Z$ induced by the homeomorphism $\T \to \T, \, y \mapsto ay$. The product 
space $\uX = \uY \times \uY$ equipped with the action induced by
    \begin{align*}
      \T^2 \to \T^2, \quad (y,z) \mapsto (ay,yz)
    \end{align*}
  is called the {\emdf skew torus} $(\uX;T)$. The projection onto the first component induces an extension $J \colon (\uY;S) \to 
(\uX;T)$. For every $k \in \Z$ the continuous function
    \begin{align*}
      f_k \colon \T^2 \to \C, \quad (y,z) \mapsto z^k
    \end{align*}
  generates an invariant submodule $\uL^\infty(\uY)f_k \subset \uL^\infty(\uX)$ and these submodules are total in 
$\uL^2(\uX)$. Thus, $J \colon (\uY;S) \to (\uX;T)$ has relative discrete spectrum. This example can be generalized to 
compact group extensions (cf.\ \cite{Zimm1976}, \cite[Sec.~6.2]{Furstenberg1981}, \cite[Sec.s 4 and 5]{Elli1987} and 
\cite{LTW2002}).
\end{example}

\medskip

\subsection{Orthogonal Complement of the Relative Kronecker Subspace}\label{s.Kronecker}

\begin{proposition}\label{Kronecker-orth}
Let $J\colon (\uY;S) \to (\uX;T)$ be an extension of measure-preserving 
$G$-systems. Then for $f\in \Ell{2}(\uX)$ the following assertions
are equivalent.
\begin{enumerate}[(a)]
\item $f \perp_{\Ell{2}} \scrE(\uX|\uY)$.
\item $\E_\uY( (f \tensor \konj{f})  h) = 0$\quad  for each $h\in 
\fix(T \times_\uY T) \cap \Ell{\infty}$.

\item $f\tensor \konj{f} \perp  
\fix(T \times_\uY T) \cap \Ell{\infty}$. 

\item $0 \in \cls{\mathrm{co}} \Bigl\{ T_tf \tensor T_t
  \konj{f} \, \big|\,
  t\in G \Bigr\}$ \quad (in $\Ell{1}$).
\item 
$\displaystyle  \inf_{t\in G} \max_{g\in F} \norm{\E_\uY\bigl((T_tf)
  g\bigr)}_{\Ell{2}} = 0$ \quad for each finite $F\subseteq \Ell{\infty}(\uX)$.  
\end{enumerate}
\end{proposition}

\begin{proof}
The conditions (a) and  (b) are invariant under multiplication with
$\Ell{\infty}(\uY)$ and $\Ell{2}(\uX)$-closed.  Hence, for the proof of
  \enquote{(a)$\Leftrightarrow$(b)}   we may suppose
that  $f\in \Ell{2}(\uX|\uY)$ (cf.\ Remark \ref{L2-L2cond}). 
In  this case---still  by Remark \ref{L2-L2cond}---(a) is  equivalent to
\begin{enumerate}
\item[(a')] $f \perp_{\uY} \Ell{2}(\uX|\uY)_{\ds}$;
\end{enumerate}
and  (b) is equivalent to
\begin{enumerate}
\item[(b')] $f \tensor \konj{f} \perp_{\uY}  \fix(T\times_{\uY} T)\cap
  \Ell{2}(\uX\times_\uY \uX|\uY)$;
\end{enumerate}
(apply  Lemma \ref{ordervsl2}).  
Since (a') and (b') are equivalent by Theorem \ref{mainthm}, we have
established the equivalence (a)$\gdw$(b).

\smallskip
\noi
Clearly, (b) implies (c). To prove that (c) implies (d) we first suppose that
$f\tensor \konj{f} \in \Ell{2}$. Then (c)  is equivalent to 
$P(f \tensor \konj{f}) = 0$, where $P$ is the
orthogonal projection in $\Ell{2}(\uX \times_\uY \uX)$ onto 
$\fix(T \times_\uY T)$. By the mean ergodic theorem
for contraction semigroups on Hilbert spaces (Birkhoff--Alaoglu theorem 
\cite[Thm.~8.32]{EFHN2015})  
\[ P(f \tensor \konj{f}) = 0 \quad \gdw\quad 
0 \in \cls{\mathrm{co}} \Bigl\{T_tf \tensor T_t
  \konj{f} \, \big|\,
  t\in G \Bigr\}.
\]
This shows that (c) implies (d) if $f\tensor \konj{f} \in \Ell{2}$. 

To see that this still holds if  $f \tensor \konj{f} \in \Ell{1}$ one needs to realize
that $P$ is a Markov operator and hence extends continuously to $\Ell{1}$, and this
extension---again denoted by $P$---is the mean ergodic projection (in the sense of
\cite[Def.~8.31]{EFHN2015}
of the group $(T_t \times_\uY T_t)_{t\in G}$ acting on $\Ell{1}$. From (c) then follows that  $P(f\tensor
\konj{f}) \perp \fix(T\times_\uY T) \cap \Ell{\infty}$, and since $P(f\tensor
\konj{f}) \in \fix(T\times_\uY T)$ it follows that $P(f\tensor
\konj{f}) = 0$, i.e., (d). 

\smallskip
\noi
Next, observe that  $\dps \abs{ \E_\uY\bigl((T_tf) g\bigr)}^2 = 
    \E_\uY\Bigl(    T_t(f \tensor \konj{f}) \cdot (g
    \tensor \konj{g}) \Bigr)$ 
for $g\in \Ell{\infty}(\uX)$. Hence, (d) implies 
\[ 0 \in \cls{\mathrm{co}}\bigl\{  \sum_{g\in F} \norm{\E_\uY\bigl((T_tf)
  g\bigr)}_{\Ell{2}}^2 \, \big|\,  t\in G \Bigr\}
\]
for each finite set $F\subseteq \Ell{\infty}(\uX)$,
and this is equivalent to (e).

\smallskip
\noi
Finally, suppose that (e) holds. 
Let $\calB \subseteq \Ell{\infty}(\uX)$ be
finite such that $u_\calB = \sum_{g\in \calB} g \tensor \konj{g} \in \fix(T\times_\uY T)$.
Then for each $t\in G$
\begin{align*}
\sum_{g \in \calB}&\abs{\E_\uY((T_tf)g)}^2 = \sum_{g\in \calB} \E_\uY \bigl(
  T_t(f\tensor \konj{f}) \cdot (g \tensor \konj{g})\bigr) 
= \E_\uY\bigl( T_t(f\tensor \konj{f})\cdot u_\calB\bigr)
\\ & = \E_\uY\bigl( T_t(f\tensor \konj{f})\cdot T_t u_\calB\bigr)
= S_t \sum_{g \in \calB}\abs{\E_\uY(fg)}^2.
\end{align*}
Integrating yields
\[  \sum_{g \in \calB} \norm{\E_\uY((T_tf)g)}_{\Ell{2}}^2 = 
  \sum_{g \in \calB}\norm{\E_\uY(fg)}_{\Ell{2}}^2
\]
for each $t\in G$. Hence, (e) implies  $f \perp_\uY \konj{g}$ for all $g\in
\calB$. Replacing
$\calB$ by $\calB'\coloneq \{ \konj{g}\,|\, g\in \calB\}$ and applying the
characterization of $\scrE(\uX|\uY)$ from \cref{Kronecker}(4) yields (a).
\end{proof}

For amenable groups one can add to (a)--(e) another equivalent
statement  in terms of an 
aymptotic condition on ergodic nets. Recall that a group $G$ is {\emdf (discretely) amenable} if 
$G$ has a {\emdf (left) F{\o}lner net} $(N_\alpha)_{\alpha }$, i.e., $N_\alpha \subset G$ is a finite 
subset of $G$ for every $\alpha$ such that
  \begin{align*}
    \lim_\alpha \frac{|tN_\alpha \Delta N_\alpha|}{|N_\alpha|} = 0 \textrm{ for every } t \in G.
  \end{align*}

Every abelian group is amenable. In case of $G = \Z$, we obtain a
F{\o}lner net $(N_k)_{k \in \N}$ by setting $N_k \coloneqq \{0,\dots,k-1\}$ for
$k \in \N$. Given a representation of $G$ on a Hilbert space, any F{\o}lner net
$(N_\alpha)_{\alpha}$ defines an ergodic net converging to the mean ergodic
projection (cf.\ \cite[Thm.~1.7]{Schr2013}). This leads to the following
extension of \cref{Kronecker-orth}.

\begin{proposition}\label{Kronecker-orth-amen}
Let $G$ be a group with F{\o}lner net $(N_\alpha)_{\alpha}$, and let
$J\colon (\uY;S) \to (\uX;T)$ be an extension of measure-preserving 
$G$-systems. Then for $f\in \Ell{2}(\uX)$ the following assertions
are equivalent.
\begin{enumerate}[(a)]
\item $f \perp_{\Ell{2}} \scrE(\uX|\uY)$.
\item[(f)] $\displaystyle \lim_{\alpha} \frac{1}{\abs{N_\alpha}} \sum_{t\in
    N_\alpha}  |\E_\uY\bigl((T_tf)
  g\bigr)|^2 = 0$ \, in $\mathrm{L}^1(\uY)$ for each $g\in \Ell{\infty}(\uX)$.
\item[(g)] $\displaystyle \lim_{\alpha} \frac{1}{\abs{N_\alpha}} \sum_{t\in
    N_\alpha}  |\E_\uY\bigl((T_tf)
  \konj{f}\bigr)| = 0$ \, in $\mathrm{L}^1(\uY)$. 
\end{enumerate}
If $f \in \mathrm{L}^2(\uX|\uY)$, then one can add the following assertion to these equivalences.
\begin{enumerate}[(h)]
\item[(h)] $\displaystyle \lim_{\alpha} \frac{1}{\abs{N_\alpha}} \sum_{t\in
    N_\alpha}  |\E_\uY\bigl((T_tf)
  \konj{f}\bigr)|^2 = 0$ \, in $\mathrm{L}^1(\uY)$.  
\end{enumerate}
\end{proposition}

\begin{proof}
Note first that if $b \colon G \rightarrow \mathrm{L}^2(\uY)_+$ is any function, we have
	\begin{align*}
		\frac{1}{\abs{N_\alpha}}\sum_{t \in N_\alpha} b(t) \leq \left(\frac{1}{\abs{N_\alpha}}\sum_{t \in N_\alpha} b(t)^2\right)^{\frac{1}{2}} \quad \textrm{ in } \mathrm{L}^2(\uY) \textrm{ for every } \alpha, 
	\end{align*}
	and hence the implication
		\begin{align*}
			 \lim_\alpha \frac{1}{\abs{N_\alpha}}\sum_{t \in N_\alpha} b(t)^2 = 0 \textrm{ in } \mathrm{L}^1(\uY) \quad \Rightarrow \quad \lim_\alpha \frac{1}{\abs{N_\alpha}}\sum_{t \in N_\alpha} b(t) = 0  \textrm{ in } \mathrm{L}^2(\uY).
		\end{align*}

\smallskip
\noi
Now by \cref{Kronecker-orth}, (a) is equivalent to $P(f \tensor \konj{f}) = 0$,
i.e.,
\begin{enumerate}[(c')]
\item $\displaystyle \lim_{\alpha} \frac{1}{\abs{N_\alpha}} \sum_{t\in
    N_\alpha} (T_tf \tensor \konj{T_tf}) =0$.
\end{enumerate}
Multiplying with $g\tensor \konj{g}$ and taking $\E_\uY$ we obtain (f).

\smallskip
\noi
We show that (f) implies (g). By the preliminary remark, (f) yields in particular 
		\begin{align*}
			\lim_{\alpha} \frac{1}{\abs{N_\alpha}} \sum_{t\in
    N_\alpha}  \abs{\E_\uY((T_tf)g)} =0 \quad
  \textrm{ in } \Ell{1}(\uY)
  \quad \textrm{ for each } g\in \Ell{\infty}(\uX).
		\end{align*}
	With $f_m \coloneqq \car_{|f| \leq m}f \in \mathrm{L}^\infty(\uX)$ for $m \in \N$ we obtain
		\begin{align*}
			\frac{1}{\abs{N_\alpha}}\sum_{t \in N_\alpha} |\E_{\uY}((T_tf)g)| \leq \frac{1}{\abs{N_\alpha}}\sum_{t \in N_\alpha} |\E_{\uY}((T_tf)f_m)| +  \frac{1}{\abs{N_\alpha}}\sum_{t \in N_\alpha} \E_{\uY}(|T_tf||f-f_m|)
		\end{align*}
	and consequently
		\begin{align*}
			 \biggl\|\frac{1}{\abs{N_\alpha}}\sum_{t \in N_\alpha} |\E_{\uY}((T_tf)g)|\biggr\|_{\mathrm{L}^1} &\leq \biggl\|\frac{1}{\abs{N_\alpha}}\sum_{t \in N_\alpha} |\E_{\uY}((T_tf)f_m)|\biggr\|_{\mathrm{L}^1} +  \frac{1}{\abs{N_\alpha}}\sum_{t \in N_\alpha} \|f\|_{\mathrm{L}^2} \|f-f_m\|_{\mathrm{L}^2} \\
			&= \biggl\|\frac{1}{\abs{N_\alpha}}\sum_{t \in N_\alpha} |\E_{\uY}((T_tf)f_m)|\biggr\|_{\mathrm{L}^1} + \|f-f_m\|_{\mathrm{L}^2}
		\end{align*}
	for every $\alpha$. This shows (g).

	\smallskip
	\noi
	We next show that (g) $\Rightarrow$ (h) if $f \in \mathrm{L}^2(\uX|\uY)$. In this case, we have 
		\begin{align*}
			|\E_{\uY}((T_tf) \overline{f})| \leq (\E_{\uY}|T_tf|^2)^{\frac{1}{2}} (\E_{\uY}|f|^2)^{\frac{1}{2}} \textrm{ for every } t \textrm{ in } G,
		\end{align*} 
	and therefore $|\E_{\uY}((T_tf) \overline{f})| \leq \|\E_{\uY}|f|^2\|_{\mathrm{L}^\infty(\uY)}\cdot \car$ for every $t \in G$. Consequently,
		\begin{align*}
			\frac{1}{\abs{N_\alpha}} \sum_{t\in
    N_\alpha}  |\E_\uY\bigl((T_tf)
  \konj{f}\bigr)|^2  \leq  \|\E_{\uY}|f|^2\|_{\mathrm{L}^\infty(\uY)} \cdot \left(\frac{1}{\abs{N_\alpha}} \sum_{t\in
    N_\alpha}  |\E_\uY\bigl((T_tf)
  \konj{f}\bigr)|\right)
		\end{align*}
	in $\mathrm{L}^1(\uY)$ for every $\alpha$, and this proves the claimed implication.

	\smallskip
	\noi
	Again by the preliminary remark we obtain that (h) implies (g) for $f \in \mathrm{L}^2(\uX|\uY)$. To finish the proof, we show that (g) implies (a) for abritrary $f \in \mathrm{L}^2(\uX)$.  For $m \in \N$ we set $f_m \coloneqq \car_{\E_{\uY}|f|^2 \leq m^2}f \in \mathrm{L}^2(\uX|\uY)$. Then
		\begin{align*}
			\frac{1}{\abs{N_\alpha}} \sum_{t\in
    N_\alpha}  |\E_\uY\bigl((T_tf_m)
  \konj{f_m}\bigr)| \leq \frac{1}{\abs{N_\alpha}} \sum_{t\in
    N_\alpha}  |\E_\uY\bigl((T_tf)
  \konj{f}\bigr)|
		\end{align*}
	in $\mathrm{L}^1(\uY)$ for every $\alpha$, and hence 
		\begin{align*}
			\lim_\alpha \frac{1}{\abs{N_\alpha}} \sum_{t\in
    N_\alpha}  |\E_\uY\bigl((T_tf_m)
  \konj{f_m}\bigr)|^2 = 0\, \textrm{ in } \mathrm{L}^1(\uY)
		\end{align*}
	for every $m \in \N$. But this means
		\begin{align*}
			(P(f_m \otimes \konj{f_m})|f_m \otimes \konj{f_m}) = \lim_\alpha \biggl\|\frac{1}{\abs{N_\alpha}} \sum_{t\in N_\alpha}  |\E_{\uY}\bigl((T_tf_m)
  \konj{f_m})|^2\biggr\|_{\mathrm{L}^1(\uY)} = 0
		\end{align*}
	and thus $P(f_m \otimes f_m) = 0$ for each $m \in \N$. We obtain from \cref{Kronecker-orth} that $f_m \perp_{\mathrm{L}^2} \scrE(\uX|\uY)$ for all $m \in \N$, and then also $f  \perp_{\mathrm{L}^2} \scrE(\uX|\uY)$, hence (a).
\end{proof}

\medskip

\subsection{Ergodic and Weakly Mixing Extensions}

A measure-preserving system $(\uX;T)$ is {\emdf ergodic} if the fixed space
$\fix(T)$ is one-dimensional. More generally, 
an extension $J \colon (\uY;S) \to (\uX;T)$ of measure-preserving systems is
{\emdf ergodic} if $J(\fix(S)) = \fix(T)$
(cf.\ \cite[Definition 6.1]{Furstenberg1981}).

\begin{examples}
\begin{enumerate}[(1)]
\item 
\label{ergodicsys}
  Let $(\uX;T)$ be a measure-preserving system. Then the trivial extension $J \colon (\{\mathrm{pt}\};\mathrm{Id}) \to (\uX;T)$ 
is ergodic if and only if $(\uX;T)$ is ergodic.

\item \label{ergodicsysimpliesext}
Suppose that $(\uX;T)$ is an ergodic system. Then each extension
$J \colon (\uY;S) \to (\uX;T)$  is ergodic. 

\item   Consider the skew torus $(\uX;T)$ and the extension $J \colon (\uY;S) \to (\uX;T)$ from \cref{skewtorus}. Then $(\uX;T)$ is 
ergodic if and only if $a$ is not a root of unity (see \cite[Prop.~10.17]{EFHN2015}).   However, the extension $J$ is always 
ergodic, as shown (in particular) in \cite[Example 17.4.4]{EFHN2015}. 
\end{enumerate}
\end{examples}

\begin{definition}
An extension $J \colon (\uY;S) \to (\uX;T)$ is {\emdf weakly mixing}
if the 
extension $(\uY;S) \to (\uX \times_{\uY} \uX; T \times_{\uY} T)$ is ergodic.
\end{definition}

\begin{proposition}\label{charweak}
    For an extension $J \colon (\uY;S) \to (\uX;T)$ of measure-preserving systems the following assertions are equivalent.
      \begin{enumerate}[(a)]
        \item The extension $J$ is weakly mixing.

\item $\scrE(\uX|\uY) = \Ell{2}(\uY)$.
        
\item $P_{\fix(T \times_\uY T)}(f \tensor \konj{f}) = 0$   
for all $f \in \uL^2(\uX)$ with $\E_{\uY}f = 0$. 
\end{enumerate}
(Here, $P_{\fix(T \times_\uY T)}$ is the Markov projection onto the
fixed space in $\Ell{1}$.)

\smallskip
\noi
If $G$ is  amenable and $(N_\alpha)_\alpha$ is any F\o{}lner net in
$G$, {\rm (a)--(c)}  are equivalent to the following assertions.
\begin{enumerate}
\item[(d)] For all $f\in \Ell{2}(\uX)$ and $h \in \uL^\infty(\uX)$ one has
\begin{align*}
\lim_{\alpha}  \frac{1}{|N_\alpha|} \sum_{t \in N_\alpha} |\E_{\uY}(T_tf \cdot h) - (S_t\E_{\uY}f) \cdot (\E_{\uY}h)|^2 = 0
          \qquad \text{in $\uL^1(\uY)$}.
         \end{align*}
\item[(e)] For all $f \in \uL^2(\uX)$ with $\E_{\uY}f = 0$ and 
$h \in \Ell{\infty}(\uX)$ one has
\begin{align*}
\lim_\alpha \frac{1}{|N_\alpha|} \sum_{t \in N_\alpha} |\E_{\uY}(T_tf \cdot
  h)|^2 = 0 \qquad \text{in  $\uL^1(\uY)$}.
\end{align*}
  \item[(f)] For all $f\in \Ell{2}(\uX)$ one has
		\begin{align*}
			\lim_{\alpha}  \frac{1}{|N_\alpha|} \sum_{t \in N_\alpha} |\E_{\uY}(T_tf \cdot \overline{f}) - (S_t\E_{\uY}f) \cdot (\E_{\uY}\overline{f})| = 0
          \qquad \text{in $\uL^1(\uY)$}.
		\end{align*}
   \item[(g)] For all $f \in \uL^2(\uX)$ with $\E_{\uY}f = 0$ one has
\begin{align*}
\lim_\alpha \frac{1}{|N_\alpha|} \sum_{t \in N_\alpha} |\E_{\uY}(T_tf \cdot
  f)| = 0 \qquad \text{in  $\uL^1(\uY)$}.
\end{align*}
\end{enumerate}
\end{proposition}   

\begin{proof}
(a)$\dann$(c):\ Suppose (a) and let $f\in \Ell{2}(\uX)$ with $\E_\uY f
= 0$. Then, for  $h\in \fix(T\times_\uY T) \cap \Ell{\infty}$ one has
$h\in \Ell{\infty}(\uY)$ and thus
\[ (f\tensor \konj{f}|\konj{h})_{\Ell{2}} 
= \int_{\uX\times_\uY \uX}  h \, \E_\uY (f\tensor \konj{f}) =  \int_{\uX\times_\uY
  \uX} 
h \,\abs{\E_\uY f}^2 = 0.
\]
This yields (c).

\smallskip
\noi
(c)$\dann$(b):\  This follows directly from the implication (c)$\dann$(a) from
Proposition \ref{Kronecker-orth}.

\smallskip
\noi
(b)$\dann$(a):\  Suppose (b) and note that we need to show
$\fix(T\times_\uY T) \subseteq \Ell{2}(\uY)$. 
Recall the characterization of $\fix(T\times_\uY T)$ in
Theorem \ref{mainthm}. Let 
$\calB = \{ e_1, \dots, e_n\}$ be a suborthonormal
system in $\Ell{2}(\uX|\uY)$ such that $u_\calB =  \sum_{j=1}^n e_j
\tensor \konj{e_j} \in \fix(T \times_{\uY} T)$.  Then $\calB \subseteq
\scrE(\uX|\uY) \subseteq \Ell{2}(\uY)$ and hence $\calB \subseteq \Ell{\infty}(\uY)$. 
It follows that $u_\calB \in
\Ell{\infty}(\uY)$. As these elements generate $\fix(T \times_{\uY} T)$ as a $\fix(S)$-module, we
are done. 

\smallskip
\noi
Finally, assume that $G$ is amenable and let $(N_\alpha)_\alpha$ be any F\o{}lner net in
$G$. Observe that \enquote{(d) $\Leftrightarrow$ (e)} and \enquote{(f) $\Leftrightarrow$ (g)} (replace $f$ with $f - \E_{\uY}f$). 
By Proposition \ref{Kronecker-orth-amen}, (e) and (g) are both equivalent to (b), and hence
the proof is complete.
\end{proof}

We restate the result in the important case of $G = \Z$.

  \begin{corollary}
Let $J \colon (\uY;S) \to (\uX;T)$ an extension of 
measure-preserving $\Z$-systems and set $\mathrm{S} \coloneqq S_1$ 
and $\mathrm{T} \coloneqq T_1$. Then the following assertions are equivalent.
      \begin{enumerate}[(a)]
        \item The extension $J$ is weakly mixing.
      
\item For all $f\in \uL^2(\uX)$ and $h\in \Ell{\infty}(\uX)$ one has
        \begin{align*}
          \lim_{N \to \infty} \frac{1}{N} \sum_{n=0}^{N-1} |\E_{\uY}(\uT^nf \cdot h) - (\uS^n\E_{\uY}f) \cdot (\E_{\uY}h)|^2 = 0
          \quad \text{ in  $\uL^1(\uY)$}.
         \end{align*}

\item For all $f \in \uL^2(\uX)$ with $\E_{\uY}f = 0$ and 
$h\in \Ell{\infty}(\uX)$ one has 
\begin{align*}
            \lim_{N \to \infty} \frac{1}{N}
  \sum_{n=0}^{N-1} |\E_{\uY}(\uT^nf \cdot h)|^2 = 0 \qquad \text{ in  $\uL^1(\uY)$}.
          \end{align*}
\item For all $f\in \uL^2(\uX)$ one has
        \begin{align*}
          \lim_{N \to \infty} \frac{1}{N} \sum_{n=0}^{N-1} |\E_{\uY}(\uT^nf \cdot \overline{f}) - (\uS^n\E_{\uY}f) \cdot (\E_{\uY}\overline{f})|^2 = 0
          \quad \text{ in  $\uL^1(\uY)$}.
         \end{align*}
\item For all $f \in \uL^2(\uX)$ with $\E_{\uY}f = 0$ one has 
\begin{align*}
            \lim_{N \to \infty} \frac{1}{N}
  \sum_{n=0}^{N-1} |\E_{\uY}(\uT^nf \cdot \overline{f})|^2 = 0 \qquad \text{ in  $\uL^1(\uY)$}.
          \end{align*}
      \end{enumerate}
  \end{corollary}

\medskip

\subsection{Furstenberg--Zimmer Structure Theorem}

We are now ready to state and prove the structure theorem of
Furstenberg and Zimmer for measure-preserving systems, comprised of
two parts. The first is the following dichotomy result.

\begin{theorem}[Furstenberg--Zimmer (dichotomy)]\label{fuzimmer1}
    Let $J \colon (\uY;S) \to (\uX;T)$ be an extension of measure-preserving systems. Then exactly one of the following statements is true.
    \begin{enumerate}[(a)]
      \item The extension $J$ is weakly mixing.
      \item There is a non-trivial extension $J_1 \colon 
(\uY;S) \to (\uZ;R)$ with relative discrete spectrum and an extension $J_2 \colon (\uZ;R) \to (\uX;T)$ such that the diagram
    \begin{align*}
        \xymatrix{
          (\uY;S) \ar[rd]_{J_1} \ar[rr]^{J} &  & (\uX;T)\\
            & (\uZ;R)  \ar[ru]_{J_2} & \\   
          }
    \end{align*}
    commutes.     
    \end{enumerate}     
  \end{theorem}

\begin{proof}
As we have  seen in Section \ref{s.Kronecker}, one can always factor
$J$ through the relative Kronecker factor
\[    (\uY; S) \to (\Kro(\uX|\uY); T) \to (\uX; T).
\]
The first extension has, by construction, relative discrete spectrum.  The result is therefore a direct consequence of \cref{charweak}.
\end{proof}

In order to  formulate the second part of the structure theorem
we recall the following concept from \cite[Definition 8.3]{Furs1977}.

  \begin{definition}
    An extension $J \colon (\uY;S) \to (\uX;T)$ is {\emdf distal} if there is an ordinal $\eta_0$ and an inductive system 
$(((X_\eta;T_\eta))_{\eta \leq \eta_0}, (J_\eta^\sigma)_{\eta \leq \sigma})$ such that
    \begin{itemize}
      \item $J_1^\eta = J$,
      \item $J_\eta^{\eta+1}$ has relatively discrete spectrum for every $\eta < \eta_0$,
      \item $(X_\eta;T_\eta) = \lim_{\mu < \eta} (X_\mu;T_\mu)$ for every limit ordinal $\mu \leq \eta_0$.
    \end{itemize}
  \end{definition}

An inductive system as above  is called a  {\emdf Furstenberg tower}.

    \begin{theorem}\label{fuzimmer2}
    Let $J \colon (\uZ;R) \to (\uX;T)$ be an extension of measure-preserving systems. Then there is a distal extension $J_1 \colon 
(\uZ;R) \to (\uY;S)$ and a weakly mixing extension $J_2 \colon (\uY;S) \to (\uX;T)$ such that the diagram
    \begin{align*}
        \xymatrix{
          (\uZ;R) \ar[rd]_{J_1} \ar[rr]^{J} &  & (\uX;T)\\
            & (\uY;S)  \ar[ru]_{J_2} & \\   
          }
    \end{align*}
commutes.     
\end{theorem}

\begin{proof}
As described in Chapter \ref{s.extensions}, the extensions into
$\Ell{2}(\uX)$ can (and now will) be identified with the invariant unital Banach
sublattices of $\Ell{2}(\uX)$. In particular,  we regard 
$E_1\coloneq \Ell{2}(\uZ)$ as a sublattice of $\Ell{2}(\uX)$. The intermediate
extensions
then correspond to invariant sublattices $E$ containing
$E_1$, and $E$  will be called a {\em distal sublattice}
if the corresponding extension $E_1 \to E$ is  a distal extension.

Starting with $E_1$ we shall construct by transfinite induction a
Furstenberg tower. 
Suppose  $\eta$ is an ordinal and we have already constructed
a Furstenberg tower $(E_\sigma)_{\sigma < \eta}$. 

If $\eta = \eta' + 1$ is not a limit ordinal, then 
to the sublattice $E_{\eta'}$ we find  
an associated extension $J_2 \colon (\uY;S) \to (\uX;T)$ 
and  let $\mathrm{E}_{\eta} \coloneq \scrE(\uX|\uY)$ be the
Kronecker space of the extension. 

If $\eta = \lim_{\sigma < \eta} \sigma$ is a limit ordinal, we let
$E_\eta$ be the closure of the union of all $E_\sigma$ for $\sigma <
\eta$. Then any system $(\uY;S)$ associated with $E_\eta$ is the
inductive limit of systems associated with  $E_\sigma$ for $\sigma <
\eta$.

In either case $(E_\sigma)_{\sigma\le \eta}$ is again a Furstenberg
tower. Now, for reasons of  cardinality, there must be an  ordinal
$\eta$ with  $E_\eta = E_{\eta + 1}$. Let $J_2 \colon (\uY;S) \to (\uX; T)$
be an associated extension. Then, by the dichotomy theorem, $J_2$ is
weakly mixing and the proof is complete.
\end{proof}

\section*{Notes and Comments}\label{s.notesIII}

The Furstenberg--Zimmer structure theorem was first proved, under
separability and 
ergodicity assumptions, by Zimmer in \cite{Zimm1977} and,
independently, by 
Furstenberg in \cite{Furs1977}. In \cite{FuKa1978}, Furstenberg and
Katznelson generalized  Furstenberg's result from $\Z$-actions to
$\Z^d$-actions. An alternative presentation was given in
\cite{FKO1982} and in Furstenberg's book \cite{Furstenberg1981}, where
the ergodicity assumption is dropped. Versions of the result can also
be found in modern textbooks 
on ergodic theory, e.g., in \cite[Chapter 9]{Glas}, \cite[Section
7.8]{EinsiedlerWard2011} and \cite[Section 3.3]{KerrLi2016}. Of
course, our exposition is strongly influenced by these 
works. However, we emphasize once more that here the result is
derived as a mere corollary of the more general structure theorem on
unitary group representations 
on Kaplansky--Hilbert modules proved in Part II of this article. This shows that the
Furstenberg--Zimmer 
structure theorem can be viewed, in essence, as a result of \enquote{relative operator theory}.

\smallskip
As is well-known, there are several alternative notions of
``structure'' leading to the Furstenberg--Zimmer theorem. Here we 
 follow Furstenberg's original approach using  ``(relative) discrete
spectrum'', but 
likewise one can use ``isometric'' extensions  or ``compact''
extensions. It is well-known 
that for  a measure-preserving \emph{system} $(\uX;T)$ the following
assertions
are equivalent:
  \begin{enumerate}[(a)]
    \item The space $\mathrm{L}^2(\uX)$ is the closed union of all finite dimensional invariant subspaces.
    \item The orbit $\{T_t f\, \mid\, t \in G\}$ is totally bounded in $\mathrm{L}^2(\uX)$ for every $f \in \mathrm{L}^2(\uX)$.
  \end{enumerate}
(This equivalence, by the way, tells that the decomposition into the 
``discrete spectrum'' part and the ``weakly mixing'' part coincides
with the Jacobs--deLeeuw--Glisckberg decomposition, cf.\ the Notes to
Part II on page \pageref{s.notesII}.)

Most approaches to the FZ-theorem take either (a) or (b) and transfer
it to the relative setting.  For example, Furstenberg and Zimmer
introduce 
structured extensions based on (a) in their original articles, 
while in \cite{FKO1982} the description (b) is employed. It is of course a
natural question how these different approaches are related.
In his book, Furstenberg shows the equivalence of several different
approaches  to structured extensions (see \cite[Theorem
6.13]{Furstenberg1981}). 
His results have been extended recently by Jamneshan in 
\cite{Jamn2020pre} showing, in particular,   that 
the \enquote{algebraic approach} of (a) and the \enquote{topological
  approach}  of (b) are still equivalent for extensions (if generalized suitably).   This is even true in the context of dynamical systems of von Neumann algebras as shown by Jamneshan and Spaas in \cite{JaSp2022}.  

Although we chose (a) for our approach, also (b) can be relativized through our KH-module setting. Namely, 
we call a subset $M$ of a KH-module $E$ over a Stone algebra
$\A$ \textbf{totally order-bounded} if the net 
\[ F \mapsto    \inf_{y\in F} \abs{x- y} \qquad (F \subseteq E \text{
  finite})
\]
decreases to $0$ {\em uniformly in $x\in M$}.
Employing \cref{Kronecker-orth} one can then show the following 
characterization of extensions with relative discrete spectrum.
\begin{theorem*}
  For an extension $J \colon (\uY;S) \rightarrow (\uX;T)$ the following assertions are equivalent.
    \begin{enumerate}[(a)]
      \item $J$ has relative discrete spectrum.
      \item $\Ell{2}(\uX|\uY) = \ocl\{ f\in
                          \Ell{2}(\uX|\uY) \,|\,  
                          \{T_t f \mid t \in G\} \text{ is totally order bounded}\}$.
    \end{enumerate}
\end{theorem*}
A proof of this result as well as a detailed examination of notions of
\enquote{order compactness}, their relations to conditional Boolean
valued analysis (``cyclical compactness'') and conditional set theory,
and their applications to ergodic theory, 
will be the content of future work.

\smallskip
Finally, let us mention that Furstenberg's ``main result on fibered
products'' \cite[Thm.~7.1]{Furs1977}, \cite[Thm.~9.21]{Glas}, i.e., the identity
\[  \scrE(\uX \times_\uY \uZ|\uY) = \scrE(\uX|\uY) \tensor_\uY
\scrE(\uZ|\uY),
\]
can be reduced to the identity $(E\tensor F)_{\ds} = E_{\ds} \tensor
F_\ds$ being valid for KH-dynamical systems.  This will be explained and proved  in the
forthcoming paper \cite{HaKr2023}.

\parindent 0pt
\parskip 0.5\baselineskip
\setlength{\footskip}{4ex}
\bibliographystyle{alpha}

\footnotesize
\end{document}